\documentclass[reqno,11pt]{amsart}
\usepackage[colorlinks=true, linkcolor=blue, citecolor=blue]{hyperref}

\usepackage{amssymb}
\usepackage{amsmath, graphicx, rotating}
\usepackage{color}
\usepackage{soul}
\usepackage[dvipsnames]{xcolor}

\usepackage{ifthen}
\usepackage{xkeyval}
\usepackage{todonotes}
\usepackage[T1]{fontenc}
\usepackage{lmodern}
\usepackage[english]{babel}

\usepackage{ upgreek }
\usepackage{stmaryrd}
\SetSymbolFont{stmry}{bold}{U}{stmry}{m}{n}
\usepackage{amsthm}
\usepackage{float}

\usepackage{ bbm }
\usepackage{ stmaryrd }
\usepackage{ mathrsfs }
\usepackage{ frcursive }
\usepackage{ comment }

\usepackage{pgf, tikz}
\usetikzlibrary{shapes}
\usepackage{varioref}
\usepackage{enumitem}

\definecolor{rouge}{rgb}{0.7,0.00,0.00}
\definecolor{vert}{rgb}{0.00,0.5,0.00}
\definecolor{bleu}{rgb}{0.00,0.00,0.8}
\usepackage[margin=1in]{geometry}
\newtheorem{theorem}{Theorem}[section]
\newtheorem*{theorem*}{Theorem}
\newtheorem{lemma}[theorem]{Lemma}

\newtheorem{proposition}[theorem]{Proposition}

\labelformat{hypothesis}{\textbf{M\kern-0.1mm#1}}

\newtheorem{condition}{Condition}

\newtheorem{conditionA}{A\kern-0.1mm}
\labelformat{conditionA}{\textbf{A\kern-0.1mm#1}}

\newtheorem{conditionB}{B\kern-0.1mm}
\labelformat{conditionB}{\textbf{B\kern-0.1mm#1}}

\theoremstyle{definition}
\newtheorem{example}[theorem]{Example}
\newtheorem{remark}[theorem]{Remark}

\def \eref#1{\hbox{(\ref{#1})}}

\numberwithin{equation}{section}

\def\geq{\geqslant}
\def\leq{\leqslant}

\def\RR{\mathbb{R}}

\def\EE{\mathbb{E}}

\def \eref#1{\hbox{(\ref{#1})}}

\def\EE{\mathbb{ E}}

\begin{document}

\title[Averaging principles for time-inhomogeneous multi-scale SDEs]
{Averaging principles for time-inhomogeneous multi-scale SDEs via nonautonomous
Poisson equations}

\author{Xiaobin Sun\quad }
\curraddr[Sun, X.]{School of Mathematics and Statistics/RIMS, Jiangsu Normal University, Xuzhou, 221116, P.R. China}
\email{xbsun@jsnu.edu.cn}
	
\author{\quad Jian Wang\quad }
\curraddr[Wang, J.]{School of Mathematics and Statistics \& Key Laboratory of Analytical Mathematics and Applications (Ministry
of Education) \& Fujian Provincial Key Laboratory of Statistics and Artificial Intelligence, Fujian Normal University, 350117, Fuzhou, P.R. China}
\email{Jianwang@fjnu.edu.cn}
	
\author{\quad Yingchao Xie}
\curraddr[Xie, Y.]{School of Mathematics and Statistics/RIMS, Jiangsu Normal University, Xuzhou, 221116, P.R. China}
\email{ycxie@jsnu.edu.cn}

\begin{abstract}
The purpose of this paper is to establish asymptotic behaviors of time-inhomogeneous multi-scale stochastic differential equations (SDEs). To achieve them, we analyze the evolution system of measures for time-inhomogeneous Markov semigroups, and investigate regular properties of nonautonomous Poisson equations. The strong and the weak averaging principle for time-inhomogeneous multi-scale SDEs, as well as explicit convergence rates, are provided. Specifically, we show the slow component in the multi-scale stochastic system converges strongly or weakly to the solution of an averaged equation, whose coefficients retain the dependence of the scaling parameter. When the coefficients of the fast component exhibit additional asymptotic or time-periodic behaviors, we prove the slow component converges strongly or weakly to the solution of an averaged equation, whose coefficients are independent of the scaling parameter. Finally, two examples are given to indicate the effectiveness of all the averaged equations mentioned above.
\end{abstract}

\date{\today}

\maketitle

\section{Introduction}

\subsection{Background and motivations}\label{section1.1}

Multi-scale phenomena is very nature in numerous systems, and multi-scale systems are widely used in various fields including nonlinear oscillations, chemical kinetics, biology and climate dynamics, e.g.\ see \cite{EE2003}. For example, the macroscopic scale of fluid meters or millimeters, which are accurately described by density, velocity and temperature fields, follow a continuous Navier-Stokes equation; while the microscopic scale of nanometers obey molecular dynamics in Newton's law, which give the actual position and velocity of each atom that makes up the fluid. Therefore, different physical laws are needed to describe the system with different scales.

The averaging principle is used to describe asymptotic behaviors of multi-scale models, and has been extensively studied for a long time. Since Bogoliubov and Mitropolsky \cite{BM1961} established the averaging principle for ordinary differential equations, Khasminskii \cite{K1968} extended it to stochastic differential equations (SDEs), which opens up a new chapter for the study of the averaging principle for multi-scale SDEs. Here the multi-scale SDEs usually can be written as follows:
$$\left\{\begin{array}{l}
\displaystyle
dX^{\varepsilon}_t=b(X^{\varepsilon}_t, Y^{\varepsilon}_t)\,dt+\sigma(X^{\varepsilon}_t, Y^{\varepsilon}_t)\,dW^1_t, \\
dY^{\varepsilon}_t={\varepsilon}^{-1}f(X^{\varepsilon}_t, Y^{\varepsilon}_t)\,dt+\varepsilon^{-1/2} g(X^{\varepsilon}_t, Y^{\varepsilon}_t)\,dW^2_t,
\end{array}\right.
$$
where $W^1:=\{W^1_t\}_{t\ge0}$ and $W^2:=\{W^2_t\}_{t\ge0}$ are two independent standard Brownian motions on a complete probability space $(\Omega, \mathscr{F}, \mathbb{P})$, $\varepsilon>0$ is a small parameter describing the ratio of the time scale between the slow component $X^{\varepsilon}:=\{X_t^\varepsilon\}_{t\ge0}$ and the fast component $Y^{\varepsilon}:=\{Y_t^\varepsilon\}_{t\ge0}$. Note that the coefficients are independent of time in the above system, which is called the time-homogeneous/autonomous multi-scale SDEs. There are numerous literatures devoted to this subject, e.g. see  \cite{BYY2017,B2012,CF2009,CL2023,DSXZ2018,ELV2005,FWL2015,G2018,GKK2006,HL2020,K1968,L2010,LRSX2020,PXW2017,V1991,WR2012,XML2015}.

\vspace{0.1cm}
It is widely recognized that the research on time-homogeneous SDEs has yielded significant theories and valuable results. Nonetheless, the investigation of time-inhomogeneous SDEs has lagged considerably owing to the lack of necessary research tools and methods. Time-inhomogeneous systems will change over time due to external factors or internal variations. For example, in certain complex multi-scale systems, such as the learning models related to neuronal activity (see \cite{GW2012}), the behaviors of the systems can fluctuate over time. Thus, these systems are influenced by parameters, inputs or disturbances that depend on the time variable, and may exhibit quite different behaviors from these for the autonomous systems. Let us take the following simple time-inhomogeneous multi-scale stochastic system on $\mathbb{R}$ for example:
$$\left\{\begin{array}{l}
\displaystyle
dX^{\varepsilon}_t=Y^{\varepsilon}_t\, dt+dW^1_t,\quad X^{\varepsilon}_0=x\in\mathbb{R},  \\
\displaystyle dY^{\varepsilon}_t=-\varepsilon^{-1}\alpha(t/\varepsilon)Y^{\varepsilon}_t\,dt+ [\varepsilon^{-1}\alpha(t/\varepsilon) ]^{1/2}\,dW^2_t,\quad Y^{\varepsilon}_0=y\in\mathbb{R},
\end{array}\right.
$$
where $\alpha(t)=c_0(1+t)^{\beta}$ with $c_0>0$ and $\beta\in (-1,\infty)$.
As claimed in Example \ref{Example 1} in Section 6, we have
$$
\sup_{t\in [0, T]}[\EE|X_{t}^{\varepsilon}-\bar{X}_{t}|^2]^{1/2}\asymp \left\{\begin{array}{l}
\displaystyle \varepsilon^{(1+\beta)/2},\quad\quad\quad \beta\in (-1,1), \\
\displaystyle \varepsilon\left(\log{1/\varepsilon}\right)^{1/2},\quad \beta=1,\\
\displaystyle \varepsilon,\quad\quad\quad \quad\quad\quad\beta\in (1,\infty),
\end{array}\right.
$$
where $\bar{X}_t=x+W^1_t$ for all $t\ge0$. This indicates that the external factor $\alpha(t)$ influences the strong convergence rate, potentially exceeding the optimal strong convergence rate of $\varepsilon^{1/2}$ established in the classical time-homogeneous setting. Consequently, investigating asymptotic behaviors of time-inhomogeneous stochastic systems is quite significant.

In general, mathematical models of time-inhomogeneous multi-scale SDEs can be expressed in the following manner:
\begin{equation*}\left\{\begin{array}{l}
\displaystyle
dX^{\varepsilon}_t=b(t,t/\varepsilon,X^{\varepsilon}_t, Y^{\varepsilon}_t)\,dt+\sigma(t,t/\varepsilon,X^{\varepsilon}_t, Y^{\varepsilon}_t)\,dW^1_t, \vspace{1mm}\\
dY^{\varepsilon}_t=\varepsilon^{-1}f(t,t/\varepsilon,X^{\varepsilon}_t, Y^{\varepsilon}_t)\,dt+\varepsilon^{-1/2}g(t,t/\varepsilon,X^{\varepsilon}_t, Y^{\varepsilon}_t)\,dW^2_t.
\end{array}\right.
\end{equation*}
Note that, by reviewing $\{(t,X^{\varepsilon}_t)\}_{t\ge0}$ as a new process, one can focus on the case where the coefficients only depend on the variables $\{(t/\varepsilon, X^{\varepsilon}_t, Y^{\varepsilon}_t)\}_{t\ge0}$.  A detailed discussion of this direction can be found in \cite{LRSX2020}. On the other hand, if we consider $t/\varepsilon$ as a rapidly varying component within the coefficients $b$ and $\sigma$, then we can apply the classical theory of the averaging principle (see \cite{FW2012, V1991}).
Consequently, we can turn to the following simpler time-inhomogeneous system:
\begin{equation}\left\{\begin{array}{l}
\displaystyle
dX^{\varepsilon}_t=b(X^{\varepsilon}_t, Y^{\varepsilon}_t)\,dt+\sigma(X^{\varepsilon}_t, Y^{\varepsilon}_t)\,dW^1_t,  \\
dY^{\varepsilon}_t=\varepsilon^{-1}f(t/\varepsilon,X^{\varepsilon}_t, Y^{\varepsilon}_t)\,dt+\varepsilon^{-1/2} g(t/\varepsilon,X^{\varepsilon}_t, Y^{\varepsilon}_t)\,dW^2_t.\label{EqI}
\end{array}\right.
\end{equation}

To the best of our knowledge, there are very limited existing works on the time-inhomogeneous stochastic system such like \eref{EqI}. If $f$ and $g$ are time periodic, i.e., there exists  $\tau>0$ such that $f(t+\tau, x, y)=f(t, x, y)$ and $g(t+\tau, x, y)=g(t, x, y)$, Wainrib \cite{W2013} and Uda \cite{U2021} studied the strong averaging principle for the stochastic system \eref{EqI}
when $\sigma\equiv 0$. Later, Cerrai and Lunardi \cite{CL2017} explored the averaging principle for time-inhomogeneous slow-fast stochastic reaction diffusion equations, where the coefficients in the fast equation satisfy the almost periodic condition. However, the convergence rates are not considered in these mentioned references. It is well known that finding optimal/explicit convergence rates is very important in multi-scale numerical problem. For instance, Br\'{e}hier \cite{B2020} mentioned that \emph{the analysis of the full error of the scheme requires as a preliminary step to estimate the error in the averaging principle}.

 \subsection{Main techniques}\label{section1.2}

Undoubtedly, a crucial prerequisite for understanding the averaging principle lies in comprehending the corresponding averaging equation, and the most important thing is how to accurately describe the averaging coefficients that heavily depends on the long time behavior of the solution to the so-called frozen equation. Note that in the present setting the corresponding frozen equation corresponding to the stochastic system \eref{EqI} is the following time-inhomogeneous SDEs:
\begin{equation}
dY^{x,y}_{t}=f(t, x, Y^{x,y}_{t})\,dt+g(t,x,Y^{x,y}_{t})\,dW^2_t,\quad x\in \RR^{n},y\in\RR^{m}.\label{TDFrozenE}
\end{equation}
According to ergodic theory of Markov processes, the unique invariant measure plays a crucial role in analyzing the long-time behavior of time-homogeneous SDEs. However, an invariant measure may not exist if concerning the SDE \eref{TDFrozenE}. A natural generalization of the concept of invariant measure is an evolution system of measures for the time-inhomogeneous SDE \eref{TDFrozenE}, see \cite{DR2006,DR2008}. Recall that,  $\{\mu^x_t\}_{t\in \RR}$ is an evolution system of measures for the time-inhomogeneous semigroup $\{P^x_{s,t}\}_{t\geq s}$, if
$$
\int_{\RR^m}P^x_{s,t} \varphi(y)\,\mu^x_s(dy)=\int_{\RR^m}\varphi(y)\,\mu^x_t(dy),\quad s\leq t,\varphi\in C_b(\RR^m).
$$
From this observation, it makes sense to construct the averaged coefficients by taking the average of the original coefficients $b$ and $\sigma$ with respect to this evolution system of measures $\{\mu^x_t\}_{t\in \RR}$.

In order to obtain optimal convergence rates, we will apply the method that is based on the Poisson equation, which has been confirmed as a powerful technique that has successfully been used to obtain optimal strong or weak convergence rates, the diffusion approximation and the central limit theorem in various stochastic systems; see the pioneer results in \cite{PV2001,PV2003,PV2005}, and  \cite{B2020,CDGOS2022,HLLS2023,PS2008,RSX2021,RX2021,SX2023,SXX2022} and references therein for related topics.

It is noteworthy that all the quoted papers above focus on an autonomous Poisson equation associated with time-homogeneous SDEs. Since the SDEs \eref{TDFrozenE} are time-inhomogeneous, the first novel contribution of this paper is to introduce and study a class of nonautonomous Poisson equations as formulated below:
$$
\partial_s\Phi(s,x,y)+\mathscr{L}^{x}(s)\Phi(s,x,\cdot)(y)=-H(s,x,y),\quad s\in\RR, x\in \RR^{n},y\in\RR^{m},
$$
where
$\mathscr{L}^{x}_s$ is the generator of SDE \eref{TDFrozenE}, that is,
$$
\mathscr{L}^{x}(s)\varphi(y):=\langle f(s,x,y),\nabla\varphi(y)\rangle+\frac{1}{2}\text{Tr} [(gg)^{\ast}(s,x,y)\nabla^2\varphi(y)],\quad \varphi\in C^2(\RR^m).
$$
In order to solve the nonautonomous Poisson equations above, it is imperative that the function $H:\RR\times\RR^n\times\RR^m\rightarrow \RR^n$ adheres to the \emph{centering condition}:
$$
\int_{\RR^m}H(s,x,y)\,\mu^{x}_s(dy)=0,\quad  s\in \RR,x\in\RR^n.
$$
Under additional conditions, the aforementioned nonautonomous Poisson equations admit the solutions of the form:
$$
\Phi(s,x,y):=\int^{\infty}_{s}\EE H(r,x,Y^{s,x,y}_r)\,dr,
$$
where $\{Y^{s,x,y}_t\}_{t\geq s}$ is the strong solution to the SDE \eref{TDFrozenE} with $Y^{s,x,y}_s=y$. To realize our approach, we need regularity properties of $\Phi(s,x,y)$. In particular, we will investigate the first and second-order derivatives of $\Phi(s,x,y)$ with respect to the variables $x$ and $y$ respectively.

 \subsection{Summary of main results}\label{section1.3}

Using the nonautonomous Poisson equations above, we aim to obtain the optimal strong and weak convergence rates. In the following, we roughly state the contribution of our paper.

\vspace{2mm}
\emph{Strong averaging principle}: Suppose additionally that $\sigma(x,y)=\sigma(x)$ for all $x\in \RR^{n}$ and $y\in\RR^{m}$. We demonstrate that
 \begin{equation}
\sup_{t\in [0, T]}\EE |X_{t}^{\varepsilon}-\bar{X}^{\varepsilon}_{t} |^2\leq C_{T,|x|,|y|}\varepsilon^{2}\left[\sup_{0\leq t\leq T}|\Lambda_{\gamma}(t/\varepsilon)|^2  +  \int^{T/\varepsilon}_0     \alpha(s)\Lambda^2(s)\,ds\right], \label{IR1}
\end{equation}
where  ${X}^{\varepsilon}:=\{{X}^{\varepsilon}_t\}_{t\ge0}$ is the slow component in the multi-scale SDE $\{(X^{\varepsilon}_t, Y^{\varepsilon}_t)\}_{t\ge0}$ given by \eqref{EqI}, and $\bar{X}^{\varepsilon}:=\{\bar{X}^{\varepsilon}_t\}_{t\ge0}$ is the solution to the following averaged equation
\begin{equation}
d\bar{X}^{\varepsilon}_{t}=\bar{b}(t/\varepsilon,\bar{X}^{\varepsilon}_t)\,dt+\sigma(\bar{X}^{\varepsilon}_t)\,d W^1_t,\quad\bar{X}^{\varepsilon}_{0}=x \label{I1.3}
\end{equation}
with \begin{equation}\label{e:drift}\bar{b}(s,x)=\int_{\RR^{m}}b(x,y)\,\mu^x_s(dy),\end{equation} and $\Lambda_{\gamma}(s)=\int^{+\infty}_s e^{-\gamma\int^r_s \alpha(v)\,dv}\,dr$ with some $\gamma\in (0,1)$ and $\Lambda(s)=\Lambda_1(s)$. Roughly speaking, the function $\alpha(t)$ above is employed to describe the convergence rate of the time-inhomogeneous semigroup $\{P^x_{s,t}\}_{t\ge s}$ to its evolution system of measures $\{\mu^x_t\}_{t\in \RR}$; that is, for any Lipschitz continuous function $\phi$ on $\RR^m$,
$$
 \Big|P^x_{s,t} \phi(y)-\int_{\RR^m}\phi(z)\,\mu^x_t(dz)\Big |\leq Ce^{-\int^t_s\alpha(r)\,dr}.
$$

\vspace{2mm}
\emph{Weak averaging principle}: It is crucial to note that the assumption of $\sigma(x,y)=\sigma(x)$ is essential for establishing the strong averaging principle, otherwise it may not hold (see a counter-example in \cite[section 4.1]{L2010}). Instead,  in the general case we will analysis the weak averaging principle. As an application of nonautonomous Poisson equations, we establish
\begin{align}
\sup_{0\leq t\leq T}|\EE \varphi(X^{\varepsilon}_t)-\EE \varphi(\bar{X}^{\varepsilon}_t)|\leq C_{\phi, T,|x|,|y|}\varepsilon\sup_{t\in [0,T]}\Lambda(t/\varepsilon), \quad \varphi\in C^4_b(\RR^n),\label{AVWOI}
\end{align}
where $\bar{X}^{\varepsilon}:=\{\bar{X}^{\varepsilon}_t\}_{t\ge0}$ is the solution to the averaged equation
\begin{equation}
d\bar{X}^{\varepsilon}_{t}=\bar{b}(t/\varepsilon,\bar{X}^{\varepsilon}_t)\,dt+\bar\sigma(t/\varepsilon,\bar{X}^{\varepsilon}_t)\,d \bar W_t,\quad\bar{X}^{\varepsilon}_{0}=x \label{I1.6}
\end{equation}
with $\bar{b}(t,x)$ being defined by \eqref{e:drift},
$$\bar{\sigma}(t,x):=\left[\overline{\sigma\sigma^{\ast}}(t,x)\right]^{1/2}:=\left[\int_{\RR^{m}}\left(\sigma\sigma^{\ast}\right)(x,y)\,\mu^x_t(dy)\right]^{1/2},$$ and  $\bar{W}:=\{\bar W_t\}_{t\ge0}$ being a standard $n$-dimensional Brownian motion.

\vspace{2mm}
To the best of our knowledge, the statements \eref{IR1} and \eref{AVWOI} are new.
If $\sup_{t\in [0,T]}|\Lambda_{\gamma}(t/\varepsilon)|^2\leq   C\int^{T/\varepsilon}_0 \alpha(s)\Lambda^2(s)\,ds
$ holds for small enough $\varepsilon>0$, then rates of the strong and weak averaging principle are reduced into $\varepsilon[\int^{1/\varepsilon}_0 \alpha(s)\Lambda^2(s)\,ds]^{1/2}$ and $\varepsilon\sup_{t\in [0,T]}\Lambda(t/\varepsilon)$ respectively, which are optimal (at least) for the strong averaging principle in some specific setting (see Example \ref{Example 1} below). When $\alpha(t)$ is a constant function, it is evident that the strong convergence order is $1/2$ and that the weak convergence order is $1$, which are consistent with the classic result (see e.g. \cite{L2010}).

\vspace{2mm}

It is worth mentioning that the coefficients in \eref{I1.3} and \eref{I1.6} still include the parameter $\varepsilon$. This is natural in the time-inhomogeneous setting, see Example \ref{Example 2} below. Therefore, in order to achieve coefficients of the averaged equation independent of $\varepsilon$, we should additionally assume that the coefficients $f$ and $g$ exhibit convergence and periodicity with respect to time, respectively.

\vspace{2mm}
For $f$ and $g$ being convergence, that is,
\begin{align}
|f(t,x,y)-\bar f(x,y)|+\|g(t,x,y)-\bar g(x,y)\|\leq \phi(t)(1+|x|+|y|),\quad t\ge0, x\in \RR^n, y\in \RR^m\label{C1}
\end{align}
with $\phi(t)\to 0$ as $t\to +\infty$, we need to verify that $|\mu^x_t(\varphi)-\mu^x(\varphi)|\to0$ as $t\to +\infty$ for any Lipschitz continuous function $\varphi$, where $\mu^x$ is the unique invariant measure of some time-homogeneous SDE with drift and diffusion coefficients being $\bar f(x,\cdot)$ and $\bar{g}(x,\cdot)$ respectively. The condition \eref{C1} is very nature for analyzing the long-time behavior of the time-inhomogeneous SDE, for instance, see \cite{BSWX2024, GO2013}. In this context, it guarantees that $$\lim_{t\to +\infty}\bar{b}(t,x)=\int_{\RR^m}b(x,y)\,\mu^x(dy)=:\bar{b}_c(x),$$
and
$$
\lim_{t\to +\infty}\overline{\sigma\sigma^{\ast}}(t,x)=\int_{\RR^{m}}\left(\sigma\sigma^{\ast}\right)(x,y)\,\mu^x(dy)=:\left(\overline{\sigma\sigma^{\ast}}\right)_c(x).
$$
As a result, the new averaged equation that does not depend on $\varepsilon$ can be formulated by using the drift coefficient $\bar{b}_c$ and the diffusion coefficient $(\overline{\sigma\sigma^{\ast}})_c$. See the SDEs \eref{AVE21} and \eref{AVE22} below. Additionally, specific convergence rates relying on $\phi$ are discussed for the strong and the weak convergences.

\vspace{2mm}
For $f$ and $g$ being $\tau$-periodic, the measures $\{\mu^x_t\}_{t \in \mathbb{R}}$ will be $\tau$-periodic too, which indicates that $\{\bar{b}(t,x)\}_{t \in \mathbb{R}}$ and
$\{\overline{\sigma\sigma^{\ast}}(t,x)\}_{t \in \mathbb{R}}$ are also $\tau$-periodic for any $x\in\RR^n$. In this framework, we define the new averaged drift and diffusion coefficients as follows:
 $$\bar{b}_p(x):=\frac{1}{\tau}\int^{\tau}_0 \bar{b}(t,x)\,dt,\quad
\left(\overline{\sigma\sigma^{\ast}}\right)_p(x):=\frac{1}{\tau}\int^{\tau}_0 \overline{\sigma\sigma^{\ast}}(t,x)\,dt.
$$
Similarly, we can formulate the averaged equation with the drift coefficient $\bar{b}_p$ and the diffusion coefficient $(\overline{\sigma\sigma^{\ast}})_p$, see the SDEs \eref{AVE31} and \eref{AVE32} below. Meanwhile, we also give the specific convergence rates in the strong and the weak convergences.

\vspace{2mm}
For simplicity, in our paper we will assume that the coefficients are sufficiently smooth. However, such kind smoothness conditions, which are closely related with regular properties of the evolution system of measures and nonautonomous Poisson equations, can be relaxed under specific frameworks, e.g. see \cite{PV2003} for autonomous Poisson equations. Our focus here is on applying the technique via nonautonomous Poisson equations to establish the averaging principle and explicit convergence rates for time-inhomogeneous multi-scale SDEs. By utilizing the established nonautonomous Poisson equations, we can further explore the central limit theorem and the diffusion approximation for time-inhomogeneous multi-scale SDEs. We believe that there are several interesting issues that merit further investigation.

\ \

The remainder of the paper is structured as follows. In Section 2, we explore regularity properties of solutions to time-inhomogeneous frozen SDEs and nonautonomous Poisson equations. As an application, we study the strong and the weak averaging principle for a class of time-inhomogeneous multi-scale SDEs with explicit convergence rates in Section 3. Sections 4 and 5 focus on the cases involving convergent and time-periodic coefficients, respectively.  We give two concrete one-dimensional examples to illustrate our main results in the last section.

\section{Time-inhomogeneous frozen SDEs and nonautonomous Poisson equations}

This section focuses on a class of time-inhomogeneous frozen SDEs and nonautonomous Poisson equations. In Subsection 2.1 we mainly investigate the differentiability of solutions to a class of time-inhomogeneous frozen SDEs. In Subsection 2.2 we consider the existence and the uniqueness of evolution system of measures associated with this class of time-inhomogeneous frozen SDEs. Subsection 2.3 is devoted to establishing the well-posedness and regularities of nonautonomous Poisson equations corresponding to the time-inhomogeneous SDEs above.

Denote by $|\cdot|$ and  $\langle\cdot, \cdot\rangle$ the Euclidean vector norm and the usual Euclidean inner product, respectively. Let $\|\cdot\|$ be the matrix norm or the operator norm if there is no confusion.
For a function $\varphi(x)$ being vector-valued or matrix-valued and defined on $\RR^n$, or a function $\varphi(x,y)$ defined on $\RR^n\times\RR^m$, the notation $\partial_i\varphi(x)$ denotes the $i$-th order derivative of $\varphi(x)$,
and the notation $\partial^i_x \partial^j_y\phi(x,y)$ denotes the $i$-th and the $j$-th partial derivatives of $\varphi(x,y)$ with respect to $x$ and $y$ respectively. For $l_1,l_2\in\mathbb{N}_{+}$, let $C^{k_1,k_2}(\RR^n\times\RR^m,\RR^{l_1\times l_2})$ be the set of functions $\varphi: \RR^n\times\RR^m\to \RR^{l_1\times l_2}$ such that $\partial^i_x \partial^j_y\varphi(x,y)$ are continuous with respect to  $x$ and $y$ for any $0\le i\le k_1$ and $0\le j\le k_2$, respectively.
Throughout this paper, we use $C$ and $C_{T}$ to represent constants whose values may vary from line to line, and we use $C_{T}$ to emphasize that the constant depends on $T$.

\subsection{Time-inhomogeneous frozen SDEs}
For any fixed $x\in \RR^n$, consider  the following time-inhomogeneous frozen SDE
\begin{equation}
dY_{t}=f(t,x,Y_{t})\,dt+g(t,x,Y_{t})\,d W^2_t,\quad Y_{s}=y\in \RR^{m}, s\in \RR, \label{FrozenE}
\end{equation} where $f(t,\cdot,\cdot)\in C^{2,3}(\RR^n\times\RR^m,\RR^m)$, $g(t,\cdot,\cdot)\in C^{2,3}(\RR^n\times\RR^m,\RR^{m\times d_2})$  and $W^2:=(W^2_t)_{t\in \RR}$ is a $d_2$-dimensional standard Brownian motion.

Throughout this paper, we always suppose that the following assumption holds:
\begin{conditionA}\label{A1} \it
There exists a constant $C>0$ such that for any $t\in\RR$, $x_1,x_2\in \RR^n$ and $y_1,y_2\in \RR^m$,
\begin{equation}\label{A10}\begin{split}
&2\langle y_1-y_2,f(t,x_1,y_1)-f(t,x_2,y_2) \rangle +3\|g (t,x_1,y_1)-g(t,x_2,y_2) \| ^{2}\\
&\le -2\alpha(t)|y_1-y_2|^{2}+C\alpha(t)|x_1-x_2|^2,\end{split}
\end{equation}
and for any $t\in \RR$, $x\in \RR^n$, $y\in \RR^m$, $i=0,1,2$ and $j=0,1,2,3$ with $1\leq i+j\leq 3$,
\begin{equation}\label{A11}\begin{split}
&\|\partial^{i}_x\partial^j_yf(t,x,y)\|+\|\partial^{i}_x\partial^j_yg(t,x,y)\|^2\leq C\alpha(t),\\
&|f(t,x,y)|\leq C\alpha(t)(1+|x|+|y|),\\
&\|g(t,x,y)\|^2\leq C\alpha(t)(1+|x|^2+|y|^2),
\end{split}\end{equation}
where $\alpha: \RR\to (0,\infty)$ satisfies
\begin{equation}
\min\left \{ \int_{-\infty }^{0} \alpha(u)\,du,~ \int_{0}^{+\infty}  \alpha(u)\,du \right \} =\infty, \quad \Lambda_{\gamma}(t):=\int^{+\infty}_{t} e^{-\gamma\int^r_t\alpha(u)\,du}\,dr<\infty,\,\, t\in \RR\label{A12}
\end{equation}
for some $\gamma\in (0,1)$.
\end{conditionA}

\begin{remark} We give the following comments on  Assumption \ref{A1}:
\begin{itemize}
 \item[{\rm(i)}] It is easy to see that  under Assumption \ref{A1} the SDE \eref{FrozenE} admits a unique strong solution, denoted by $Y^{s,x,y}:=(Y^{s,x,y}_t)_{t\ge s}$, which is a time-inhomogeneous Markov process. \eref{A10} together with \eqref{A12} is a dissipative assumption, guaranteeing the existence and the uniqueness of evolution system of measures for the SDE \eqref{FrozenE}, see Subsection 2.2 for the details. In particular, the condition \eqref{A12} yields the long-time stability of the frozen equation \eqref{FrozenE}.

\item[{\rm(ii)}]It follows from  \eref{A10} that for any $x\in\RR^n$, $y,l\in\RR^m$ and $\delta>0$,
$$
2\langle f(x,y+\delta l)-f(x,y), \delta l\rangle+3\|g (t,x,y+\delta l)-g(t,x,y) \| ^{2}\leq -2\alpha(t)\delta^2|l|^2,
$$
which is equivalent to saying that
$$
2\langle \delta^{-1}[f(x,y+\delta l)-f(x,y)],  l\rangle+3\|\delta^{-1}(g (t,x,y+\delta l)-g(t,x,y)) \| ^{2}\leq -2\alpha(t)|l|^2.
$$
Then, if $f(t,x,y)$ and $g(t,x,y)$ are differentiable with respect to $y$, letting $\delta\rightarrow 0$ in the inequality above yields that for any $x\in\RR^n$ and $y,l\in\RR^m$
\begin{equation}\label{ConDP}
2 \langle \partial_yf(t,x,y)\cdot l,l \rangle +3\|\partial_yg(t,x,y)\cdot l\|^{2}\le -2\alpha(t)|l|^{2}.
\end{equation}

\item[{\rm(iii)}] Smooth conditions on the coefficients $f(t,\cdot,\cdot)\in C^{2,3}(\RR^n\times\RR^m,\RR^m)$ and $g(t,\cdot,\cdot)\in C^{2,3}(\RR^n\times\RR^m,\RR^{m\times d_2})$  as well as their uniform estimates of partial derivatives in \eqref{A11} are used to prove twice differentiable of $Y^{s,x,y}_t$ in the mean square sense with respect to $x$ and $y$.

 \item[{\rm(iv)}] The conditions \eqref{A10} and \eqref{A11} together imply that for any $\beta\in (0,1)$ there exists a constant $C_{\beta}>0$ such that for all $t\in \RR$, $x\in \RR^n$ and $y\in \RR^m$,
\begin{equation}
2\langle y,f(t,x,y) \rangle +3\|g(t,x,y)\|^2\leq -2\beta\alpha(t)|y|^2+C_{\beta}\alpha(t)(1+|x|^2).\label{RE2}
\end{equation}
The constants $2$ and $3$ involved  in \eref{RE2} guarantees that  the solution $\{Y^{s,x,y}_t\}_{t\geq s}$ has finite fourth moment, which will be used in our arguments later.
\end{itemize}
\end{remark}
\begin{lemma}
Suppose that Assumption {\rm\ref{A1}} holds. Then there is a constant $C>0$ such that
\begin{itemize}
\item[{\rm (i)}] for $t \ge s$, $x\in\RR^{n}$ and $y\in\RR^{m}$,
\begin{equation}
\EE|Y_{t}^{s,x,y} |^{4} \le e^{-4\gamma\int^{t}_s \alpha(u)\,du}|y|^4+C(1+|x|^4);\label{uniformEY}
\end{equation}
\item[{\rm (ii)}]   for $t \ge s$, $x_1,x_2\in\RR^{n}$ and $y_1,y_2\in\RR^{m}$,
\begin{equation}
\mathbb{E}|Y_t^{s,x_1,y_1}-Y_t^{s,x_2,y_2}|^{4}\le e^{-4\gamma\int^t_s \alpha(u)\,du }|y_1-y_2|^{4}+C|x_1-x_2|^{4}, \label{FR1}
\end{equation}\end{itemize}
where $\gamma\in(0,1)$ is the constant in \eqref{A12}.
\end{lemma}

\begin{proof}
(i) Recall that
$$
Y_{t}^{s,x,y} =y+\int_{s}^{t} f(r,x, Y_{r }^{s,x,y})\,dr + \int_{s}^{t} g(r,x, Y_{r }^{s,x,y}) \,d W^2_{r},\quad t\ge s.
$$
According to the It\^{o} formula,  for all $t\ge s$,
\begin{align*}
\mathbb{E}|Y_{t}^{s,x,y}|^{4}=&|y|^{4}+4\mathbb{E}\int_{s}^{t} |Y_{r}^{s,x,y} |^{2} \langle Y_{r}^{s,x,y},f(r,x,Y_{r}^{s,x,y}) \rangle\, dr\\	 &+4\mathbb{E}\int_{s}^{t}|g^{*}(r,x,Y_{r}^{s,x,y})\cdot Y_{r}^{s,x,y}|^{2}\,dr+2\mathbb{E}\int_{s}^{t} |Y_{r}^{s,x,y} |^{2}\|g(r,x,Y_{r}^{s,x,y}) \|^{2}\,dr.
\end{align*}
Hence, by \eref{RE2} and Young's inequality,
\begin{align*}
\frac{d}{dt}\mathbb{E} |Y_{t}^{s,x,y} |^{4}
\leq & 4\mathbb{E} [ |Y_{t}^{s,x,y} |^{2}  \langle Y_{t}^{s,x,y},f(t,x,Y_{t}^{s,x,y})  \rangle ]+6\mathbb{E} [ |Y_{t}^{s,x,y} |^2 \|g(t,x,Y_{t}^{s,x,y})  \|^{2} ] \\
\leq & -4\gamma\alpha(t)\mathbb{E} |Y_{t}^{s,x,y} |^{4}+C\alpha(t)(1+|x|^4).
		\end{align*}
This yields
\begin{align*}
\mathbb{E} |Y_{t}^{s,x,y} |^{4}\leq &e^{-4\gamma\int^t_s \alpha(u)\,du}|y|^4+C(1+|x|^4)\int^t_s e^{-4\gamma\int^t_u\alpha(r)\,dr}\alpha(u)\,du\\
\leq &e^{-4\gamma\int^{t}_s \alpha(u)\,du}|y|^4+C(1+|x|^4).
		\end{align*}

(ii) The proof is similar to (i). Note that, for $t \ge s$, $x_i\in\RR^{n}$ and $y_i\in\RR^{m}$ with $i=1, 2$,
		\begin{align*}
			d (Y^{s,x_1,y_1}_t-Y^{s,x_2,y_2}_t )= & [f(t,x_1, Y^{s,x_1,y_1}_t)-f(t,x_2, Y^{s,x_2,y_2}_t) ]\,dt\\
 &+ [g(t,x_1, Y^{s,x_1,y_1}_t)-g(t,x_2, Y^{s,x_2,y_2}_t) ]\,d W^{2}_t
		\end{align*}
		with $Y^{s,x_1,y_1}_s-Y^{s,x_2,y_2}_s=y_1-y_2$.
Using It\^o's formula and taking the expectation on both sides of the equality above, we get
\begin{align*}
& \mathbb{E} |Y_t^{s,x_1,y_1}-Y_t^{s,x_2,y_2}  |^{4} \\
&=   |y_1-y_2|^{4} + 4\mathbb{E} \int_s^t |Y_r^{s,x_1,y_1} - Y_r^{s,x_2,y_2}  |^2 \langle Y^{s,x_1,y_1}_r-Y^{s,x_2,y_2}_r, f(r,x_1, Y^{s,x_1,y_1}_r) - f(r,x_2,Y^{s,x_2,y_2}_r) \rangle \,dr\\
&\quad +4\mathbb{E}\int_{s}^{t}   |(g(r,x_1, Y^{s,x_1,y_1}_r)-g(r,x_2, Y^{s,x_2,y_2}_r))^{\ast} (Y_r^{s,x_1,y_1}-Y_r^{s,x_2,y_2}) |^2 \,dr\\
&\quad +2\mathbb{E}\int_{0}^{t}  |Y_r^{s,x_1,y_1}-Y_r^{s,x_2,y_2}  |^2  \|g(r,x_1, Y^{s,x_1,y_1}_r)-g(r,x_2, Y^{s,x_2,y_2}_r)  \|^2 \,dr.
\end{align*}
Then, by   \eref{A10} and Young's inequality,
\begin{align*}
&\frac{d}{dt}\mathbb{E} |Y_t^{s,x_1,y_1}-Y_t^{s,x_2,y_2} |^{4}\\
&\leq 4\mathbb{E} [ |Y_t^{s,x_1,y_1}-Y_t^{s,x_2,y_2} |^{2} \langle Y^{s,x_1,y_1}_t-Y^{s,x_2,y_2}_t, f(t,x_1, Y^{s,x_1,y_1}_t)-f(t,x_2,Y^{s,x_2,y_2}_t) \rangle ]\\
& \quad+6\mathbb{E} [ |Y_t^{s,x_1,y_1}-Y_t^{s,x_2,y_2} |^{2}  \|g(t,x_1, Y^{s,x_1,y_1}_t)-g(t,x_2, Y^{s,x_2,y_2}_t) \|^2 ]\\
&\leq\mathbb{E} [ |Y_t^{s,x_1,y_1}-Y_t^{s,x_2,y_2} |^{2} (-4\alpha(t) |Y_t^{s,x_1,y_1}-Y_t^{s,x_2,y_2} |^2+C\alpha(t)|x_1-x_2|^2
 ) ]\\		
&\leq -4\gamma\alpha(t)\mathbb{E} |Y_t^{x_1,y_1}-Y_t^{x_2,y_2} |^{4}+C\alpha(t)|x_1-x_2|^{4}.
\end{align*}
Therefore,
\begin{align*}
\mathbb{E} |Y_t^{s,x_1,y_1}-Y_t^{s,x_2,y_2} |^{4}\le& e^{-4\gamma\int^t_s \alpha(u)\,du }|y_1-y_2|^{4}+C\int^t_s e^{-4\gamma\int^t_r \alpha(u)\,du}\alpha(r)\,dr|x_1-x_2|^{4}\\
\le & e^{-4\gamma\int^t_s \alpha(u)\,du }|y_1-y_2|^{4}+C|x_1-x_2|^{4}.
\end{align*}
The proof is complete.
\end{proof}

Let $\Psi: \RR^n\times\RR ^m \to \RR^m$ be a random variable. Its first partial
derivative with respect to $y$ in the mean square sense is defined as a random variable $\partial _y\Psi(x,y)=(\partial _{y_1}\Psi(x,y),\cdots,\partial _{y_m}\Psi(x,y))$, where for all $1\le i\le m$,
$$
\lim_{\delta\rightarrow 0}\EE\left|\frac{\Psi(x,y+\delta e_i)-\Psi(x,y)}{\delta}-\partial _{y_i}\Psi(x,y)\right|^2=0
$$
with $\{e_i\}_{i=1,\ldots,m}$ being an orthogonal basis of $\RR^m$. Similarly, we can define $\partial _x\Psi(x,y)$ as the first partial derivative of $\Psi(x,y)$ with respect to $x$ in the mean square sense (if it exists). Furthermore, we denote $\partial^2_x\Psi(x,y)$ and $\partial^2_y\Psi(x,y)$ as the second partial derivative of $\Psi(x,y)$ in the mean square sense with respect to $x$ and $y$ (if they exist), respectively.

\begin{lemma}\label{DFY}
Suppose that Assumption {\rm\ref{A1}} holds. Let $Y^{s,x,y}:=\{Y_t^{s,x,y}\}_{t\ge s}$ be the unique strong solution to the  SDE \eqref{FrozenE}. Then, for any $t\ge s$, $Y^{s,x,y}_t$ is twice differentiable in the mean square sense with respect to $x$ and $y$, respectively; moreover, for $t\ge s$,
\begin{equation}\label{EPY}\begin{split}
&\EE \|\partial _yY_{t}^{s,x,y}\|^{4}\leq e^{-4\int^t_s\alpha(u)\,du},
\quad \EE \|\partial^2_yY_{t}^{s,x,y}\|^{2}\leq e^{-2\gamma\int^t_s\alpha(u)\,du},\\
&\sup_{t\geq s}\mathbb{E} \|\partial_x Y^{s,x,y}_t\|^{4}\leq C,
\quad\sup_{t\geq s}\mathbb{E} \|\partial^2_x Y^{s,x,y}_t\|^{2}\leq C,
\end{split}\end{equation}
where $\gamma\in(0,1)$ is the constant in \eqref{A12}.
\end{lemma}

\begin{proof}
We first prove that $Y^{s,x,y}_t$ is first differentiable in the mean square sense with respect to $y$. In fact, it is sufficient to prove that for any $l\in \RR^m$,
\begin{equation}
\lim_{\delta\to 0}\mathbb{E} \left|\frac{Y^{s,x,y+\delta l}_t-Y^{s,x,y}_t}{\delta}-\partial_y Y^{s,x,y}_t\cdot l \right|^2=0,\label{Partialy Y}
\end{equation}
where $\partial_y Y^{s,x,y}_t\cdot l$  satisfies
\begin{equation}\label{Apartial y}
d [\partial _yY_{t}^{s,x,y}\cdot l ]= \partial_yf(t,x,Y_{t}^{s,x,y})(\partial_yY_{t}^{s,x,y}\cdot l)\,dt
 +\partial_yg(t,x,Y_{t}^{s,x,y})(\partial_yY_{t}^{s,x,y}\cdot l)\,d W_{t}^{2}\end{equation}
with $\partial _yY_{s}^{s,x,y}\cdot l=l$.
To do this, denoted by
$$Z^{s,\delta,l}_t:=\frac{Y^{s,x,y+\delta l}_t-Y^{s,x,y}_t}{\delta}-\partial _yY_{t}^{s,x,y}\cdot l.$$
Then $\{Z^{s,\delta,l}_t\}_{t\ge s}$ satisfies the following equation
$$\begin{cases}
\displaystyle
dZ^{s,\delta,l}_t=\delta^{-1} [f(t,x,Y_{t}^{s,x,y+\delta l})-f(t,x,Y_{t}^{s,x,y})-\delta\partial_yf(t,x,Y_{t}^{s,x,y})(\partial_yY_{t}^{s,x,y}\cdot l) ]\,dt\\
\quad\quad\quad\quad+\delta^{-1} [g(t,x,Y_{t}^{x,y+\delta l})-g(t,x,Y_{t}^{s,x,y})-\delta\partial_yg(t,x,Y_{t}^{s,x,y})(\partial_yY_{t}^{s,x,y}\cdot l) ]\,d W_{t}^{2},\\
\,\,Z^{s,\delta,l}_s=0.
\end{cases}
$$
Using It\^{o}'s formula, we have
\begin{align*}
 \frac{d}{dt}\EE|Z^{s,\delta,l}_t|^2
&=2\delta^{-1}\EE \langle f(t,x,Y_{t}^{s,x,y+\delta l}) - f(t,x,Y_{t}^{s,x,y}) - \delta\partial_yf(t,x,Y_{t}^{s,x,y})(\partial_yY_{t}^{s,x,y}\cdot l), Z^{s,\delta,l}_t \rangle  \\
 &\quad +\delta^{-2}\EE [  \|g(t,x,Y_{t}^{s,x,y+\delta l}) - g(t,x,Y_{t}^{s,x,y}) - \delta\partial_yg(t,x,Y_{t}^{s,x,y})(\partial_yY_{t}^{s,x,y}\cdot l) \|^2  ] \\
&=:Q_1(t)+Q_2(t).
\end{align*}

For $Q_1(t)$, by Taylor's formula,  \eref{A11} and \eref{FR1}, one has
\begin{align*}
Q_1(t)
= &2\delta^{-1}\EE\langle  f(t,x,Y_{t}^{s,x,y+\delta l})  -  f(t,x,Y_{t}^{s,x,y})-\partial_yf(t,x,Y_{t}^{s,x,y})\cdot (Y^{s,x,y+\delta l}_t-Y^{s,x,y}_t), Z^{s,\delta,l}_t\rangle\\
 &+2\EE \langle \partial_yf(t,x,Y_{t}^{s,x,y})\cdot Z^{s,\delta,l}_t, Z^{s,\delta,l}_t \rangle\\
\leq &C\alpha(t)\delta^{-1}\EE [|Y_{t}^{s,x,y+\delta l}-Y_{t}^{s,x,y}|^2| Z^{s,\delta,l}_t|]+2\EE \langle \partial_yf(t,x,Y_{t}^{s,x,y})\cdot Z^{s,\delta,l}_t, Z^{s,\delta,l}_t \rangle\\
\leq &(2  -  2\gamma)\alpha(t)\EE| Z^{s,\delta,l}_t|^2+2\EE \langle \partial_yf(t,x,Y_{t}^{s,x,y})\cdot Z^{s,\delta,l}_t, Z^{s,\delta,l}_t\rangle+C\alpha(t) e^{-4\gamma\int^t_s\alpha(u)\,du}\delta^2|l|^4,
\end{align*}
where $\gamma\in(0,1)$ is given in \eref{A12}. The similar argument yields that
\begin{align*}
Q_2(t)
\leq&2\delta^{-2}\EE  \|g(t,x,Y_{t}^{s,x,y+\delta l})  -  g(t,x,Y_{t}^{s,x,y}) - \partial_yg(t,x,Y_{t}^{s,x,y})(Y^{s,x,y+\delta l}_t  -  Y^{s,x,y}_t)  \|^2\\
&+2\EE \|\partial_yg(t,x,Y_{t}^{s,x,y})\cdot Z^{s,\delta,l}_t \|^2\\
\leq&C\alpha(t)\delta^{-2}\EE  |Y_{t}^{s,x,y+\delta l}-Y_{t}^{s,x,y} |^4 + 2\EE \|\partial_yg(t,x,Y_{t}^{s,x,y})\cdot Z^{s,\delta,l}_t  \|^2\\
\leq&C\alpha(t) e^{-4\gamma\int^t_s\alpha(u)\,du}\delta^2|l|^4+2\EE \|\partial_yg(t,x,Y_{t}^{s,x,y})\cdot Z^{s,\delta,l}_t \|^2.
\end{align*}
Combining with the above estimates, we have
\begin{align*}
\frac{d}{dt}\EE |Z^{s,\delta,l}_t |^2\leq &(2-2\gamma){\alpha(t)} \EE |Z^{s,\delta,l}_t |^2+C\alpha(t) e^{-4\gamma\int^t_s\alpha(u)\,du} \delta^2|l|^4\\
&+\EE  [2 \langle \partial_yf(x,Y_{t}^{s,x,y})\cdot Z^{s,\delta,l}_t, Z^{s,\delta,l}_t \rangle+2 \|\partial_yg(x,Y_{t}^{s,x,y})\cdot Z^{s,\delta,l}_t  \|^2 ].
\end{align*}
This along with \eref{ConDP} gives us that
$$
\frac{d}{dt}\EE |Z^{s,\delta,l}_t |^2\leq  -2\gamma\alpha(t)\EE |Z^{s,\delta,l}_t |^2+C\alpha(t) e^{-4\gamma\int^t_s\alpha(u)\,du} \delta^2|l|^4.
$$
Hence,
\begin{align*}
\EE|Z^{s,\delta,l}_t|^2
\leq &C\delta^2|l|^4 \int^t_s \alpha(u)\exp \left(-2\gamma \left(\int^t_u  \alpha(u)\,du - 2\int^u_s \alpha(r)\,dr \right) \right)\,du  \\
\leq  &C\delta^2|l|^4e^{-2\gamma\int^t_s\alpha(u)\,du} \to 0
	\end{align*}
as $\delta \to 0$, which implies that \eref{Partialy Y} holds.

Furthermore, by It\^{o}'s formula,
\begin{align*}
\frac{d}{dt}\mathbb{E} |\partial _yY_{t}^{s,x,y}\cdot l| ^{4}\leq &4\mathbb{E} [  |\partial _yY_{t}^{s,x,y}\cdot l  |^{2}  \langle \partial_yf(t,x,Y_{t}^{s,x,y})(\partial_yY_{t}^{s,x,y}\cdot l), \partial _yY_{t}^{s,x,y}\cdot l \rangle ]\\
&  +6\mathbb{E} [ |\partial _yY_{t}^{s,x,y}\cdot l  |^2 \|\partial_yg(t,x,Y_{t}^{s,x,y})(\partial_yY_{t}^{s,x,y}\cdot l) \|^2].
\end{align*}
This along with \eref{ConDP} gives us
$$
\frac{d}{dt}\mathbb{E}\left|\partial _yY_{t}^{s,x,y}\cdot l\right|^{4}
\leq -4\alpha(t)\mathbb{E}\left|\partial _yY_{t}^{s,x,y}\cdot l\right |^{4}.
$$
Hence,
\begin{align}
\EE\left|\partial _yY_{t}^{s,x,y}\cdot l\right|^{4}\leq e^{-4\int^t_s \alpha(u)\,du}|l|^{4},\label{EPartialy}
\end{align}
which implies the first estimate in \eref{EPY} is satisfied.

According to the same arguments as above, we can see that for any unit vectors $l_1,l_2\in\RR^m$,
\begin{equation*}
\lim_{\delta\to 0}\mathbb{E} \left|\frac{\partial_yY^{s,x,y+\delta l_2}_t\cdot l_1-\partial_yY^{s,x,y}_t\cdot l_1}{\delta}\cdot l_2-\partial^2_y Y^{s,x,y}_t\cdot (l_1,l_2) \right|^2=0,\label{Partialyy Y}
\end{equation*}
where $\partial^2_y Y^{s,x,y}_t\cdot (l_1,l_2)$  satisfies
\begin{align*}
&d[\partial _y^2Y_{t}^{s,x,y}\cdot (l_1,l_2)]\\
&=\partial^2_yf(t,x,Y_{t}^{s,x,y})\left(\partial_yY_{t}^{s,x,y}\cdot l_1, \partial_yY_{t}^{s,x,y}\cdot l_2\right)\,dt
+  \partial_yf(t,x,Y_{t}^{s,x,y})(\partial^2_yY_{t}^{s,x,y}\cdot (l_1, l_2))\,dt\\
&\quad +  \partial^2_yg(t,x,Y_{t}^{s,x,y})\left(\partial_yY_{t}^{s,x,y}\cdot l_1,\partial_yY_{t}^{s,x,y}\cdot l_2\right)\,d W_{t}^{2}
+\partial_yg(t,x,Y_{t}^{s,x,y})(\partial^2_yY_{t}^{s,x,y}\cdot (l_1, l_2))\,d W_{t}^{2}
\end{align*}
with $\partial _y^2Y_{t}^{s,x,y}\cdot (l_1,l_2)=0$. Thus $Y^{s,x,y}_t$ is twice differentiable in the mean square sense with respect to $y$. Furthermore, applying It\^{o}'s formula to $|\partial^2_yY_{t}^{s,x,y}\cdot (l_1,l_2)|^{2}$ and following the proof of \eref{EPartialy}, we can obtain
$$
\EE |\partial^2_yY_{t}^{s,x,y}\cdot (l_1,l_2) |^{2}\leq e^{-2\gamma\int^t_s\alpha(u)\,du},
$$
where  $\gamma$ is the constant in \eqref{A12}. We note that here the Young inequality is also used. Therefore, the second estimate in \eref{EPY} holds.

The twice differentiable property of $Y^{s,x,y}_t$ in the mean square sense with respect to $x$, as well as the third and fourth estimates in \eref{EPY} can be proved similarly, and the details are omitted here. The proof is complete.
\end{proof}

\subsection{Evolution system of measures}

In order to prove the existence and the uniqueness of evolution system of measures for the inhomogeneous SDE \eqref{FrozenE}, we primarily adhere to the corresponding concept in \cite{DR2006}. For this aim, we need the following lemma.

\begin{lemma}\label{L:2.4} Suppose that Assumption {\rm\ref{A1}} holds. Let $Y^{s,x,y}:=\{Y_{t}^{s,x,y}\}_{t\ge s}$ be the unique strong solution to the SDE \eqref{FrozenE}. Then, for any $t\in \RR$ and $x\in \RR^n$, there exists $\eta^x_{t} \in L^{2} (\Omega ,\mathscr{F},\mathbb{P} )$ such that for all $t\ge s$, $x\in \RR^n$ and $y\in \RR^m$,
\begin{equation}\label{F3.2}
\EE|Y_{t}^{s,x,y}-\eta^x_{t}|^{2} \le C (1+|x|^2+|y|^{2} )e^{-2\int_{s}^{t}\alpha(u)\,du},\quad \sup_{t\in\RR}\EE|\eta^x_{t}|^2\leq C(1+|x|^2),
\end{equation}
and $\{\eta^x_{t}\}_{t\in\RR}$ satisfies that
\begin{equation}
\eta^x_{t}=\eta^x_{s}+\int^t_s f(r,x,\eta^x_{r})\,dr+\int^t_s g(r,x,\eta^x_{r})\,d W^2_r,\quad t\ge s.\label{F3.3}
\end{equation}
Moreover, $\eta^x_{t}$ is twice differentiable in mean square with respect to $x$, and the first derivative $\partial _x\eta^x_{t}$ and the second derivative $\partial^2_x\eta^x_{t}$ satisfy that for all $t\ge s$, $x\in \RR^n$ and $y\in \RR^m$
\begin{equation}\label{Epartial_xeta}\begin{split}
& \mathbb{E}\|\partial_x Y^{s,x,y}_t-\partial_x \eta^x_{t}\|^{2} \leq C  (1+|x|^2+|y|^{2}  )e^{-2\gamma\int^t_s\alpha(u)\,du},  \\
& \mathbb{E}\|\partial^2_x Y^{s,x,y}_t-\partial^2_x \eta^x_{t}\|^{2} \leq C  (1+|x|^2+|y|^{2} )e^{-2\gamma\int^t_s\alpha(u)\,du},\\
& \sup_{t\in \RR}\EE\|\partial_x\eta^{x}_{t}\|^2\leq C,\quad \sup_{t\in \RR}\EE\|\partial^2_x\eta^{x}_{t}\|^2\leq C,
\end{split}\end{equation}
where $\gamma\in(0,1)$ is the constant in \eqref{A12}.
\end{lemma}

\begin{proof}
(i) Note that, for any $h>0$, $t\ge s$, $x\in \RR^n$ and $y\in \RR^m$,
\begin{equation}
Y_{t}^{s-h,x,y}= Y_{s}^{s-h,x,y}+\int_{s}^{t} f(u, x, Y_{u }^{s-h,x,y} )\,du + \int_{s}^{t} g(u, x, Y_{u}^{s-h,x,y}) \,d W^2_{u}\label{F3.4}
\end{equation}
and
$$
Y_{t}^{s,x,y}= y+\int_{s}^{t} f(u,x,Y_{u}^{s,x,y})\,du + \int_{s}^{t} g(u,x,Y_{u}^{s,x,y}) \,dW^2_{u}.
$$

Put $Z^{h,x,y}_{s,t}:=Y_{t}^{s-h,x,y}-Y_{t}^{s,x,y}$. One has
\begin{align*}
Z^{h,x,y}_{s,t}=&Y_{s}^{s-h,x,y}-y+ \int_{s}^{t}(f(u,x,Y_{u }^{s-h,x,y})-f(u,x,Y_{u}^{s,x,y}))\,du\\
 &+ \int_{s}^{t}(g(u, x, Y_{u}^{s-h,x,y} )-g(u,x,Y_{u}^{s,x,y})) \,dW^2_{u}.
\end{align*}
By the It\^{o} formula,
\begin{align*}
 |Z^{h,x,y}_{s,t} |^{2}= & |Y_{s}^{s-h,x,y}-y |^{2} +2\int_{s}^{t} \langle Z^{h,x,y}_{s,u}, f(u,x,Y_{u}^{s-h,x,y})-f(u,x,Y_{u }^{s,x,y} ) \rangle \,du\\
 &+\int_{s}^{t} \|g(u,x,Y_{u}^{s-h,x,y} )- g(u,x,Y_{u}^{s,x,y} ) \|^{2}\, du\\
 &+2\int_{s}^{t}  \langle Z^{h,x,y}_{s,u}, g(u,x,Y_{u}^{s-h,x,y} )-g(u,x,Y_{u}^{s,x,y} )  \rangle \,d W^2_{u}.
\end{align*}
Then, according to \eref{A10},
\begin{eqnarray*}
\frac{d\EE |Z^{h,x,y}_{s,t} |^{2}}{dt}\leq -2\alpha(t)\EE |Z^{h,x,y}_{s,t} |^{2}.
\end{eqnarray*}
This along with \eref{uniformEY} yields that for any $h>0$, $t\ge s$, $x\in \RR^n$ and $y\in \RR^m$,
\begin{align*}
\EE |Z^{h,x,y}_{s,t}|^{2}\leq \EE |Y_{s}^{s-h,x,y}-y |^{2}e^{-2\int^t_s \alpha(u)\,du}\leq  C (1+|x|^2+|y|^2 )e^{-2\int^t_s \alpha(u)\,du}.
\end{align*}
By this and \eref{A12}, we know that for any $t\in \RR$, $x\in \RR^n$ and $y\in \RR^m$,
$\{Y^{s,x,y}_t\}_{t\geq s}$ is a Cauchy sequence in $L^2(\Omega ,\RR^m )$ as $s\to -\infty$. Hence, there exists a random variable $\eta^{x,y}_t\in L^2(\Omega ,\RR^m )$ such that for any $t,s\in \RR$ with $t\ge s$, $x\in \RR^n$ and $y\in \RR^m$,
\begin{equation}
\EE |{Y_{t}^{s,x,y}} -\eta^{x,y}_{t} |^2\leq C (1+|x|^2+|y|^2 )e^{-2\int^t_s \alpha(u)\,du} .\label{F3.5}
\end{equation}
Next, we are going to prove that $\eta^{x,y}_{t}$ is independent of $y$. Indeed, \eref{FR1} implies that for any $y_1,y_2\in \mathbb{R} ^{m} $,
$$\lim_{s \to -\infty} \EE|Y_{t}^{s,x,y_1} -Y_{t}^{s,x,y_2}|^2=0.$$
Note that
\begin{align*}
\EE |\eta^{x,y_1}_{t}  -  \eta^{x,y_2}_t|^2  \leq    C\EE  |\eta^{x,y_1}_{t}  -  Y_{t}^{s,x,y_1}  |^2  +  C \EE |\eta^{x,y_2}_{t}-Y_{t}^{s,x,y_2}  |^2  +   C\EE |Y_{t}^{s,x,y_1} -  Y_{t}^{s,x,y_2} |^2.
\end{align*}
Taking $s\rightarrow -\infty$ and using \eqref{F3.5}, we get $\EE |\eta^{x,y_1}_{t}-\eta^{x,y_2}_{t} |^2=0$. So~$\eta^{x,y_1}_{t}=\eta^{x,y_2}_{t}$ in $L^2(\Omega ,\RR^m )$. Below, we denote by $\eta_t^{x,y}$ by $\eta_t^x$.

Furthermore, according to \eqref{uniformEY},  \eqref{F3.5} and \eref{A12},
\begin{equation}
\sup_{t\in\RR}\EE|\eta^{x}_{t}|^2\leq C(1+|x|^2).\label{Eeta}
\end{equation}
On the other hand, using \eref{F3.5} and taking the limit as $h$ goes to $+\infty$ on both sides of \eref{F3.4},  we get that for any $t\geq s$ and $x\in \RR^n$,
$$
\eta^x_{t}=\eta^x_{s}+\int^t_s f(r,x,\eta^x_{r})\,dr+\int^t_s g(r,x,\eta^x_{r})\,d W^2_r.
$$ This finishes the proofs of \eqref{F3.2} and \eqref{F3.3}.

(ii) Note that for any $k\in \RR^n$,
\begin{align*}
 \partial _xY_{t}^{s-h,x,y}\cdot k
&= \partial _xY_{s}^{s-h,x,y}\cdot k  +    \int^t_s [\partial_xf(r,x,Y_{r}^{s-h,x,y})\cdot k  +  \partial_yf(r,x,Y_{r}^{s-h,x,y})\cdot(\partial_xY_{r}^{s-h,x,y}\cdot k)]\,dr\\
 &\quad +\int^t_s  [\partial_xg(r,x,Y_{r}^{s-h,x,y})\cdot k+\partial_yg(r,x,Y_{r}^{s-h,x,y})\cdot (\partial_xY_{r}^{s-h,x,y}\cdot k)  ]\,d W_{r}^{2}
\end{align*}
and
\begin{align*}
\partial _xY_{t}^{s,x,y}\cdot k= &\int^t_s [\partial_xf(r,x,Y_{r}^{s,x,y})\cdot k  +  \partial_yf(r,x,Y_{r}^{s,x,y})\cdot(\partial_xY_{r}^{s,x,y}\cdot k) ]\,dr\\
 &+\int^t_s  [\partial_xg(r,x,Y_{r}^{s,x,y})\cdot k  +  \partial_yg(r,x,Y_{r}^{s,x,y})\cdot (\partial_xY_{r}^{s,x,y}\cdot k)  ]\,d W_{r}^{2}.
\end{align*}
Then, following the same argument as in (i) and applying the Young inequality, we can obtain that
$$
\mathbb{E} |\partial_x Y^{s,x,y}_t\cdot k-\partial_x Y^{s-h,x,y}_t\cdot k |^{2} \leq C (1+|x|^2+|y|^2 )|k|^2e^{-2\gamma\int^t_s\alpha(u)\,du},
$$
where $\gamma\in (0,1)$  is the constant in \eqref{A12}.
In particular, due to \eqref{A12}, $\{\partial_xY^{s,x,y}_t\cdot k\}_{t\geq s}$ is a Cauchy sequence in $L^{2}(\Omega,\RR^m)$ as $s\to -\infty$. Furthermore, we can also repeat the argument in (i) and obtain that there exists a random variable $\xi^{x,k}_t\in L^{2}(\Omega,\RR^m)$ which is independent of $y$ such that for all $t\ge s$, $x, k\in \RR^n$ and $y\in \RR^m$,
\begin{equation}
\EE |\partial_x{Y_{t}^{s,x,y}}\cdot k -\xi^{x,k}_t  |^2\leq C (1+|x|^2+|y|^2 )|k|^2e^{-2\gamma\int^t_s \alpha(u)\,du}.\label{F3.52}
\end{equation}
Note that $\partial_xY^{s,x,y}_t\cdot k$ is linear with respect to the vector $k$, so there exists $\xi^{x}_t\in L^{2}(\Omega,\RR^{m\times n})$ such that $\xi^{x,k}_t=\xi^{x}_t\cdot k$.

By \eref{F3.2}, \eref{F3.52} and the fact that $Y^{s,x,y}_t$ is differentiable in the mean square sense with respect to $x$, we know that for any $t\in \RR$ and $x,k\in \RR^n$,
$$
\lim_{\delta\rightarrow 0}\mathbb{E}\left|\frac{\eta^{x+\delta k}_t-\eta^{x}_t}{\delta}-\xi^{x}_t\cdot k\right|^2=0.
$$
This yields that $\eta^x_{t}$ is differentiable in the mean square sense with respect to $x$, and $\partial _x\eta^x_{t} =\xi^{x}_t $. Thus the first inequality in \eref{Epartial_xeta} holds.
Furthermore, according to \eref{EPY} and \eref{F3.52}, $\sup_{t\in \RR}\EE\|\partial_x\eta^{x}_{t}\|^2\leq C$.

The other assertions in \eref{Epartial_xeta} for the second derivative $\partial^2_x\eta^x_{t}$ can be proved similarly, and the details are omitted here.
\end{proof}

\begin{remark}  From the proof above, we can obtain that,
if for any $0\leq i\leq 4$ and $0\leq i\leq 5$ with $1\leq i+j\leq 5$,
$$
\|\partial^{i}_x\partial^j_yf(t,x,y)\|+\|\partial^{i}_x\partial^j_yg(t,x,y)\|^2\leq C\alpha(t),
$$
then $\eta^x_{t}$ is fourth differentiable in the mean square sense with respect to $x$, and
\begin{align}
\sup_{t\in \RR}\sum^4_{i=1}\EE\|\partial^i_x\eta^{x}_{t}\|^4\leq C.\label{FourthDeta}
\end{align}
\end{remark}

In the following, let $\{P^{x}_{s,t}\}_{t\ge s}$ be the semigroup of the process $\{Y_{t}^{s,x,y}\}_{t\geq s}$, i.e., for any bounded measurable function $\varphi$ on $\RR^{m}$,
$$
P^{x}_{s,t}\varphi(y)=\EE\varphi(Y_{t}^{s,x,y}), \quad y\in\RR^{m}.
$$
Let $\eta^x_t$ be given in Lemma \ref{L:2.4}, and, for any $t\in \RR$, denote by $\mu^x_{t}$ the law of $\eta^x_t$.

\begin{proposition}\label{P:2.6} Suppose that Assumption {\rm\ref{A1}} holds. Then, $\{\mu^x_{t}\}_{t\in\RR}$ is an evolution system of measures for the semigroup $\{P^x_{s,t}\}_{t\geq s}$, i.e., for any $t\ge s$, $x\in \RR^n$ and $\varphi\in C_b(\RR^m)$,
\begin{equation}
\int_{\RR^m}P^x_{s,t} \varphi(y)\,\mu^x_s(dy)=\int_{\RR^m}\varphi(y)\,\mu^x_t(dy).\label{ESM}
\end{equation}  Moreover, there is a constant $C>0$ such that for any $t\ge s$, $x\in \RR^n$, $y\in \RR^m$ and Lipschitz continuous function $\phi$ on $\RR^m$,
\begin{equation}
 \left|P^x_{s,t} \phi(y)-\int_{\RR^m}\phi(z)\,\mu^x_t(dz) \right|\leq C {\rm Lip}(\phi)(1+|x|+|y|)e^{-\int^t_s\alpha(u)\,du},\label{WErodicity}
\end{equation}
where ${\rm Lip}(\phi):=\sup_{x\neq y}|\phi(x)-\phi(y)|/|x-y|$.
Furthermore, if $\{\nu^x_{t}\}_{t\in\RR}$ is another evolution system of measures for $\{P^x_{s,t}\}_{t\geq s}$ and satisfies that for all $x\in \RR^n$,
$$
\sup_{t\in\RR}\int_{\RR^m}\,|z|\,\nu^x_t(dz)<\infty,
$$
then $\nu^x_t=\mu^x_t$ for all  $t\in\RR$ and $x\in\RR^n$.
\end{proposition}

\begin{proof}
(i) According to \eqref{F3.2}, for any $t\in \RR$, $x\in \RR^n$ and $ \varphi \in C_b(\mathbb{R}^{m}) $,
$$
\lim_{s \to -\infty} P^x_{s,t}\varphi (y)=\lim_{s \to -\infty}\EE\varphi (Y^{s,x,y}_t)=\EE\varphi (\eta^x_{t} )=\int_{\mathbb{R}^{m}}\varphi (y)\,\mu^x_{t}(dy).
$$
By this and the fact that $P^x_{s,t}\varphi\in C_b(\RR^m)$,
$$\int_{\mathbb{R}^{m}}P^x_{s,t}\varphi(y)\,\mu^x_s(dy)=\lim_{r\to -\infty}P^x_{r,s} P^x_{s,t}\varphi (y)=\lim_{r\to-\infty}P^{x}_{r,t}\varphi (y)=\int_{\mathbb{R}^{m}}\varphi(y)\,\mu^x_t(dy).$$ This proves \eqref{ESM}.

For any Lipschitz function $\phi$ on $\RR^m$, \eref{F3.2} yields that
\begin{align*}
\left|P^x_{s,t} \phi(y)  -  \int_{\RR^m}\phi(z)\,\mu^x_t(dz)\right|
=  {\rm Lip}(\phi)\EE|Y^{s,x,y}_t - \eta^x_{t}|
\le  C {\rm Lip}(\phi)(1 + |x| + |y|)e^{-\int^t_s\alpha(u)\,du},
\end{align*} and so \eref{WErodicity} holds.

(ii) Let $\{\nu^x_{t}\}_{t\in\RR}$ be another evolution system of measures for $\{P^x_{s,t}\}_{t\geq s}$ so that
$
\sup_{t\in\RR}\int_{\RR^m}\,|z|\,\nu^x_t(dz)<\infty.
$
It holds that for any $\varphi \in C_b^1(\mathbb{R}^{m})$
$$\int_{\mathbb{R}^{m}}P^x_{s,t} \varphi (y)\,\nu^x _s(dy)=\int_{\mathbb{R}^{m}}\varphi (y)\,\nu^x_t(dy). $$
In order to prove $\nu^x_t=\mu^x_t$, it is sufficient to verify that for any $\varphi \in C_b^1(\mathbb{R}^{m})$
\begin{equation}
\lim_{s \to -\infty} \int_{\mathbb{R}^{m}}P^x_{s,t} \varphi (y)\,\nu^x_s(dy)=\int_{\mathbb{R}^{m}}\varphi (y)\,\mu^x_t(dy).\label{SufC}
\end{equation}
Indeed, note that
$$\int_{\mathbb{R}^{m}}P^x_{s,t} \varphi (y)\,\nu^x_s(dy)=\int_{\mathbb{R}^{m}} \left(P^x_{s,t} \varphi (y)-\int_{\mathbb{R}^{m}}\varphi (z) \,\mu^x_t(dz)\right)\,\nu^x_s(dy) +\int_{\mathbb{R}^{m}}\varphi (y)\,\mu^x_t(dy).$$
\eref{WErodicity} in turn yields that
\begin{align*}
 \left |\int_{\mathbb{R}^{m}}\left(P^x_{s,t} \varphi (y)-\int_{\mathbb{R}^{n}}\varphi (z) \mu^x_t(dz)\right )\,\nu^x_s(dy) \right|
 &\le  C\int_{\mathbb{R}^{m}}(1 + |x| + |y|)\,\nu^x_s(dy)e^{-\int^t_s\alpha(u)\,du}\\
 &
\leq C\left(1 + |x| + \sup_{t\in\RR}\int_{\RR^m} |z|\,\nu^x_t(dz)\right)e^{-\int^t_s\alpha(u)\,du} .
\end{align*}
It shows that \eref{SufC} holds by taking $s \to -\infty$ in the inequality above.
\end{proof}

\subsection{Nonautonomous Poisson equations}

Let $H:\RR\times\RR^n\times\RR^m\rightarrow \RR^n$ satisfy the  following \emph{centering condition}
\begin{equation}\label{CenCon}
\int_{\RR^m}H(s,x,y)\,\mu^{x}_s(dy)=0,\quad  s\in \RR,x\in\RR^n
\end{equation}
 and that for $i=0,1,2$ and $j=1,2,3$ with $1\leq i+j\leq 3$,
\begin{equation}
\sup_{t\in\RR,x\in\RR^n,y\in\RR^m} \|\partial_{x}^{i}\partial_{y}^{j}H(t,x,y) \|<\infty.\label{CenPol}
\end{equation}
For fixed $x\in  \RR^{n}$, we consider the following nonautonomous Poisson equation:
\begin{equation} \label{TimePE}
\partial_s\Phi(s,x,y)+\mathscr{L}^{x}_2(s)\Phi(s,x,\cdot)(y)=-H(s,x,y),\quad s\in\RR, y\in\RR^{m},
\end{equation}
where $\mathscr{L}^{x}_{2}(s)$ is the generator of the SDE \eref{FrozenE}, that is,
\begin{equation}
\mathscr{L}^{x}_{2}(s)\varphi(y):=\langle f(s,x,y),\nabla\varphi(y)\rangle+\frac{1}{2}\text{Tr} [(gg)^{\ast}(s,x,y)\nabla^2\varphi(y) ],\quad \varphi\in C^2(\RR^m).\label{L_2}
\end{equation}

\begin{proposition}\label{P3.6}
Suppose that Assumption {\rm\ref{A1}} holds. For any $H:\RR\times\RR^n\times\RR^m\rightarrow \RR^n$ satisfying \eqref{CenCon} and \eqref{CenPol}, define
\begin{equation}
\Phi(s,x,y):=\int^{+\infty}_{s}\EE H(r,x,Y^{s,x,y}_r)\,dr.\label{SPE}
\end{equation}
Then, $\Phi(s,x,y)$ is a solution to the nonautonomous Poisson equation \eqref{TimePE}. Moreover, there exists a constant $C>0$ such that for $s\in\RR$, $x\in \RR^n$ and $y\in \RR^m$,
\begin{equation}\label{E1}
\begin{split}
&|\Phi(s,x,y)|\leq C(1+|x|+|y|)\Lambda(s),\\
& \|\partial_x \Phi(s,x,y)\|+ \|\partial^2_x \Phi(s,x,y)\|\leq C(1+|x|+|y|)\Lambda_{\gamma}(s),\\
& \|\partial_y \Phi(s,x,y)\|\leq C\Lambda(s),\quad  \|\partial^2_y \Phi(s,x,y)\|\leq C\Lambda_{\gamma}(s),
\end{split}
\end{equation}
where $\Lambda_{\gamma}(s)$ is defined in \eqref{A12} and $\Lambda(s)=\Lambda_1(s)$.
\end{proposition}
\begin{remark}
Here we do not claim that the solution to nonautonomous Poisson equation \eref{TimePE} is unique. Furthermore,  additionally assume that the solution $\Phi(t,x,y)$ satisfies the \emph{centering condition} (i.e., $\int_{\RR^m}\Phi(s,x,y)\,\mu^{x}_s(dy)=0$ for any $s\in\RR$ and $ x\in\RR^n$) and \eref{E1}. Then the solution is unique. In fact, if $\tilde\Phi(t,x,y)$ is another solution to \eref{TimePE} and satisfies the \emph{centering condition} and \eref{E1}, then, applying It\^{o}'s formula and taking expectation on both sides, we find
\begin{align*}
\EE\tilde\Phi(t,x,Y^{s,x,y}_{t}) = &\tilde\Phi(s,x,y) + \EE\int^t_s(\partial_r\tilde\Phi(r,x,Y^{s,x,y}_{r})
+ \mathscr{L}^{x}_2(r)\tilde\Phi(r,x,\cdot)(Y^{s,x,y}_{r}))\,dr\\
=& \tilde\Phi(s,x,y) - \int^t_s \EE H(r,x,Y^{s,x,y}_{r})\,dr.
\end{align*}
On the other hand, by \eqref{WErodicity} and \eqref{E1},
\begin{align*}|\EE\tilde\Phi(t,x,Y^{s,x,y}_{t})|=&\Big|\EE\tilde\Phi(t,x,Y^{s,x,y}_{t})-\int_{\RR^m}\tilde\Phi(t,x,y)\,\mu_t^x(dy)\Big|\\
\le& C\|\tilde\Phi(t,x,\cdot)\|_{{\rm Lip}}(1+|x|+|y|)e^{-\int^t_s\alpha(u)\,du}\le C(1+|x|+|y|)\Lambda(t)e^{-\int^t_s\alpha(u)\,du}.\end{align*}
Then, $\lim_{t\to +\infty}|\EE\tilde\Phi(t,x,Y^{s,x,y}_{t})|=0$ and so
 $
0=\tilde\Phi(s,x,y)-\Phi(s,x,y)$ for all $s\in\RR$, $x\in\RR^n$ and $y\in\RR^m.
$
\end{remark}

\begin{proof} [Proof of Proposition $\ref{P3.6}$]
(i) It is well known (see e.g. \cite[(1.2)]{DR2008}) that $P^{x}_{s,r}H(r,x,\cdot)(y)=\EE H(r,x,Y^{s,x,y}_r)$ and
$$
\frac{d}{ds}P^{x}_{s,r}H(r,x,\cdot)(y)=-\mathscr{L}^{x}_2(s) [P^{x}_{s,r}H(r,x,\cdot)(y) ].
$$
Hence, according to the definition of $\Phi(s,x,y)$,
\begin{align*}
\partial_s\Phi(s,x,y)=&-H(s,x,y)+\int^{+\infty}_s \frac{d}{ds}P^{x}_{s,r}H(r,x,\cdot)(y)\, dr\\
= &-H(s,x,y)+\int^{+\infty}_s -\mathscr{L}^{x}_2(s) [P^{x}_{s,r}H(r,x,\cdot)(y) ] \,dr\\
=&-H(s,x,y)-\mathscr{L}^{x}_2(s)\Phi(s,x,y),
\end{align*}
which implies $\Phi(s,x,y)$ is a solution to the equation \eref{TimePE}.

Note that for any $l,l_1,l_2\in \RR^m$,
$$
\partial_y \EE H(r,x,Y^{s,x,y}_r)\cdot l=\EE[\partial_y H(r,x,Y^{s,x,y}_r)\cdot(\partial_yY^{s,x,y}_r\cdot l)]
$$ and
\begin{align*}
&\partial^2_y \EE H(r,x,Y^{s,x,y}_r)\cdot(l_1, l_2)\\
&=  \EE  [\partial^2_y H(r,x,Y^{s,x,y}_r)\cdot (\partial_yY^{s,x,y}_r\cdot l_1, \partial_yY^{s,x,y}_r\cdot l_2  ) ]
 + \EE[\partial_y H(r,x,Y^{s,x,y}_r)\cdot(\partial^2_yY^{s,x,y}_r\cdot (l_1, l_2))].
\end{align*}
Then, according to \eref{CenPol} and \eref{EPY},
\begin{align*}\label{partialyP}
 \|\partial_y \EE H(r,x,Y^{s,x,y}_r)\|\leq Ce^{-\int^r_s \alpha(u)\,du},\quad \|\partial^2_y \EE H(r,x,Y^{s,x,y}_r)\|\leq Ce^{-\gamma\int^r_s \alpha(u)\,du}.
\end{align*}
Hence,
\begin{align*}
\|\partial_y\Phi(s,x,y)\|\leq &\int^{+\infty}_{s}\| \partial_y \EE H(r,x,Y^{s,x,y}_r)\|\,dr
\leq   C\int^{+\infty}_{s}e^{-\int^r_s\alpha(u)\,du}\,dr
\leq   C\Lambda(s).
\end{align*}
Similarly,
$$
\|\partial^2_y\Phi(s,x,y)\|\leq C\Lambda_\gamma(s).
$$
Hence, we prove the third inequality in \eqref{E1}.

(ii) Under the centering condition \eref{CenCon},  for $r\geq s$,
$$
\EE H(r,x,Y^{s,x,y}_r)
=\EE H(r,x,Y^{s,x,y}_r)-\EE H(r,x,\eta^x_r).
$$ Note that, by \eref{CenPol},
$$|H(r,x,y_1)-H(r,x,y_2)|\leq C|y_1-y_2|.$$
Then, according to \eref{WErodicity},
\begin{align*}
|\Phi(s,x,y)|\leq &\int^{+\infty}_{s}\left|\EE H(r,x,Y^{s,x,y}_r)-\EE H(r,x,\eta^x_r)\right|\,dr\\
\leq & C(1+|x|+|y|)\int^{+\infty}_{s}e^{-\int^r_s\alpha(u)\,du}\,dr\\
=&C(1+|x|+|y|)\Lambda(s),
\end{align*}
which proves the first inequality in \eqref{E1}.

(iii) Thanks to \eref{CenCon} again, it also holds that for any unit vector $k \in \RR^n$,
\begin{align*}
 \partial_x\EE H(r,x,Y^{s,x,y}_r)\cdot k&=\left[\EE \partial_x H(r,x,Y^{s,x,y}_r)\cdot k-\EE \partial_x H(r,x,\eta^x_r)\cdot k\right]\\
&\quad+\EE\left[\partial_y H(r,x,Y^{s,x,y}_r)\cdot(\partial_xY^{s,x,y}_r\cdot k)\right]-\EE \left[\partial_y H(r,x,\eta^x_r)\cdot(\partial_x \eta^x_r\cdot k)\right]\\
&=:I_1+I_2.
\end{align*}
It follows from \eqref{F3.2} that
$$
|I_1|\leq C [\EE |Y^{s,x,y}_r-\eta^x_r |^2  ]^{1/2}
\leq C(1+|x|+|y|)e^{-\int^r_s \alpha(u)\,du}.$$
 On the other hand, \eqref{F3.2}, \eref{Epartial_xeta} and \eref{EPY} show that
\begin{align*}
|I_2|\leq& |\EE [\partial_y H(r,x,Y^{s,x,y}_r)\cdot(\partial_xY^{s,x,y}_r\cdot k)  ]-\EE   [\partial_y H(r,x,\eta^x_r)\cdot(\partial_xY^{s,x,y}_r\cdot k)  ]  |\\
 &+ |\EE  [\partial_y H(r,x,\eta^x_r)\cdot(\partial_xY^{s,x,y}_r\cdot k)  ]-\EE  [\partial_y H(r,x,\eta^x_r)\cdot(\partial_x \eta^x_r\cdot k)  ]  |\\
\leq &C (\EE|Y^{s,x,y}_r-\eta^x_r|^2 )^{1/2} (\EE|\partial_xY^{s,x,y}_r\cdot k|^2 )^{1/2}+C (\EE|\partial_xY^{s,x,y}_r\cdot k-\partial_x\eta^x_r\cdot k|^2 )^{1/2}\\
\leq &C(1+|x|+|y|)e^{-\gamma\int^r_s \alpha(u)\,du}.
\end{align*}
Hence,
\begin{align*}
\|\partial_x\Phi(s,x,y)\|\leq   C(1+|x|+|y|)\int^{+\infty}_{s}e^{-\gamma\int^r_s \alpha(u)\,du}\,dr\leq C(1+|x|+|y|)\Lambda_{\gamma}(s).
\end{align*}

Furthermore, note that for any unit vectors $k_1,k_2\in \RR^n$,
\begin{align*}
& \partial^2_x \{\EE H(r,x,Y^{s,x,y}_r)-\EE H(r,x,\eta^x_r) \}\cdot (k_1,k_2)\\
&= \EE [\partial^2_x H(r,x,Y^{s,x,y}_r)\cdot (k_1,k_2)]-\EE [\partial^2_x H(r,x,\eta^x_r)\cdot (k_1,k_2)]\\
&\quad+\EE [\partial_y\partial_{x} H(r,x,Y^{s,x,y}_r)\cdot (k_1,\partial_xY^{s,x,y}_r\cdot k_2)]-\EE [\partial_y\partial_{x} H(r,x,\eta^x_r)\cdot (k_1,\partial_x\eta^x_r\cdot k_2)]\\
&\quad+\EE [\partial_{x}\partial_{y} H(r,x,Y^{s,x,y}_r)\cdot(\partial_xY^{s,x,y}_r\cdot k_1, k_2) ]-\EE [\partial_x\partial_{y} H(r,x,\eta^x_r)\cdot(\partial_x\eta^x_r\cdot k_1, k_2) ]\\
&\quad+\EE [\partial^2_y H(r,x,Y^{s,x,y}_r)\cdot(\partial_xY^{s,x,y}_r\cdot k_1,\partial_xY^{s,x,y}_r\cdot k_2) ]-\EE  [\partial^2_y H(r,x,\eta^x_r)\cdot(\partial_x \eta^x_r\cdot k_1,\partial_x \eta^x_r\cdot k_2)  ]\\
&\quad+\EE  [\partial_y H(r,x,Y^{s,x,y}_r)\cdot(\partial^2_xY^{s,x,y}_r\cdot (k_1,k_2))  ]-\EE [\partial_y H(r,x,\eta^x_r)\cdot(\partial^2_x \eta^x_r\cdot (k_1,k_2))  ]\\
&=: \sum^5_{j=1}\tilde{I}_j.
\end{align*}

From \eref{F3.2} and \eqref{CenPol}, one has
$$
|\tilde I_1|\leq C\EE|Y^{s,x,y}_r-\eta^x_r|\leq C(1+|x|+|y|)e^{-\int^r_s \alpha(u)\,du}.
$$
Using \eref{EPY}, \eref{F3.2}, \eref{Epartial_xeta} and \eqref{CenPol}, we obtain
\begin{align*}
|\tilde I_2|
\leq &C\EE\left|Y^{s,x,y}_r-\eta^x_r\right|\left|\partial_xY^{s,x,y}_r\cdot k_2\right|+C\EE\left|\partial_xY^{s,x,y}_r\cdot k_2-\partial_x\eta^x_r\cdot k_2\right| \\
\leq &C(1+|x|+|y|)e^{-\gamma\int^r_s \alpha(u)\,du}.
\end{align*}
Similarly, we can get that
$$
|\tilde I_3|
\leq C(1+|x|+|y|)e^{-\gamma\int^r_s \alpha(u)\,du}.$$

For $\tilde I_4$ and $\tilde I_5$, we apply \eref{EPY}, \eref{F3.2}, \eref{Epartial_xeta} and \eqref{CenPol} to obtain
\begin{align*}
|\tilde I_4|
\leq &C\EE(|Y^{s,x,y}_r-\eta^x_r||\partial_xY^{s,x,y}_r\cdot k_1||\partial_xY^{s,x,y}_r\cdot k_2|)\\
& +C\EE(|\partial_xY^{s,x,y}_r\cdot k_1-\partial_x\eta^x_r\cdot k_1||\partial_xY^{s,x,y}_r\cdot k_2|)\\
& +C\EE(|\partial_xY^{s,x,y}_r\cdot k_2-\partial_x\eta^x_r\cdot k_2||\partial_x\eta^x_r\cdot k_1|)\\
\leq &C(1+|x|+|y|)e^{-\gamma\int^r_s \alpha(u)\,du}
\end{align*}
and
\begin{align*}
|\tilde I_5|
\leq &C\EE(|Y^{s,x,y}_r-\eta^x_r||\partial^2_xY^{s,x,y}_r\cdot (k_1,k_2)|) +C\EE|\partial^2_xY^{s,x,y}_r\cdot (k_1,k_2)-\partial^2_x\eta^x_r\cdot (k_1,k_2)| \\
\leq &C(1+|x|+|y|)e^{-\gamma\int^r_s \alpha(u)\,du}.
\end{align*}
Therefore,
\begin{align*}
 \|\partial^2_x\Phi(s,x,y)\|\leq &\int^{+\infty}_{s} \| \partial^2_x \{\EE H(r,x,Y^{s,x,y}_r)-\EE H(r,x,\eta^x_r) \}\|\,dr\\
\leq & C(1+|x|+|y|)\int^{+\infty}_{s}e^{-\gamma\int^r_s \alpha(u)\,du}\,dr \\
= & C(1+|x|+|y|)\Lambda_{\gamma}(s).
\end{align*}
Combining with all the estimates above, we obtain the second inequality in \eref{E1}.  The proof is complete.
\end{proof}

\section{Time-inhomogeneous multi-scale SDEs: general case}

As an application of nonautonomous Poisson equations discussed in Section 2, we study the averaging principle of the following time-inhomogeneous multi-scale SDE:
\begin{equation} \label{Equation}
\begin{cases}
d X^{\varepsilon}_t = b(X^{\varepsilon}_{t}, Y^{\varepsilon}_{t})\,dt+\sigma(X^{\varepsilon}_{t},Y^{\varepsilon}_{t})\,d W^1_t,\quad X^{\varepsilon}_0=x\in\RR^{n}, \vspace{1mm}\\
d Y^{\varepsilon}_t =\varepsilon^{-1}f(t/\varepsilon,X^{\varepsilon}_{t},Y^{\varepsilon}_{t})\,dt+\varepsilon^{-1/2}g(t/\varepsilon,X^{\varepsilon}_{t},Y^{\varepsilon}_{t})\,d W^{2}_t,\quad Y^{\varepsilon}_0=y\in\RR^{m},
\end{cases}
\end{equation}
where $b\in C^{2,3}(\RR^n\times\RR^m,\RR^n)$, $\sigma\in C(\RR^{n}\times\RR^m, \RR^{n\times d_1})$,  $f(t,\cdot,\cdot)\in C^{2,3}(\RR^n\times\RR^m,\RR^m)$ and $g(t,\cdot,\cdot)\in C^{2,3}(\RR^n\times\RR^m,\RR^{m\times d_2})$ for any $t\ge0$,
and $W^1:=\{W_t^1\}_{t\ge0}$ and $W^2:=\{W_t^2\}_{t\ge0}$ are two independent $d_1$ and $d_2$ dimensional standard Brownian motions respectively. From this section,  we always suppose that the following conditions hold.

\begin{conditionB}\label{B1}
Suppose $f$ and $g$ satisfies Assumption {\rm\ref{A1}} for any $t\geq 0$, $x\in \RR^n$ and $y\in \RR^m$ with $\alpha: [0,\infty)\to (0,\infty)$, which fulfills
\begin{equation}
\int_{0}^{+\infty}  \alpha(u)\,du =\infty, \quad \int^{+\infty}_{t} e^{-\gamma\int^r_t\alpha(u)\,du}\,dr<\infty\hbox{ for all } t\in [0,\infty)\label{alpha2}
\end{equation}
with some $\gamma\in (0,1)$.
\end{conditionB}

\begin{conditionB}\label{A2}
There exists a constant $C>0$ such that for any $x_1,x_2\in \RR^n$ and $y_1,y_2\in \RR^m$,
\begin{equation}\label{A21}
|b(x_1,y_1)-b(x_2,y_2)|+ \|\sigma(x_1,y_1)-\sigma(x_2,y_2)\| \leq C(|x_1-x_2|+|y_1-y_2|),
\end{equation}
and for any $i=0,1,2$ and $j=1,2,3$ with $1\leq i+j\leq 3$,
\begin{equation}
\sup_{x\in\RR^n,y\in\RR^m}\|\partial^{i}_x\partial^j_yb(x,y)\|\leq C,\quad \sup_{x\in\RR^n,y\in\RR^m}\|\sigma(x,y)\|\leq C.\label{A22}
\end{equation}
\end{conditionB}

\begin{remark}
To make full use of the evolution system of measures in Section 2, we need to extend the fast component $\{Y^{\varepsilon}_t\}_{t\ge0}$ in the multi-scale SDE \eqref{Equation} to $\{Y^{\varepsilon}_t\}_{t\in \RR}$. Take another $d_2$-dimensional standard Brownian motion $\{\tilde{W}^2_t\}_{t\geq 0}$ that is independent of $\{W^2_t\}_{t\geq 0}$, and set
\begin{equation*}
\hat{W}^2_t=\begin{cases}
W^2_t,\quad\,\, t\geq 0,\\
\tilde{W}^2_{-t}, \quad t< 0.
\end{cases}
\end{equation*}
For any $x\in\RR^n$ and $y\in\RR^m$, define
\begin{equation*}
\hat{\alpha}(t)=
\begin{cases}
\alpha(t),\quad\,\,  t\geq 0,\\
\alpha(-t),\quad t< 0,
\end{cases}
\hat f(t,x,y)=
\begin{cases}
f(t,x,y),\quad\,\,  t\geq 0,\\
f(-t,x,y),\quad  t< 0,
\end{cases}
\hat g(t,x,y)=
\begin{cases}
g(t,x,y),\quad  t\geq 0,\\
g(-t,x,y), \quad  t< 0.
\end{cases}
\end{equation*}
Thus, $\hat{f}$ and $\hat{g}$ satisfy Assumption {\rm\ref{A1}} with the function $\hat{\alpha}$ on $\RR$. Hence, the associated time-inhomogeneous frozen SDE
\begin{equation*}
dY_{t}=\hat f(t,x,Y_{t})\,dt+\hat g(t,x,Y_{t})\,d \hat W^2_t,\quad Y_{s}=y\in \RR^{m}, s\in \RR,
\end{equation*}
admits a unique strong solution $Y^{s,x,y}:=\{Y^{s,x,y}_t\}_{t\ge s}$. Let $\{P^{x}_{s,t}\}_{t\ge s}$ be the semigroup of the process $\{Y_{t}^{s,x,y}\}_{t\geq s}$. In the rest of this paper, we always let $\{\mu^x_t\}_{t\in\RR}$ be an evolution system of measures of the semigroup $\{P^x_{s,t}\}_{t\geq s}$ given in Proposition \ref{P:2.6}.
\end{remark}
 \begin{remark}
\begin{itemize}
\item[{\rm(i)}] Under Assumptions {\rm\ref{B1}} and {\rm\ref{A2}}, the coefficients of the SDE (\ref{Equation}) satisfy the one-sided Lipschitz condition and the globally linear growth condition. Therefore, there exists a unique solution $\{(X^{\varepsilon}_t,Y^{\varepsilon}_t)\}_{t\ge0}$ to the SDE (\ref{Equation}) for any $\varepsilon>0$, $x\in\RR^{n}$ and $y\in \RR^{m}$.
\item[{\rm(ii)}] The first condition in \eref{A22} corresponds to \eqref{CenPol} as we will take $H(t,x,y)=b(t,x,y)-\bar{b}(t,x)$, where $\bar{b}(t,x)$ is defined by \eqref{barb}.
On the other hand, the second condition in \eref{A22}, i.e., $\sigma$ is bounded, which will be used in the proof of our first main result Theorem \ref{main result 1} below;  however, such kind requirement can be weaken if the solution $\{(X^{\varepsilon}_t,Y^{\varepsilon}_t)\}_{t\ge0}$ has higher finite moment.
    \end{itemize}
\end{remark}

We now establish some uniform estimates of the solution $\{(X^{\varepsilon}_t, Y^\varepsilon_t)\}_{t\ge0}$.

\begin{lemma} \label{PMY}
Suppose that Assumptions {\rm\ref{B1}} and {\rm\ref{A2}} hold. Then, for any $T>0$, there exists a constant $C_{T}>0$ such that
\begin{equation}
\sup_{\varepsilon>0}\mathbb{E} \Big(\sup_{t\in [0, T]}|X_{t}^{\varepsilon}|^4 \Big)\leq C_{T}(1+|x|^4+|y|^4) \label{X}
\end{equation} and
\begin{equation}
\sup_{\varepsilon>0}\sup_{t\in [0,T]}\mathbb{E}|Y_{t}^{\varepsilon}|^4\leq C_{T}(1+|x|^4+|y|^4).\label{Y}
\end{equation}
\end{lemma}

\begin{proof}
By It\^{o}'s formula and \eref{RE2},
\begin{align*}
\frac{d}{dt}\mathbb{E} |Y_t^{\varepsilon} |^{4}=&\frac 4{\varepsilon}\mathbb{E}(|Y_t^{\varepsilon }|^{2} \langle Y_t^{\varepsilon },f(t/\varepsilon,X_t^{\varepsilon },Y_t^{\varepsilon }) \rangle)\\
&+\frac 2{\varepsilon}\mathbb{E}(|Y_t^{\varepsilon }|^{2}\|g(t/\varepsilon,X_t^{\varepsilon },Y_t^{\varepsilon }) \|^{2})+\frac{4}{\varepsilon}\mathbb{E} (|g^{*}(t/\varepsilon,X_t^{\varepsilon },Y_t^{\varepsilon })Y^{\varepsilon}_t|^2)\\
\leq & \frac{4}{\varepsilon}\mathbb{E}(|Y_t^{\varepsilon }|^{2}  \langle Y_{t}^{\varepsilon },f(t/\varepsilon,X_{t}^{\varepsilon },Y_{t}^{\varepsilon })  \rangle)+\frac{6}{\varepsilon}\mathbb{E}(|Y_t^{\varepsilon }|^{2} \|g(t/\varepsilon,X_{t}^{\varepsilon },Y_{t}^{\varepsilon })  \|^{2})  \\
\leq &   -\frac{4\gamma\alpha(t/\varepsilon) }{\varepsilon}\mathbb{E} |Y_{t}^{\varepsilon } |^4+\frac{C\alpha(t/\varepsilon)}{\varepsilon} (1+\EE|X_{t}^{\varepsilon }|^4 ).
\end{align*}
This along with the Grownall inequality and \eqref{A12} implies
\begin{align*}
\mathbb{E} |Y_t^{\varepsilon } |^{4}\leq &e^{-\frac{4\gamma}{\varepsilon}\int^t_0 \alpha(u/\varepsilon)\,du}|y|^4+\frac{C}{\varepsilon}\int^t_0 e^{- \frac{4\gamma}{\varepsilon}\int^t_u\alpha(r/\varepsilon)\,dr }\alpha(u/\varepsilon) (1+\EE|X_{u}^{\varepsilon }|^4 )\,du \\
\leq &e^{-4\gamma\int^{t/\varepsilon}_0 \alpha(u)\,du}|y|^4+C\sup_{0\leq s\leq t} (1+\EE|X_{s}^{\varepsilon }|^4 ).
\end{align*}

On the other hand, by \eqref{A21} and \eqref{A22}, for $t\in [0, T]$,
\begin{align*}
\mathbb{E} \Big(\sup_{s\in [0,t]}|X^{\varepsilon}_s|^4 \Big)\leq &C_{T}(1+|x|^4)+C_{T}\int^{t}_{0}\EE|X^{\varepsilon}_s|^4\,ds+C_{T}\int^{t}_{0}\EE|Y^{\varepsilon}_s|^4\, ds\\
\leq &C_{T}(1+|x|^4+|y|^4)+C_{T}\int^{t}_{0}\mathbb{E} \Big(\sup_{u\in [0,s]}|X^{\varepsilon}_u|^4 \Big)\,ds.
\end{align*}
Using the Grownall inequality again, we find that
$$
\mathbb{E} \Big(\sup_{t\in [0,T]}|X^{\varepsilon}_t|^4\Big )
\leq C_{T} (1+|x|^4+|y|^4),
$$
which implies \eref{X}. As a consequence, \eref{Y} holds obviously.
\end{proof}

\subsection{Strong averaging principle: general case }
In this subsection, we focus on the strong averaging principle of the stochastic system \eref{Equation}. A counterexample in \cite[Section 4.1]{L2010} shows that the strong averaging principle may fail if the diffusion coefficient $\sigma(x,y)$ depends on $y$. Hence, in this part we always assume that $\sigma(x,y)\equiv\sigma(x)$, i.e., $\sigma(x,y)$ is independent of $y$. Then, we consider the following averaged equation:
\begin{equation}
d\bar{X}^{\varepsilon}_{t}=\bar{b}(t/\varepsilon,\bar{X}^{\varepsilon}_t)\,dt+\sigma(\bar{X}^{\varepsilon}_t)\,d W^1_t,\quad\bar{X}^{\varepsilon}_{0}=x, \label{AVE}
\end{equation}
where
\begin{align}
\bar{b}(t,x)=\displaystyle\int_{\RR^{m}}b(x,y)\,\mu^x_t(dy),\quad t\geq 0.\label{barb}
\end{align}

\begin{lemma} \label{PMA} Suppose that Assumptions {\rm\ref{B1}} and {\rm\ref{A2}} hold, and that $\sigma(x,y)=\sigma(x)$ for all $x\in \RR^n$ and $y\in \RR^m$. Then, for any $x\in\RR^{n}$ and $\varepsilon>0$, the time-inhomogeneous SDE \eqref{AVE}  has a unique solution, denoted by $\{\bar{X}^{\varepsilon}_t\}_{t\ge0}$. Moreover, for $T>0$, there exists a constant $C_{T}>0$ such that
\begin{equation}
\sup_{\varepsilon>0}\mathbb{E}\Big(\sup_{t\in [0, T]}|\bar{X}^{\varepsilon}_{t}|^{2}\Big)\leq C_{T}(1+|x|^{2}).\label{EAVE}
\end{equation}
\end{lemma}

\begin{proof}
According to \eref{WErodicity}, \eref{A21} and \eref{FR1}, for any $t\geq 0$, $t\ge s$ and $x_1,x_2\in\RR^n$,
\begin{align*}
 |\bar b(t,x_{1})-\bar b(t,x_{2}) |
= & \left|\int_{\RR^m}b(x_1,y)\mu^{x_1}_t(dy)-\int_{\RR^m}b(x_2,y)\mu^{x_2}_t(dy) \right|\\
\leq &  \left|\int_{\RR^m}b(x_1,y)\mu^{x_1}_t(dy)-\EE b(x_{1},Y^{s,x_1,0}_t) \right| + \left|\int_{\RR^m}b(x_2,y)\mu^{x_2}_t(dy)-\EE b(x_{2},Y^{s,x_2,0}_t) \right|\\
& + |\EE b(x_{1},Y^{s,x_1,0}_t)-\EE b(x_{2},Y^{s,x_2,0}_t) |\\
\le &C(1+|x_1|+|x_2|)e^{-\int^t_s  \alpha(u)\,du}+C|x_1-x_2| +C\EE |Y^{s,x_1,0}_t-Y^{s,x_2,0}_t |\\
\le &C(1+|x_1|+|x_2|)e^{-\int^t_s  \alpha(u)\,du}+C|x_1-x_2|.
\end{align*}
Letting $s\to -\infty$ in the inequality above and using \eqref{A12},
\begin{equation}
 |\bar b(t,x_{1}) - \bar b(t,x_{2}) |\le C|x_{1} - x_{2}|. \label{barc1}
\end{equation}

On the other hand, by \eqref{F3.2}, there exists $C>0$ such that for any $t\geq 0$,
\begin{align}
|\bar b(t,x_{1})|\leq \EE|b(x,\eta^x_t)|\leq C(1+|x|+\EE|\eta^x_t|)\leq C(1+|x|).\label{linear barb}
\end{align}

Combining \eqref{linear barb} and \eqref{barc1} with \eqref{A21}, we can easily see that the SDE \eref{AVE} admits a unique solution, and \eref{EAVE} holds obviously.
\end{proof}

The following is the first main result of our paper.

\begin{theorem}\label{main result 1}
Suppose that Assumptions {\rm\ref{B1}} and {\rm\ref{A2}} hold, and $\sigma(x,y)=\sigma(x)$ for all $x\in \RR^n$ and $y\in \RR^m$. Let $\{(X_t^\varepsilon, Y_t^\varepsilon )\}_{t\ge0}$ and $\{\bar X_t^\varepsilon\}_{t\ge0}$ be the solutions to \eqref{Equation} and \eqref{AVE} respectively. Then for $T>0$, there exists $C_T>0$ such that for $(x,y)\in\RR^{n}\times\RR^{m}$ and $\varepsilon>0$,
 \begin{equation}\label{R1}\begin{split}
\sup_{t\in [0, T]}\EE |X_{t}^{\varepsilon}  -  \bar{X}^{\varepsilon}_{t} |^2  \leq    C_T (1 + |x|^4 + |y|^4 )\varepsilon^{2}\left[\sup_{0\leq t\leq T}|\Lambda_{\gamma}(t/\varepsilon)|^2  +  \int^{T/\varepsilon}_0     \alpha(s)\Lambda^2(s)\,ds\right].
\end{split}
\end{equation}
\end{theorem}

\begin{proof}
Note that, by \eqref{Equation} and \eqref{AVE},
\begin{align*}
X_{t}^{\varepsilon}-\bar{X}^{\varepsilon}_{t}
= &\int_{0}^{t}(b(X_{s}^{\varepsilon},Y_{s}^{\varepsilon})-\bar{b}(s/\varepsilon,X^{\varepsilon}_{s}))\,ds
+\int_{0}^{t} (\bar{b}(s/\varepsilon,X^{\varepsilon}_{s})-\bar{b}(s/\varepsilon,\bar{X}^{\varepsilon}_{s}))\,ds\\
 &+\int_{0}^{t}(\sigma(X_{s}^{\varepsilon})-\sigma(\bar{X}^{\varepsilon}_{s}) )\,dW^1_s.
\end{align*}
Hence, according to \eqref{barc1} and \eref{A21},  for any $t\in [0,T]$,
\begin{align*}
\sup_{0\leq s\leq t}\mathbb{E} |X^{\varepsilon}_s-\bar{X}^{\varepsilon}_s |^2
 \leq   C_T\sup_{0\leq s\leq t}\mathbb{E}\left|\int^s_0  (b(X^{\varepsilon}_{r},Y^{\varepsilon}_{r})  -  \bar{b}(r/\varepsilon,X^{\varepsilon}_{r}) ) \,dr\right|^2 + C_T\int_{0}^{t}\mathbb{E}|X_s^\varepsilon-\bar{X}^{\varepsilon}_s|^{2}\,ds.
\end{align*}
Consequently,
$$
\sup_{0\leq t\leq T}\mathbb{E}|X^{\varepsilon}_t-\bar{X}^{\varepsilon}_t|^2
\leq C_T\sup_{0\leq t\leq T}\mathbb{E}\left|\int^t_0 (b(X^{\varepsilon}_{s},Y^{\varepsilon}_{s})-\bar{b}(s/\varepsilon,X^{\varepsilon}_{s})) \,ds\right|^2.
$$

Denote by $H(s,x,y)=b(x,y)-\bar{b}(s,x)$. By the definition of $\bar{b}(s,x)$ and \eref{A22}, it is easy to check that $H(s,x,y)$ satisfies \eref{CenCon} and \eref{CenPol} for any $s\geq 0$. According to Proposition \ref{P3.6}, there exists a solution $\Phi(s,x,y)$ satisfying the following equation
$$
\partial_s\Phi(s,x,y)+\mathscr{L}^{x}_2(s)\Phi(s,x,\cdot)(y)=-H(s,x,y),\quad s\geq 0, x\in \RR^{n},y\in\RR^{m},
$$
and \eref{E1} holds for any $s\geq 0$, where $\mathscr{L}^x_{2}(s)$ is defined by \eref{L_2}. Thus,
\begin{equation}\begin{split}
\sup_{0\leq t\leq T}\mathbb{E} |X^{\varepsilon}_t-\bar{X}^{\varepsilon}_t |^2
 \leq C_T\sup_{0\leq t\leq T}\mathbb{E}\left|\int^t_0  (\partial_s\Phi(s/\varepsilon,X_{s}^{\varepsilon},Y^{\varepsilon}_s)
+  \mathscr{L}^{X_{s}^{\varepsilon}}_2(s/\varepsilon)\Phi(s/\varepsilon,X_{s}^{\varepsilon},\cdot)(Y^{\varepsilon}_{s}) ) \,ds\right|^2.\label{S5.61}\end{split}
\end{equation}

On the other hand, by It\^o's formula,
\begin{align*}
\Phi(t/\varepsilon,X_{t}^{\varepsilon},Y^{\varepsilon}_{t}) = &\Phi(0,x,y) + \varepsilon^{-1}\int^t_0(\partial_s\Phi(s/\varepsilon ,X_{s}^{\varepsilon},Y^{\varepsilon}_{s})
+ \mathscr{L}^{X_{s}^{\varepsilon}}_2(s/\varepsilon)\Phi(s/\varepsilon,X_{s}^{\varepsilon},\cdot)(Y^{\varepsilon}_{s}))\,ds\\
&+\int^t_0 \mathscr{L}^{Y^{\varepsilon}_{s}}_{1}\Phi(s/\varepsilon,\cdot, Y_{s}^{\varepsilon})(X^{\varepsilon}_{s})\,ds +M^{\varepsilon,1}_{t}+M^{\varepsilon,2}_{t},
\end{align*}
where
\begin{equation}\label{e:addop}
\mathscr{L}^{y}_{1}\phi(x):=\langle b(x,y), \nabla \phi(x)\rangle+\frac12\text{Tr} [(\sigma\sigma)^{\ast}(x)\nabla^2\phi(x) ], \quad \phi\in C^2(\RR^m),
\end{equation} and
$M^{\varepsilon,1}_{t}$ and $M^{\varepsilon,2}_{t}$ are local martingales defined respectively by
$$M^{\varepsilon,1}_{t}:=\int^t_0 \partial_x\Phi(s/\varepsilon,X_{s}^{\varepsilon}, Y^{\varepsilon}_{s})\cdot \sigma(X^{\varepsilon}_s)\,dW^1_s,$$
$$M^{\varepsilon,2}_{t}:=\varepsilon^{-1/2}\int^t_0 \partial_y\Phi(s/\varepsilon,X_{s}^{\varepsilon}, Y^{\varepsilon}_{s})\cdot g(s/\varepsilon,X^{\varepsilon}_s,Y^{\varepsilon}_s)\,dW^2_s.$$
Consequently,
\begin{align*}
& -\int^t_0 (\partial_s\Phi(s/\varepsilon,X_{s}^{\varepsilon},Y^{\varepsilon}_{s}) + \mathscr{L}^{X_{s}^{\varepsilon}}_2(s/\varepsilon)
\Phi(s/\varepsilon,X_{s}^{\varepsilon},\cdot)(Y^{\varepsilon}_{s})) \,ds\\
&= \varepsilon(\Phi(0,x,y) - \Phi(t/\varepsilon,X_{t}^{\varepsilon},Y^{\varepsilon}_{t}))+\varepsilon\int^t_0 \mathscr{L}^{Y^{\varepsilon}_{s}}_{1} \Phi(s/\varepsilon, \cdot,Y_{s}^{\varepsilon})(X^{\varepsilon}_{s})\,ds  + \varepsilon M^{\varepsilon,1}_{t} + \varepsilon M^{\varepsilon,2}_{t}.
\end{align*}

Combining this with \eref{S5.61}, we obtain that
\begin{align*}
&\sup_{t\in [0, T]}\EE |X_{t}^{\varepsilon}-\bar{X}^{\varepsilon}_{t} |^2\\
&\leq C_{T} \varepsilon^{2}\sup_{t\in[0,T]}\EE  |\Phi(0,x,y)-\Phi(t/\varepsilon,X_{t}^{\varepsilon},Y^{\varepsilon}_{t})  |^2
+  C_{T}\varepsilon^2\EE\left(\int^T_0  |\mathscr{L}^{Y^{\varepsilon}_{s}}_{1}(s/\varepsilon)\Phi(s/\varepsilon,\cdot, Y_{s}^{\varepsilon})(X^{\varepsilon}_{s}) |^2\, ds \right)\\
&\quad  +\varepsilon^2 \sup_{t\in[0,T]}\EE |M^{\varepsilon,1}_{t} |^2+\varepsilon^2\sup_{t\in[0,T]}\EE |M^{\varepsilon,2}_{t} |^2 \\
&=: \sum^4_{i=1}I^{\varepsilon}_i(T).
\end{align*}

It follows from \eqref{E1} and Lemma \ref{PMY} as well as $\Lambda(0)\le \sup_{0\le t\le T} \Lambda(t/\varepsilon)$ that
$$
I^{\varepsilon}_1(T)\leq C_T\varepsilon^{2} (1+|x|^{2}+|y|^{2} )\sup_{0\leq t\leq T}|\Lambda(t/\varepsilon)|^2.$$
According to \eqref{E1}, Assumption \ref{A2} and Lemma \ref{PMY}, we have
\begin{align*}
I^{\varepsilon}_2(T)\leq &C_{T}\varepsilon^2\int^T_0 \EE\left|\partial_x \Phi(s/\varepsilon,X_{s}^{\varepsilon},Y^{\varepsilon}_{s})\cdot b(X^{\varepsilon}_s, Y^{\varepsilon}_s) + \frac12\text{Tr} [(\sigma\sigma^{\ast})(X_{s}^{\varepsilon})\partial^2_x \Phi(s/\varepsilon,X_{s}^{\varepsilon},Y^{\varepsilon}_{s}) ]\right|^2\,ds\\
\leq &C_{T}\varepsilon^{2} (1+|x|^{4}+|y|^{4} )\int^T_0 \Lambda^2_{\gamma}(s/\varepsilon)\,ds\\
\leq &C_{T}\varepsilon^{2} (1+|x|^{4}+|y|^{4} )\sup_{0\leq t\leq T}|\Lambda_{\gamma}(t/\varepsilon)|^2,
\end{align*}
where the boundedness of $\sigma$ is used in the second inequality.
Lemma \ref{PMY} together with \eqref{A22} and \eqref{E1} yield that
\begin{align*}
I^{\varepsilon}_3(T)\leq & C_{T}\varepsilon^{2} \int^T_0 \Lambda^2_{\gamma}(s/\varepsilon)\EE (1 + |X_{s}^{\varepsilon}|^2  +  |Y^{\varepsilon}_{s}|^2 )\,ds
\leq C_{T}\varepsilon^{2} (1 + |x|^{2}+|y|^{2} )\sup_{0\leq t\leq T}|\Lambda_{\gamma}(t/\varepsilon)|^2.
\end{align*}
Furthermore, by \eqref{E1}, \eqref{A11} and Lemma \ref{PMY}, it is easy to see that
\begin{align*}
I^{\varepsilon}_4(T)\leq & C_{T}\varepsilon\int^{T}_0 \alpha(s/\varepsilon)\Lambda^2(s/\varepsilon)\EE(1+|X_{s}^{\varepsilon}|^2+|Y^{\varepsilon}_{s}|^2)\,ds\\
\leq &C_{T}\varepsilon^{2} (1+|x|^{2}+|y|^{2} )\int^{T/\varepsilon}_0 \alpha(s)\Lambda^2(s)\,ds.
\end{align*}
Hence, the desired assertion follows by putting all the estimates above together.
\end{proof}

In the following, we make some remarks on convergence rates in Theorem \ref{main result 1}.

\begin{lemma}\label{RE3.5}
Suppose that  $\alpha(\cdot)$ is differentiable on $(0,\infty)$ such that
\begin{equation}\label{e:remarkadd}
\lim_{t\to +\infty}\alpha(t)\in [0,\infty],\quad
 \lim_{t\rightarrow +\infty} \frac  {e^{-\gamma\int^t_0\alpha(u)\,du}}{\alpha(t)}=0, \quad \lim_{t\rightarrow +\infty}[1/\alpha(t)]'=0.
\end{equation}
Then for any $T>0$, there are constants $\varepsilon_0, C>0$ so that
 for any $\varepsilon\in (0,\varepsilon_0]$,
$$
\sup_{t\in [0,T]}|\Lambda_{\gamma}(t/\varepsilon)|^2\leq   C\int^{T/\varepsilon}_0 \alpha(s)\Lambda^2(s)\,ds.
$$
\end{lemma}
\begin{proof} (i) According to \eref{A12}, $\Lambda_{\gamma}(0)=\int^{+\infty}_{0}e^{-\gamma\int^r_0\alpha(u)\,du}\,dr<\infty$. This along with \eqref{e:remarkadd} yields that
$$
\lim_{t\rightarrow +\infty}\alpha(t)\Lambda_{\gamma}(t)=\lim_{t\rightarrow +\infty}\frac{\int^{+\infty}_t e^{-\gamma\int^r_0\alpha(u)du}dr}{e^{-\gamma\int^t_0\alpha(u)du}/\alpha(t)}=\lim_{t\rightarrow +\infty}\frac{1}{\gamma-[1/\alpha(t)]'}=\frac{1}{\gamma}.
$$
Similarly, it holds that $\lim_{t\rightarrow +\infty}\alpha(t)\Lambda(t)=1$. In particular,
$$
\lim_{t\rightarrow +\infty}\frac{\Lambda_{\gamma}(t)}{\Lambda(t)}=\frac{1}{\gamma}$$ and $$\lim_{t\to +\infty}\Lambda_{\gamma}(t)=\gamma^{-1}\lim_{t\to +\infty}1/\alpha(t).$$

When $\lim_{t\to+\infty}\alpha(t)\in (0,\infty]$, it is clear that there are constants  $T_0, C_0>0$ such that for all $t\ge T_0$,
\begin{equation}\label{ConA34}\Lambda^2_{\gamma}(t)\leq C_0\int^{t}_0 \alpha(s)\Lambda^2(s) \,ds.\end{equation} When $\lim_{t\to+\infty}\alpha(t)=0$,
\begin{align*}
\lim_{t\to +\infty}\frac{\Lambda^2_{\gamma}(t)}{\int^t_0 \alpha(s)\Lambda^2(s)\,ds}  =  \frac{1}{\gamma^2}\lim_{t\to+\infty}\frac{\Lambda^2(t)}{\int^t_0 \alpha(s)\Lambda^2(s)\,ds}  =  \frac{1}{\gamma^2}\lim_{t\to+\infty}\frac{2\Lambda(t)[\Lambda(t)\alpha(t)-1]}{\alpha(t)\Lambda^2(t)}=0,
\end{align*}
and so \eref{ConA34} is also satisfied.

(ii) For any $T>0$, take $\varepsilon_0>0$ with $T/\varepsilon_0=T_0$. Then, by \eref{ConA34}, for any $\varepsilon\in (0,\varepsilon_0]$,
$$
\sup_{t\in [\varepsilon T/\varepsilon_0,T]}|\Lambda_{\gamma}(t/\varepsilon)|^2\leq C_0\int^{T/\varepsilon}_0 \alpha(s)\Lambda^2(s)\,ds.
$$
Note also that $\Lambda_{\gamma}(t)$ is continuous on $[0, T/\varepsilon_0]$, and so
$$
\sup_{t\in [0,\varepsilon T/\varepsilon_0]}|\Lambda_{\gamma}(t/\varepsilon)|^2=\sup_{t\in [0,T/\varepsilon_0]}|\Lambda_{\gamma}(t)|^2\leq C_T.
$$
Therefore, for any $\varepsilon\in (0,\varepsilon_0]$,
\begin{align*}
\sup_{t\in [0,T]}|\Lambda_{\gamma}(t/\varepsilon)|^2\leq & \sup_{t\in [\varepsilon T/\varepsilon_0,T]}|\Lambda_{\gamma}(t/\varepsilon)|^2+\sup_{t\in [0,\varepsilon T/\varepsilon_0]}|\Lambda_{\gamma}(t/\varepsilon)|^2\\
\leq& C_T\int^{T/\varepsilon}_0 \alpha(s)\Lambda^2(s)\,ds.
\end{align*} The proof is complete. \end{proof}

\begin{remark}\label{remark3.7}
Let $\alpha(t)=c_0(1+t)^{\beta}$ with $c_0>0$ and $\beta\in (-1,\infty)$. Then, \eqref{e:remarkadd} is satisfied.
According to the proof of Lemma \ref{RE3.5}, there exist two positive constants $c_1\leq c_2$ such that
$$c_1\int^{t}_0 \alpha^{-1}(s)\,ds\leq\int^{t}_0 \alpha(s)\Lambda^2(s)\,ds\leq c_2\int^{t}_0 \alpha^{-1}(s)\,ds,\quad  t\geq 0.$$ Thus, by Theorem \ref{main result 1}
and Lemma \ref{RE3.5},
\begin{equation}\label{EEquation}
\sup_{t\in [0, T]}[\EE |X_{t}^{\varepsilon}-\bar{X}^{\varepsilon}_{t}|^2]^{1/2}\leq \left\{\begin{array}{l}
\displaystyle C\varepsilon^{({1+\beta})/{2}},\quad\quad\quad\quad -1<\beta<1, \\
\displaystyle C\varepsilon\left(\log{1/\varepsilon}\right)^{1/2},\quad \beta=1,\\
\displaystyle C\varepsilon,\quad\quad\quad \quad\quad\quad\beta>1,
\end{array}\right.
\end{equation}
In particular, when $\beta=0$, the rate is of order $\varepsilon^{1/2}$, which is consistent with the optimal result for the convergence rate of the strong averaging principle in time-homogeneous settings (see e.g. \cite{L2010}). On the other hand, \eref{EEquation} indicates that the convergence rate can be slower or faster than $\varepsilon^{1/2}$ in time-inhomogeneous settings. That is, the convergence rate of the strong averaging principle is influenced by the function $\alpha(t)$ involved in Assumption {\rm\ref{B1}}.
\end{remark}

\subsection{Weak averaging principle: general case}\label{section3.2}
In Theorem \ref{main result 1}, the diffusion coefficient $\sigma(x,y)$ is assumed to be independent of $y$. This crucial assumption ensures the validity of the strong averaging principle. For general diffusion coefficient $\sigma(x,y)$, we can investigate the corresponding weak convergence. To facilitate this, we assume that Assumption {\rm\ref{B1}} holds and the following assumption is satisfied:
\begin{conditionB}\label{A3}
 \begin{itemize} \it
 \item[{\rm(i)}] Assume that $f(t,\cdot,\cdot)\in C^{4,5}(\RR^n\times\RR^m,\RR^m)$ and $g(t,\cdot,\cdot)\in C^{4,5}(\RR^n\times\RR^m,\RR^{m\times d_2})$. Moreover, there exists $C>0$ such that for any $x\in\RR^n,y\in\RR^m$ and $t\geq 0$
\begin{equation}\label{ReA11}
\|\partial^{i}_x\partial^j_yf(t,x,y)\|+\|\partial^{i}_x\partial^j_yg(t,x,y)\|^2\leq C\alpha(t),\quad  0\leq i\leq 4,0\leq j\leq 5~ {\text with}~ 1\leq i+j\leq 5.
\end{equation}

\item[{\rm(ii)}] Assume that $b\in C^{4,4}(\RR^n\times\RR^m,\RR^n)$ and $\sigma\in C^{4,4}(\RR^n\times\RR^m,\RR^{n\times d_1})$. Moreover, there exists $C>0$ such that for any $x\in\RR^n,y\in\RR^m$,
\begin{align}
\|\partial^{i}_x\partial^j_yb(x,y)\|+\|\sigma(x,y)\|+\|\partial^{i}_x\partial^j_y\sigma(x,y)\|\leq C, \quad 0\leq i,j\leq 4~ {\text with}~ 1\leq i+j\leq 4.\label{ConA32}
\end{align}

\item[{\rm(iii)}] Assume that
\begin{align}
\inf_{x\in\RR^n, y\in \mathbb{R}^m, z\in \RR^{n}\backslash \{0\}}\frac{\langle(\sigma\sigma^{\ast})(x,y)\cdot z, z\rangle}{|z|^2}>0. \label{NonD}
		\end{align}
 \end{itemize}
\end{conditionB}

\begin{remark}
It is clear that Assumption {\rm\ref{A2}} holds under Assumption {\rm\ref{A3}}(ii), and so Assumption {\rm\ref{A3}} is stronger than Assumption {\rm\ref{A2}}. Stronger smoothness conditions \eqref{ReA11} and \eqref{ConA32} on the coefficients are employed to study regularity estimates of partial derivatives for the averaged coefficients $\bar{b}(t,x)$ and $\bar{\sigma}(t,x)$ in the averaged equation \eqref{AVE3} below with respect to $x$ up to the fourth order. For this purpose, the uniformly non-degeneracy condition \eqref{NonD} is also required. \end{remark}

In this case, the corresponding averaged equation is given by
\begin{equation}
d\bar{X}^{\varepsilon}_{t}=\bar{b}(t/\varepsilon,\bar{X}^{\varepsilon}_t)\,dt+\bar\sigma(t/\varepsilon,\bar{X}^{\varepsilon}_t)\,d \bar{W}_t,\quad\bar{X}^{\varepsilon}_{0}=x, \label{AVE3}
\end{equation}
where $\bar{b}(t,x)$ is defined by \eref{barb}, $$\bar{\sigma}(t,x):=\left[\overline{\sigma\sigma^{\ast}}(t,x)\right]^{1/2}:=\Big(\int_{\RR^{m}}(\sigma\sigma^{\ast})(x,y)\,\mu^x_t(dy)\Big)^{1/2},\quad t\geq 0,$$ and $\{\bar{W_t}\}_{t\ge0}$ is a $n$-dimensional standard  Brownian motion.

Since $\bar{b}(t,x)=\EE b(x,\eta^x_t)$ and $(\overline{\sigma\sigma^{\ast}})(t,x)=\EE(\sigma\sigma^{\ast})(x,\eta^x_t)$, \eref{ConA32} and \eref{FourthDeta} imply that $\bar{b}(t,\cdot)\in C^{4}(\RR^n,\RR^n)$ and $\overline{\sigma\sigma^{\ast}}(t,\cdot)\in C^{4}(\RR^n,\RR^{n\times n})$ satisfy
\begin{align}
\sup_{t\geq 0,x\in\RR^n}\sum^4_{i=1}\big( \|\partial^i_x \bar{b}(t,x) \|+ \|\partial^{i}_x(\overline{\sigma\sigma^{\ast}})(t,x) \|\big)\leq C.\label{FourthDbarb}
\end{align}
This in turn yields that the SDE \eref{AVE3} has a unique solution, which is denoted by $\{\bar{X}^{\varepsilon,x}_t\}_{t\ge 0}$. Let $\{\bar{X}^{\varepsilon,s,x}_t\}_{t\geq s}$ be the solution of \eref{AVE3} with initial value starts from time $s\geq 0$ at $x$. In particular, $\{\bar{X}^{\varepsilon}_t\}_{t\ge 0}=\{\bar{X}^{\varepsilon,0,x}_t\}_{t\ge 0}$ obviously.

Furthermore, by $\eref{NonD}$, $\overline{\sigma\sigma^{\ast}}$ is non-degenerate, i.e.,
$$
\inf_{t\geq 0, x,z\in\RR^n}\frac{\langle \overline{\sigma\sigma^{\ast}}(t,x)\cdot z, z\rangle}{\|z\|^2}>0.
$$
Applying  this and \eqref{FourthDbarb} and following the proof of \cite[Lemma A.7]{CLX}, we have
$$
\sup_{t\geq 0,x\in\RR^n}\sum^4_{i=1}\|\partial^{i}_x\bar{\sigma}(t,x)\|\leq C,$$
which along with \eref{FourthDbarb} again and Assumption \ref{A3}(ii) gives us that $\bar X^{\varepsilon,s,x}_t$ is fourth differentiable in the mean square sense with respect to $x$, and for any $T>0$, there exists $C_T>0$ such that
\begin{align}\label{D barX}
\sup_{0\leq s \leq t\leq T}\sum^{4}_{i=1}\EE \|\partial^i_x \bar X^{\varepsilon,s,x}_t \|^4\leq C_T;
		\end{align}
see the proof of Lemma \ref{DFY}.

Next, for any  $\varphi\in C^{4}_b(\RR^{n})$, consider the following nonautonomous Kolmogorov equation:
\begin{equation}\left\{\begin{array}{l}\label{KE1}
			\displaystyle
			\partial_s u^{\varepsilon}(s,t,x)=-\bar{\mathscr{L}}^{\varepsilon}_s u^{\varepsilon}(s,t,x),\quad s\in[0, t], \\
			u^{\varepsilon}(t,t, x)=\varphi(x),
		\end{array}\right.
\end{equation}
where
$\bar{\mathscr{L}}^{\varepsilon}_s$ is the generator of the averaged equation \eref{AVE3}, i.e.,
\begin{align}
\bar{\mathscr{L}}^{\varepsilon}_s\varphi(x):= \langle \bar b(s/\varepsilon,x), \nabla \varphi(x) \rangle+\frac{1}{2}\text{Tr} [\overline{\sigma\sigma^{\ast}}(s/\varepsilon,x)\nabla^2\varphi(x) ].\label{DebarL}
\end{align}
Hence, the unique solution of \eref{KE1} is given by
$$
u^{\varepsilon}(s,t,x)=\EE\varphi(\bar{X}^{\varepsilon,s,x}_t),\quad t\ge s.
$$
Then, according to $\varphi\in C^{4}_b(\RR^{n})$ and \eref{D barX}, we know that for any $T>0$, there exists a constant $C_T>0$ such that
\begin{equation}
\label{Lemma 5.1-0}\sup_{0\leq s\leq t\leq T, x\in\RR^n}\sum^{4}_{i=1}\|\partial^i_{x}u^{\varepsilon}(s,t,x)\|\leq C_{T}
\end{equation}
 and
\begin{equation}\label{Lemma 5.1}
\sup_{0\leq s\leq t\leq T}\sum^{2}_{i=1}\|\partial_s(\partial^i_x u^{\varepsilon}(s,t,x))\|\leq C_T(1+|x|).
\end{equation}
We note that partial derivatives of $ u^{\varepsilon}(s,t,x)$ with respect to $x$ up to the fourth order are taken into account because regularity condition such like \eref{CenPol} is necessary for analyzing a new Poisson equation, see \eref{WPE1} below. On the other hand, the term $(1+|x|) $ arises from the linear growth of $\bar{b}(t,\cdot)$, see \eqref{linear barb}.

We are in a position to present the second main result in our paper.

\begin{theorem} \label{main result 2}
Suppose that Assumptions {\rm\ref{B1}} and {\rm\ref{A3}} hold. Let $\{(X_t^\varepsilon, Y_t^\varepsilon)\}_{t\ge0}$ and $\{\bar X_t^\varepsilon\}_{t\ge0}$ be the solutions to \eqref{Equation} and \eqref{AVE3} respectively. Then for any $T>0$ and $\varphi\in C^4_b(\RR^n)$, there is a constant $C_{\varphi,T}>0$ so that for all $(x,y)\in\RR^{n}\times\RR^{m}$ and $\varepsilon>0$,
\begin{align}
\sup_{0\leq t\leq T}|\EE \varphi(X^{\varepsilon}_t)-\EE \varphi(\bar{X}^{\varepsilon}_t)|\leq C_{\varphi,T}(1+|x|^2+|y|^2)\varepsilon\sup_{t\in [0,T]}\Lambda_\gamma(t/\varepsilon).\label{AVWO}
\end{align}
\end{theorem}

\begin{proof}
Fix $t\in (0,T]$, and denote by $\tilde{u}^{\varepsilon,t}(s,x)=u^{\varepsilon}(s,t,x)$ for any $s\in [0,t]$. It\^{o}'s formula implies
	\begin{align*}
		\tilde{u}^{\varepsilon,t}(t, X^{\varepsilon}_t)=\tilde{u}^{\varepsilon,t}(0, x)+\int^t_0 \partial_s \tilde{u}^{\varepsilon,t}(s,X^{\varepsilon}_s )\,ds+\int^t_0 \mathscr{L}^{Y^{\varepsilon}_s}_{1}\tilde{u}^{\varepsilon,t}(s, \cdot)(X^{\varepsilon}_s)\,ds+\tilde{M}_t,
	\end{align*}
where
\begin{eqnarray*}
\mathscr{L}^{y}_{1}\varphi(x):=\langle b(x,y), \nabla \varphi(x)\rangle+\frac{1}{2}\text{Tr}\big[({\sigma\sigma})^{\ast}(x,y)\nabla^2\varphi(x)\big],\varphi\in C^2(\RR^n),
\end{eqnarray*}
and $$\tilde{M}_t:=\int^t_0 \langle \partial_{x}\tilde{u}^{\varepsilon,t}(s,X_{s}^{\varepsilon}),\sigma(X_{s}^{\varepsilon},Y_{s}^{\varepsilon})\,dW_{s}^{1} \rangle$$  is a local martingale.

Since $\tilde{u}^{\varepsilon,t}(t, X^{\varepsilon}_t)=\varphi(X^{\varepsilon}_t)$, $\tilde{u}^{\varepsilon,t}(0, x)=\EE\varphi(\bar{X}^{\varepsilon}_t)$, and, thanks to
\eqref{KE1},
$$\partial_s \tilde{u}^{\varepsilon,t}(s, x)=-\bar{\mathscr{L}}^{\varepsilon}_s \tilde{u}^{\varepsilon,t}(s,\cdot)(x),$$
we have	\begin{align*}
 |\EE\varphi(X^{\varepsilon}_{t})-\EE\varphi(\bar{X}^{\varepsilon}_{t}) |
 =&\left|\EE\int^t_0 \mathscr{L}^{Y^{\varepsilon}_s}_{1}(s)\tilde{u}^{\varepsilon,t}(s,\cdot)(X^{\varepsilon}_s)\,ds-\EE\int^t_0 \bar{\mathscr{L}}^{\varepsilon}_s \tilde{u}^{\varepsilon,t}(s, \cdot)(X^{\varepsilon}_s )\,ds\right| \\
=&\bigg|\EE \int^t_0   \bigg(\langle b(X^{\varepsilon}_s,Y^{\varepsilon}_s) - \bar{b}(s/\varepsilon,X^{\varepsilon}_s),\partial_x \tilde{u}^{\varepsilon,t}(s,X^{\varepsilon}_s ) \rangle \\
		& \qquad \quad +\frac{1}{2}\text{Tr}\big[(\sigma\sigma^{\ast}(X^{\varepsilon}_s,Y^{\varepsilon}_s)  -  \overline{\sigma \sigma^{\ast}}(s/\varepsilon,X^{\varepsilon}_s))\partial_x^{2}\tilde{u}^{\varepsilon,t}(s, X^{\varepsilon}_s)\big] \bigg) \,ds\bigg|.
	\end{align*}
	
For $r\in [0,t]$, $s\geq 0$, $x\in\RR^{n}$ and $y\in\RR^m$, define
\begin{align*}
F^t(r,s,x,y):= \langle b(x,y)-\bar{b}(s,x), \partial_x \tilde{u}^{\varepsilon,t}(r,x) \rangle+\frac{1}{2}\text{Tr} [ (\sigma\sigma^{*}(x,y)  -   \overline{\sigma \sigma^{\ast}}(s,x) )
\partial_x^{2}\tilde{u}^{\varepsilon,t}(r,x) ].
	\end{align*}
It is easy to check that $F^t$ satisfies the following centering condition
$$
\int_{\RR^m}F^t(r,s,x,y)\,\mu^x_s(dy)=0,\quad  r\in [0,t], s\geq 0, x\in\RR^{n}
$$
and  \eref{CenPol} as well as
$$\sup_{0\leq r\leq t\leq T,s\geq 0,y\in\RR^m}|\partial_y\partial_rF^t(r,s,x,y)|\leq C_T(1+|x|),$$ thanks to \eqref{Lemma 5.1-0}, \eqref{Lemma 5.1} and \eref{ConA32}.
Let  $\mathscr{L}^{x}_2(s)$ be the operator given by \eqref{L_2}.
Consider the following nonautonomous Poisson equation
	\begin{align}
		\partial_s \tilde{\Phi}^t(r,s,x,y)+\mathscr{L}^{x}_2(s)\tilde{\Phi}^t(r,s,x,y)=-F^t(r,s,x,y).\label{WPE1}
	\end{align}
Then, according to Proposition \ref{P3.6} and its proof, we can claim that \eqref{WPE1} admits a solution
$$\tilde{\Phi}^t(r,s,x,y)=\int^{+\infty}_{s}\EE F^{t}(r,u,x,Y^{s,x,y}_u)\,du,$$
and for any $T>0$, there exists $C_T>0$ such that for any $s\geq 0$, $x\in\RR^{n}$ and $y\in\RR^m$,
	\begin{align*}
	&\sup_{0\leq r\leq t\leq T} \{ |\tilde{\Phi}^t(r,s,x,y) |+ |\partial_r \tilde{\Phi}^t(r,s,x,y) |+ \|\partial_x \tilde{\Phi}^t(r,s,x,y) \|+ \| \partial_x^{2}\tilde{\Phi}^t(r,s,x,y) \| \}\\
	&\leq C_T\Lambda_\gamma(s)(1+|x|^2+|y|^2).
	\end{align*}

Using It\^o's formula, we get
\begin{align*}
&\EE\tilde{\Phi}^t(t,t/\varepsilon,X_{t}^{\varepsilon},Y^{\varepsilon}_{t})\\ &=\tilde{\Phi}^t(0,0,x,y) +  \EE \int^t_0 \partial_r\tilde{\Phi}^t(s,s/\varepsilon,X_{s}^{\varepsilon},Y^{\varepsilon}_s)\,ds + \EE \int^t_0 \mathscr{L}^{Y^{\varepsilon}_{s}}_{1}\tilde{\Phi}^t(s,s/\varepsilon,\cdot,Y_{s}^{\varepsilon})(X^{\varepsilon}_{s})\,ds \\
&\quad+\varepsilon^{-1}\EE\int^t_0(\partial_s\tilde{\Phi}^t(s,s/\varepsilon,X_{s}^{\varepsilon},Y^{\varepsilon}_{s})
+\mathscr{L}^{X_{s}^{\varepsilon}}_2(s/\varepsilon)\tilde{\Phi}^t(s,s/\varepsilon,X_{s}^{\varepsilon},\cdot)(Y^{\varepsilon}_{s}))\,ds,
\end{align*}
which implies that
\begin{align*}
&-\EE\int^t_0[\partial_s\tilde{\Phi}^t(s,s/\varepsilon,X_{s}^{\varepsilon},Y^{\varepsilon}_{s})
+\mathscr{L}^{X_{s}^{\varepsilon}}_2(s/\varepsilon)\tilde{\Phi}^t(s,s/\varepsilon,X_{s}^{\varepsilon},\cdot)(Y^{\varepsilon}_{s})]\,ds\\
&=\varepsilon\Big[\tilde{\Phi}^t(0,0,x,y)  -  \EE\tilde{\Phi}^t(t,t/\varepsilon,X_{t}^{\varepsilon},Y^{\varepsilon}_{t})  +  \EE\int^t_0     \partial_r\tilde{\Phi}^t(s,s/\varepsilon,X_{s}^{\varepsilon},Y^{\varepsilon}_s)\,ds\\
&\quad\quad+\EE\int^t_0     \mathscr{L}^{Y^{\varepsilon}_{s}}_{1}\tilde{\Phi}^t(s,s/\varepsilon, \cdot,Y_{s}^{\varepsilon})(X^{\varepsilon}_{s})\,ds\Big].
\end{align*}

According to all the conclusions above and Lemma \ref{PMY},
\begin{align*}
&\sup_{0\leq t\leq T} |\EE\varphi(X^{\varepsilon}_{t})-\EE\varphi(\bar{X}^{\varepsilon}_{t}) |\\
 & =\sup_{0\leq t\leq T} \left|\EE\int^t_0 F^t(s,s/\varepsilon,X_{s}^{\varepsilon},Y^{\varepsilon}_{s})\,ds \right|\\
 &=\sup_{0\leq t\leq T}\left|\EE\int^t_0 (\partial_s\tilde{\Phi}^t(s,s/\varepsilon,X_{s}^{\varepsilon},Y^{\varepsilon}_{s})
 +\mathscr{L}^{X_{s}^{\varepsilon}}_2(s/\varepsilon)\tilde{\Phi}^t(s,s/\varepsilon,X_{s}^{\varepsilon},\cdot)(Y^{\varepsilon}_{s}))\,ds\right|\\
&\leq \varepsilon\Bigg[\sup_{0\leq t\leq T} |\tilde{\Phi}^t(0,0,x,y) |+\sup_{0\leq t\leq T} |\EE\tilde{\Phi}^t(t,t/\varepsilon,X_{t}^{\varepsilon},Y^{\varepsilon}_{t}) |+\sup_{0\leq t\leq T}\EE\int^t_0 |\partial_r\tilde{\Phi}^t(s,s/\varepsilon,X_{s}^{\varepsilon},Y^{\varepsilon}_s) |\,ds\\
		&\quad\quad+\sup_{0\leq t\leq T}\EE\int^t_0|\mathscr{L}^{Y^{\varepsilon}_{s}}_{1}\tilde{\Phi}^t(s,s/\varepsilon, \cdot, Y^{\varepsilon}_{s})(X_{s}^{\varepsilon})|\,ds\Bigg]\\
&\leq C_T(1+|x|^2+|y|^2)\varepsilon\sup_{t\in [0,T]}\Lambda_\gamma(t/\varepsilon).
	\end{align*}
The proof is complete.
\end{proof}

\begin{remark}
 Let $\alpha(t)=c_0(1+t)^{\beta}$ with $c_0>0$ and $\beta\in (-1,\infty)$. Under the assumptions of Theorem \ref{main result 2}, we know that for any $\varphi\in C^4_b(\RR^n)$,
\begin{equation}\label{WEEquation}
\sup_{0\leq t\leq T}|\EE\varphi(X^{\varepsilon}_{t})-\EE\varphi(\bar{X}^{\varepsilon}_{t})|\leq \left\{\begin{array}{l}
\displaystyle C\varepsilon^{1+\beta},\quad -1<\beta<0, \\
\displaystyle C\varepsilon,\quad\quad\quad\beta\geq 0.
\end{array}\right.
\end{equation}
In particular, if $\beta=0$, then the convergence rate is of the order $1$, which is the same as that for the convergence rate of the weak averaging principle in time-homogeneous settings (see e.g. \cite{B2020}). When $\beta<0$, \eref{WEEquation} shows that the weak convergence order can be slower than the order $1$.
\end{remark}

\section{Time-inhomogeneous multi-scale SDEs: convergent coefficients case }
In the previous section both of the averaged SDEs \eref{AVE} and \eref{AVE3} depend on the scaling parameter $\varepsilon$. One can expect
them to be independent of $\varepsilon$  if we further assume some additional conditions on the coefficients $f,g$ of the fast component $\{Y_t\}_{t\ge0}$ in the stochastic system (\ref{Equation}). In this section, we are restricted ourselves to the case that $f$ and $g$ converge as $t\to +\infty$. Specifically, we suppose that the following condition is satisfied:

\begin{conditionB}\label{A4}
Assume that  $\limsup_{t\to +\infty}\alpha(t)=\alpha>0$, and that there exist functions $\bar{f}:\RR^n\times \RR^m\to \RR^m$ and $\bar{g}:\RR^n\times \RR^m\to \RR^{m\times d_2}$ such that
\begin{align}
|f(t,x,y)-\bar{f}(x,y)|+\|g(t,x,y)-\bar{g}(x,y)\|\leq \phi(t)(1+|x|+|y|),\label{A41}
\end{align}
where $\phi:[0, \infty)\to (0, \infty)$ is locally bounded and satisfies that $\lim_{t\to +\infty}\phi(t)=0$.
\end{conditionB}

In the following, we suppose that Assumptions {\rm\ref{B1}} and {\rm\ref{A4}} hold. According to \eref{A10}, \eqref{A11} and \eref{A41}, for all $x_1,x_2\in \RR^n$ and $y_1,y_2\in \RR^m$,
\begin{align}\label{Disspative}\begin{split}
&2 \langle y_1-y_2,\bar f(x_1,y_1) - \bar f(x_2,y_2)  \rangle  + 3\left\|\bar g (x_1,y_1)-\bar g(x_2,y_2) \right\| ^{2}
 \le  -2\alpha|y_1 - y_2|^{2}+C|x_1-x_2|^2,\end{split}
\end{align}
and for all $x\in \RR^n$ and $y\in \RR^m$,
$$
|\bar{f}(x,y)|+\| \bar g(x,y)\|\leq C(1+|x|+|y|).$$
Thus, the following SDE
\begin{equation}
d\bar Y_{t}=\bar{f}(x,\bar Y_{t})\,dt+\bar{g}(x,\bar Y_{t})\,d W^2_t,\quad \bar Y_{0}=y\in \RR^{m}\label{barFrozen}
\end{equation}
admits a unique strong solution $\{\bar Y_{t}^{x,y}\}_{t\geq 0}$. Furthermore, it can be verified (see the proof of \eqref{FR1}) that for any $t \ge 0$, $x_1,x_2\in\RR^{n}$ and $y_1,y_2\in\RR^{m}$,
\begin{equation}
\mathbb{E} |\bar Y_t^{x_1,y_1}-\bar Y_t^{x_2,y_2} |^2\le Ce^{-2\alpha t}|y_1-y_2|^2+C|x_1-x_2|^2.\label{barflow}
\end{equation}

Let $\{\bar{P}^x_t\}_{t\ge0}$ be the transition semigroup of $\{\bar Y_{t}^{x,y}\}_{t\ge0}$, i.e., for any bounded measurable function $\varphi:\RR^{m}\rightarrow \mathbb{R}$,
$$
\bar P^{x}_{t}\varphi(y)=\EE\varphi(\bar Y_{t}^{x,y}), \quad  y\in\RR^{m}, t\geq 0.
$$
We can prove that $\{\bar P^x_t\}_{t\geq 0}$ admits a unique invariant measure $\mu^x$ such that
	\begin{equation}
		\int_{\RR^m}|y|^2\,\mu^x(dy)\leq C(1+|x|^2).\label{F3.17}
	\end{equation}
This along with \eqref{barflow} yields that for any Lipschitz function $\varphi$ on $\RR^m$,
\begin{equation}
\big|\mathbb{E}\varphi(\bar Y_t^{x,y})-\mu^x(\varphi)\big|\le C{\rm Lip} (\varphi)(1+|x|+|y|) e^{-\alpha t}.   \label{barErgodicity}
\end{equation}

\begin{lemma}\label{Lemma6.2}
Suppose that Assumptions {\rm\ref{B1}} and {\rm\ref{A4}} hold. Then, for any Lipschitz  function $\varphi$ and $\beta\in (0,1)$,
\begin{equation}
 |\mu^x_t(\varphi)-\mu^x(\varphi) |\le C_{\beta}{\rm Lip}(\varphi)(1+|x|)\left\{e^{-\beta\alpha t}+\left[\int^t_0 e^{-2\beta\alpha (t-r)}\phi^2(r)\,dr\right]^{1/2}\right\},   \label{SErgodicity}
\end{equation} where  $\{\mu^x_t\}_{t\in\RR}$ is an evolution system of measures of the semigroup $\{P^x_{s,t}\}_{t\geq s}$ given in Proposition $\ref{P:2.6}$.
\end{lemma}

\begin{proof}
It follows from  \eref{F3.3} and \eref{barFrozen} that
$$
d (\eta^x_t-\bar Y^{x,y}_t )=(f(t,x,\eta^x_t)-\bar f(x,\bar Y^{x,y}_t))\,dt+(g(t,x,\eta^x_t)-\bar g(x,\bar Y^{x,y}_t))\,d W^2_t,
$$
with $\eta^{x}_0-\bar Y^{x,y}_0=\eta^{x}_0-y$.
Then, according to It\^o's formula, \eref{Disspative}, Young's inequality, \eqref{A41} and \eref{F3.2}, for any $\beta\in (0,1)$,
\begin{align*}
\frac{d}{dt}\mathbb{E} |\eta^{x}_t-\bar Y^{x,y}_t |^2
= &2\mathbb{E} \langle \eta^{x}_t - \bar Y^{x,y}_t, f(t,x, \eta^{x}_t) - \bar f(x,\bar Y^{x,y}_t) \rangle  + \mathbb{E}  \|g(t,x,\eta^{x}_t)-\bar g(x, \bar Y^{x,y}_t) \|^2\\
\leq &2\mathbb{E} \langle \eta^{x}_t - \bar Y^{x,y}_t, f(t,x, \eta^{x}_t) - \bar f(x, \eta^{x}_t)\rangle  + 2\mathbb{E}  \|g(t,x,\eta^{x}_t)-\bar g(x,\eta^{x}_t) \|^2\\
&+2\mathbb{E} \langle \eta^{x}_t - \bar Y^{x,y}_t, \bar f(x, \eta^{x}_t) - \bar f(x,\bar Y^{x,y}_t) \rangle  + 2\mathbb{E}  \|\bar g(x,\eta^{x}_t)-\bar g(x, \bar Y^{x,y}_t) \|^2\\
\leq &-2\beta\alpha\mathbb{E} |\eta^{x}_t-\bar Y^{x,y}_t |^{2}
+C\phi^2(t) (1+|x|^{2}+\EE|\eta^{x}_t|^{2} )\\
\leq &-2\beta\alpha\mathbb{E} |\eta^{x}_t-\bar Y^{x,y}_t |^{2}
+C \phi^2(t)(1+|x|^{2}).
\end{align*}
This along with \eqref{F3.2} again implies that
\begin{align*}
\mathbb{E}|\eta^{x}_t-\bar Y^{x,y}_t|^{2}\le & e^{-2\beta\alpha t}\EE|\eta^{x}_0-y|^{2}+C_{\beta}(1+|x|^{2})\int^t_0 e^{-2\beta\alpha (t-r)}\phi^2(r)\,dr \\
\le & C(1+|x|^{2}+|y|^2)e^{-2\beta\alpha t}+C_{\beta}(1+|x|^{2})\int^t_0 e^{-2\beta\alpha (t-r)}\phi^2(r)\,dr.
\end{align*}

Combining this with \eref{barErgodicity}, we get that for any Lipschitz function $\varphi$,
\begin{align*}
|\mu^x_t(\varphi)-\mu^x(\varphi)|\leq &  |\EE(\varphi(\eta^{x}_t)-\varphi(\bar Y^{x,0}_t)) |+ |\EE\varphi(\bar Y^{x,0}_t)-\mu^x(\varphi) |\\
\leq& C{\rm Lip}(\varphi)(1+|x|)e^{-\beta\alpha t}+C_{\beta}{\rm Lip}(\varphi)(1+|x|)\left[\int^t_0 e^{-2\beta\alpha (t-r)}\phi^2(r)\,dr\right]^{1/2}.
\end{align*}
The proof is complete.
\end{proof}
\begin{remark}
We note that for any $\eta>0$,
\begin{align}
\lim_{t\rightarrow +\infty} \int^{t}_0 e^{-\eta (t-r)}\phi^2(r)\,dr=0.\label{Re4.11}
\end{align}
Indeed, for any $\delta>0$, it follows from the fact  $\lim_{t\to +\infty}\phi(t)=0$ that there exists $N_1>0$ so that $\phi^2(t)\leq \eta\delta/2$ for all $t\geq N_1.$
On the other hand, since $\phi$ is locally bounded, there exists $N_2>0$ such that for all $t\geq N_2$,
$$e^{-\eta t}\int^{N_1}_0 e^{\eta r}\phi^2(r)\,dr\leq \delta/2.$$
Then, for any $t\geq N_1\vee N_2$,
\begin{align*}
\int^{t}_0 e^{-\eta(t-r)}\phi^2(r)\,dr=e^{-\eta t}\int^{N_1}_0 e^{\eta r}\phi^2(r)\,dr+e^{-\eta t}\int^{t}_{N_1} e^{\eta r}\phi^2(r)\,dr
\leq \delta/2+\delta/2=\delta.
\end{align*}
In particular,
\begin{align*}
\lim_{t\to +\infty}\int^{t}_0 e^{-\eta(t-r)}\phi^2(r)\,dr=0,
\end{align*}
which proves the desired assertion \eref{Re4.11}.
\end{remark}

\subsection{Strong averaging principle: convergent coefficients case}

In this part, we suppose that $f$ and $g$ fulfill Assumptions {\rm\ref{B1}} and {\rm\ref{A4}}, and $b$ and $\sigma$ satisfy Assumption {\rm\ref{A2}} such that $\sigma(x,y)=\sigma(x)$ for all $x\in \RR^n$ and $y\in \RR^m$. Consider the averaged SDE:
\begin{equation}
d\bar{X}_{t}=\bar b_c(\bar{X}_t)\,dt+\sigma(\bar{X}_t)\,d W^1_t,\quad\bar{X}_{0}=x, \label{AVE21}
\end{equation}
where
\begin{eqnarray}
\bar{b}_c(x)=\displaystyle\int_{\RR^{m}}b(x,y)\,\mu^x(dy) \label{barb21}
\end{eqnarray}
and $\mu^x$ is the unique invariant measure of the SDE \eref{barFrozen}.

\begin{lemma} Suppose that Assumptions {\rm\ref{B1}}, {\rm\ref{A2}} and {\rm\ref{A4}} hold. Then, for $x\in\RR^{n}$, \eqref{AVE21} has a unique solution $\{\bar{X}\}_{t\ge0}$; moreover, for $T>0$, there exists a constant $C_{T}>0$ such that
\begin{equation}
\mathbb{E}\Big(\sup_{t\in [0, T]}|\bar{X}_{t}|^{2}\Big)\leq C_{T}(1+|x|^{2}).\label{FianlEE}
\end{equation}
\end{lemma}

\begin{proof}
Note that
for $x_1,x_2\in\RR^n$, by \eref{A21} and \eref{barErgodicity}, one has
\begin{align*}
 |\bar{b}_c(x_{1})-\bar{b}_c(x_{2}) |
= & \left|\int_{\RR^m}b(x_1,y)\,\mu^{x_1}(dy)-\int_{\RR^m}b(x_2,y)\,\mu^{x_2}(dy)\right|\\
\leq &\left|\int_{\RR^m}b(x_1,y)\,\mu^{x_1}(dy)-\EE  b(x_{1},\bar Y^{x_1,0}_t)\right|+ \left|\int_{\RR^m}b(x_2,y)\,\mu^{x_2}(dy)-\EE b(x_{2},\bar Y^{x_2,0}_t) \right|\\
 &+ |\EE  b(x_{1},\bar Y^{x_1,0}_t)-\EE  b(x_{2},\bar Y^{x_2,0}_t) |\\
\le& C(1+|x_1|+|x_2|)
e^{-\alpha t}+C|x_1-x_2|+C\EE |\bar Y^{x_1,0}_t-\bar Y^{x_2,0}_t |.
\end{align*}
Using \eref{barflow} and letting $t\to +\infty$, we obtain that
\begin{align}
|\bar{b}_c(x_{1})-\bar{b}_c(x_{2})| \leq C|x_{1}-x_{2}|.\label{Lipbarb}
\end{align}
Therefore, \eref{AVE21} admits a unique strong solution, and \eref{FianlEE} holds obviously.  		
\end{proof}

Now, we state our third main result.

\begin{theorem}\label{main result 3}
Suppose that Assumptions {\rm\ref{B1}}, {\rm\ref{A2}} and {\rm\ref{A4}} hold, and $\sigma(x,y)=\sigma(x)$. Let $\{X_t^\varepsilon, Y_t^\varepsilon \}_{t\ge0}$ and $\{\bar X_t\}_{t\ge0}$ be the solutions of \eqref{Equation} and \eqref{AVE21} respectively. Then, for any $T>0$ and $\beta\in(0,1)$, there exists a constant $C_{T,\beta}>0$ so that for $\varepsilon>0$,
\begin{align*}
\sup_{t\in [0, T]}\EE |X_{t}^{\varepsilon}-\bar{X}_{t} |^2\leq  C_{T,\beta}(1 + |x|^2 + |y|^2)\varepsilon^2\Bigg\{&\left[\int_0^{T/\varepsilon}\left(\int^{s}_0 e^{-2\beta\alpha (s-r)}\phi^2(r)\,dr\right)^{1/2}\,ds\right]^2\\
&  +\sup_{0\leq t\leq T}|\Lambda_{\gamma}(t/\varepsilon)|^2  +  \int^{T/\varepsilon}_0     \alpha(s)\Lambda^2(s)\,ds\Bigg\}.
\end{align*}
 \end{theorem}

\begin{proof}
According  to Assumption \ref{B1} and Lemma \ref{Lemma6.2}, for $\beta\in (0,1)$, there exists $C_{\beta}>0$ such that
\begin{equation}\label{barbtbarb}\begin{split}
 |\bar{b}(t,x)-\bar{b}_c(x) |
\leq & \left|\int_{\RR^m}b(x,y)\,\mu^x_t(dy)-\int_{\RR^m}b(x,y)\,\mu^x(dy) \right|  \\
\leq &C_{\beta}(1+|x|)\left[e^{-\beta\alpha t}+\left(\int^t_0 e^{-2\beta\alpha (t-r)}\phi^2(r)\,dr\right)^{1/2}\right].
\end{split}\end{equation}
On the other hand, let $\{\bar X_t^\varepsilon\}_{t\ge0}$ be the solution to the SDE \eqref{AVE}. Then,
by \eref{Lipbarb} and \eref{A21}, for any $t\in [0,T]$,
$$
\bar X_{t}^{\varepsilon}-\bar{X}_{t}= \int_{0}^{t}(\bar{b}(s/\varepsilon,\bar{X}^{\varepsilon}_{s})-\bar{b}_c(\bar{X}_{s}))\,ds
+\int_{0}^{t}(\sigma(\bar{X}^{\varepsilon}_{s})-\sigma(\bar{X}_{s}))\,dW^1_s,
$$
which implies that for all $t\in [0,T]$,
\begin{align*}
\sup_{s\in [0,t]}\EE |\bar X_s^\varepsilon-\bar{X}_{s} |^{2}\leq &C\EE\Big(\int_0^t |\bar{b}(s/\varepsilon,\bar{X}^{\varepsilon}_{s})-\bar{b}_c(\bar{X}^{\varepsilon}_{s}) |\, ds\Big)^2+C_T\EE\int_0^t |\bar{b}_c(\bar{X}^{\varepsilon}_{s})-\bar{b}_c(\bar{X}_{s}) |^2\,ds\\
&+C\int_0^t \EE \|\sigma(\bar{X}^{\varepsilon}_s)-\sigma(\bar{X}_s) \|^2\,ds\\
\leq&C\EE\Big(\int_0^t |\bar{b}(s/\varepsilon,\bar{X}^{\varepsilon}_{s})-\bar{b}_c(\bar{X}^{\varepsilon}_{s}) |\,ds\Big)^2+C_T\int_0^t\EE |\bar X_s^\varepsilon-\bar{X}_{s} |^{2}\,d s.
\end{align*}

Hence, according to \eref{barbtbarb}, \eref{EAVE} and the Gronwall inequality, for any $\beta\in (0,1)$,
\begin{align*}
\sup_{t\in [0,T]}\EE |\bar X_t^\varepsilon-\bar{X}_{t} |^2 &\leq  C_T\EE\left(\int_0^T|\bar{b}(s/\varepsilon,\bar{X}^{\varepsilon}_{s})-\bar{b}_c(\bar{X}^{\varepsilon}_{s})|\,ds\right)^2\\
&\leq  C_{T,\beta}(1+|x|^{2}+|y|^{2})\left[\int_0^T \left(e^{-{\beta\alpha s}/{\varepsilon}}+\left(\int^{s/\varepsilon}_0 e^{-2\beta\alpha (s/\varepsilon-r)}\phi^2(r)\,dr\right)^{1/2}\right)\,ds \right]^2 \\
&\leq  C_{T,\beta}\varepsilon^2(1+|x|^{2}+|y|^{2})\left\{1+\int_0^{T/\varepsilon}\left[\left(\int^{s}_0 e^{-2\beta\alpha (s-r)}\phi^2(r)\,dr\right)^{1/2}\right]\,ds \right\}^2.
\end{align*}
Therefore, the desired assertion immediately follows from the estimate above and \eref{R1}.
\end{proof}

\subsection{Weak averaging principle: convergent coefficients case}

In this subsection, suppose that $f$ and $g$ fulfill Assumptions {\rm\ref{B1}} and {\rm\ref{A4}}, and that $b$ and $\sigma$ satisfy Assumption {\rm\ref{A3}}. Consider the averaged SDE:
\begin{equation}
d\bar{X}_{t}=\bar{b}_c(\bar{X}_t)\,dt+\bar{\sigma}_c(\bar{X}_t)\,d \bar W_t,\quad\bar{X}_{0}=x, \label{AVE22}
\end{equation}
where $\bar{b}_c(x)$ is defined in \eref{barb21}, $$\bar{\sigma}_c(x):=\left[\left(\overline{\sigma\sigma^{\ast}}\right)_{c}(x)\right]^{1/2}:=\Big[\int_{\RR^{m}}\left(\sigma\sigma^{\ast}\right)(x,y)\,\mu^x(dy)\Big]^{1/2},$$ and $\{\bar{W_t}\}_{t\ge0}$ is a $n$-dimensional standard Brownian motion.

Following the proof of \eqref{FourthDbarb}, we can obtain
\begin{align}
\sup_{t\in\RR,x\in\RR^n}\sum^4_{i=1}[\|\partial^i_x \bar{b}_c(x)\|+\|\partial^{i}_x\bar{\sigma}_c(x)\|]\leq C,\label{FourthDbarsigmA1}
\end{align}
and so the SDE \eref{AVE22} has a unique solution $\{\bar X^{x}_t\}_{t\ge 0}$. Moreover, we can verify that for any $t\ge0$, $\bar X^{x}_t$ is fourth differentiable in the mean square sense with respect to $x$, and, for any $T>0$, there exists $C_T>0$ such that
\begin{align}\label{D barX2}
\sup_{t\in [0,T]}\sum^{4}_{i=1}\EE \|\partial^i_x \bar X^{x}_t\|^4\leq C_T.
		\end{align}
For any $\varphi\in C^{4}_b(\RR^{n})$, consider the Kolmogorov equation:
\begin{equation}
\label{KE2}
\displaystyle
			\partial_t u(t,x)=\bar{\mathscr{L}}_c u(t,x), \quad u(0, x)=\varphi(x),
		 \end{equation}
where $\bar{\mathscr{L}}_c$ is the generator of \eref{AVE22}, i.e.,
$$
\bar{\mathscr{L}}_c\varphi(x):= \langle \bar{b}_c(x), \nabla \varphi(x) \rangle+\frac{1}{2}\text{Tr} [\left(\overline{\sigma\sigma^{\ast}}\right)_c(x)\nabla^2\varphi(x)].
$$
It is easy to see that \eref{KE2} has a unique solution, which is given by $u(t,x)=\EE\varphi(\bar{X}^{x}_t)$ for all $t\ge 0.$ According to $\varphi\in C^{4}_b(\RR^{n})$ and \eref{D barX2}, for any $T>0$,
\begin{align}
\sup_{0\leq t\leq T, x\in\RR^n}\sum^{4}_{i=1}\|\partial^i_{x}u(t,x)\|\leq C_{T}.\label{F4.18}
\end{align}

Now, we state our fourth main result in this paper.

\begin{theorem} \label{main result 4}
Suppose that Assumptions {\rm\ref{B1}}, {\rm\ref{A3}} and {\rm\ref{A4}} hold. Let $\{(X_t^\varepsilon, Y_t^\varepsilon) \}_{t\ge0}$ and $\{\bar X_t\}_{t\ge0}$ be the solutions to \eqref{Equation} and \eqref{AVE22} respectively. Then for any $T>0$, $\beta\in(0,1)$ and $\varphi\in C^4_b(\RR^n)$, there is a constant $C_{\varphi,T,\beta}>0$ such that for all $\varepsilon>0$,
\begin{equation}\label{AVWO2}\begin{split}
&\sup_{0\leq t\leq T}|\EE \varphi(X^{\varepsilon}_t)-\EE \varphi(\bar{X}_t)|\\
&\leq C_{\varphi,T,\beta}\left(1+|x|+|y|\right)\varepsilon\Big\{\int_0^{T/\varepsilon}\Big[\Big(\int^{s}_0 e^{-2\beta\alpha (s-r)}\phi^2(r)\,dr\Big)^{1/2}\Big]\,ds+\sup_{t\in [0,T]}\Lambda_\gamma(t/\varepsilon)\Big\}.
\end{split}\end{equation}
\end{theorem}

\begin{proof}
Following the proof of \eref{barbtbarb}, we can  obtain
\begin{align}
\left\|\left(\overline{\sigma \sigma^{\ast}}\right)(t,x)-\left(\overline{\sigma \sigma^{\ast}}\right)_c(x)\right\|
\leq C_{\beta}(1+|x|)\Big[e^{-\beta\alpha t}+\Big(\int^t_0 e^{-2\beta\alpha (t-r)}\phi^2(r)\,dr\Big)^{1/2}\Big].\label{barsigamtbarsigma}
\end{align}

Fix $t\le T$ and define $\tilde{u}^{t}(s,x)=u(t-s,x)$ for $0\le s\le t$. Let $\{\bar{X}^{\varepsilon}_t\}_{t\ge0}$ be the solution to the SDE \eqref{AVE3}. It\^{o}'s formula implies
	\begin{align*}
		\tilde{u}^{t}(t, \bar{X}^{\varepsilon}_t)=\tilde{u}^{t}(0, x)+\int^t_0 (\partial_s\tilde{u}^{t}(s, \bar{X}^{\varepsilon}_s)+\bar{\mathscr{L}}^{\varepsilon}_s\tilde{u}^{t}(s, \cdot)(\bar X^{\varepsilon}_s) )\,ds+\tilde{M}'_t,
	\end{align*}
where $\bar{\mathscr{L}}^{\varepsilon}_s$ is defined by \eref{DebarL} and $$\tilde{M}'_t:=\int^t_0 \langle \partial_{x}\tilde{u}^{t}(s,\bar X_{s}^{\varepsilon}),\bar\sigma(s/\varepsilon,\bar X_{s}^{\varepsilon})\,d\bar W_{s} \rangle$$  is a local martingale.

Note that $\tilde{u}^{t}(t, \bar X^{\varepsilon}_t)=\varphi(\bar X^{\varepsilon}_t)$, $\tilde{u}^{t}(0, x)=\EE\varphi(\bar{X}_t)$, and, by \eqref{KE2},
$$\partial_s \tilde{u}^{t}(s, x)=-\bar{\mathscr{L}}_c\tilde{u}^{t}(s,\cdot)(x).$$
\eref{barbtbarb}, \eref{barsigamtbarsigma} and \eref{F4.18} yield that
\begin{equation*}
\begin{split}
\sup_{t\in [0,T]} |\EE\varphi(\bar X^{\varepsilon}_{t})-\EE\varphi(\bar{X}_{t}) |
 =&\sup_{t\in [0,T]}\left|\EE\int^t_0 -\bar{\mathscr{L}}_c \tilde{u}^{t}(s, \cdot)(\bar X^{\varepsilon}_s )\,ds+\EE\int^t_0 \bar{\mathscr{L}}^{\varepsilon}_s\tilde{u}^{t}(s,\cdot)(\bar X^{\varepsilon}_s)\,ds\right| \\
=&\sup_{t\in [0,T]}\Big|\EE \int^t_0 \Big( \langle \bar b(s/\varepsilon,\bar X^{\varepsilon}_s) - \bar{b}_c(\bar X^{\varepsilon}_s),\partial_x \tilde{u}^{t}(s,\bar X^{\varepsilon}_s ) \rangle \\
		&\qquad\qquad +\frac{1}{2}\text{Tr}\big[((\overline{\sigma \sigma^{\ast}})(s/\varepsilon,\bar X^{\varepsilon}_s) - (\overline{\sigma\sigma^{\ast}})_c(\bar X^{\varepsilon}_s))\partial_x^{2}\tilde{u}^{t}(s, \bar X^{\varepsilon}_s)\big] \Big) \,ds\Big| \\
\leq & C_{T,\beta}(1+|x|+|y|) \int_0^T \Big[e^{-{\beta\alpha s}/{\varepsilon}} + \Big( \int^{s/\varepsilon}_0 e^{-2\beta\alpha (s/\varepsilon-r)}\phi^2(r)\,dr\Big)^{1/2}\Big] \,ds \\
\leq & C_{T,\beta}(1+|x|+|y|)\varepsilon\Big\{1+\int_0^{T/\varepsilon}\Big(\int^{s}_0 e^{-2\beta\alpha (s-r)}\phi^2(r)\,dr\Big)^{1/2}\,ds \Big\}.
	\end{split}\end{equation*}
Hence, the desired assertion immediately follows from the estimate above and \eref{AVWO}.
\end{proof}

\section{Time-inhomogeneous multi-scale SDEs: periodic coefficients case}

To ensure that the averaged SDE does not depend on the parameter $\varepsilon$, in this section we concentrate on the case that $f$ and $g$ are periodic. Such kind condition differs significantly from Assumption \ref{A4}. Specifically, we will impose the following condition:

\begin{conditionB}\label{A5}
Assume that $f(\cdot,x,y)$ and $g(\cdot,x,y)$ are $\tau$-periodic for some $\tau>0$, that is, for any $t\in \RR$, $x\in \RR^n$ and $y\in \RR^m$,
\begin{align}
f(t+\tau,x,y)=f(t,x,y),\quad g(t+\tau,x,y)=g(t,x,y).\label{A51}
\end{align}
\end{conditionB}

We start with two simple lemmas.

\begin{lemma}\label{P5.1}
Let $h$ be a  bounded and $\tau$-periodic function on $\RR$. Then, for any $T>0$,
\begin{align*}
\sup_{a\in\RR}\left|\frac{1}{T}\int^{T+a}_a h(s)\,ds-\frac{1}{\tau}\int^{\tau}_0h(s)\,ds\right|\leq \frac{2\tau M}{T},
\end{align*}
where $M=\sup_{s\in [0,\tau]}|h(s)|$.
\end{lemma}
\begin{proof}
Note that for any $T>0$, $\left|\left[\frac{T}{\tau}\right]/T - 1/\tau \right|\leq 1/T.$
Since $h$ is $\tau$-periodic, for any $a\in\RR$,
\begin{align*}
\left| \frac{1}{T} \int_{a}^{a+T} h(s) \, ds - \frac{1}{\tau} \int_{0}^{\tau} h(s) \, ds \right|
& \leq\left| \frac{\left[\frac{T}{\tau}\right]}{T} \int_{0}^{\tau} h(s) \, ds - \frac{1}{\tau} \int_{0}^{\tau} h(s) \, ds \right| + \frac{1}{T} \int_{a + \lfloor \frac{T}{\tau} \rfloor \tau}^{a+T} |h(s)| \, ds \\
&\leq  \frac{\tau M}{T} + \frac{\tau M}{T} =\frac{2\tau M}{T}.
\end{align*}
The proof is complete.
\end{proof}

\begin{lemma}\label{Le5.2}
Suppose that Assumptions {\rm\ref{B1}} and {\rm\ref{A5}} hold. Then, the evolution system of measures $\{\mu^x_t\}_{t\in\RR}$ given in Proposition $\ref{P:2.6}$ is $\tau$-periodic, i.e., for any $t\in \RR$,
$\mu^x_{t+\tau}=\mu^x_t$.
\end{lemma}

\begin{proof} The proof follows from that of \cite[Theorem 4.1]{DT1995}. Let $\varphi$ be a continuous and bounded function.
Consider the following the nonautonomous Kolmogorov equation
$$
			\partial_s v(s,t,y)=-\mathscr{L}^x_2(s) v(s,t,y),\,\,v(t,t, y)=\varphi(y),\quad s\in[0, t],
$$
where $\mathscr{L}^x_2(s)$ is defined in \eref{L_2}. It is easy to see that $v(s,t,y)=P^x_{s,t}\varphi(y)$.
Since $f$ and $g$ are $\tau$-periodic, $P^x_{s+\tau,t+\tau}\varphi(y)=P^x_{s,t}\varphi(y)$ for $t\geq s$.
Hence, noting that the law of $Y^{s,x,0}_t$ convergence weakly to $\mu^x_t$ by \eref{F3.2},
\begin{align*}
\int_{\RR^m}\varphi(y)\mu^x_{t+\tau}(dy)=\lim_{s\to -\infty}P^x_{s+\tau, t+\tau}\varphi(0)
=\lim_{s\to -\infty}P^x_{s,t}\varphi(0)=\int_{\RR^m}\varphi(y)\mu^x_{t}(dy).
\end{align*}
Hence, $\{\mu^x_t\}_{t\in\RR}$ is $\tau$-periodic.
\end{proof}

\subsection{Strong averaging principle: periodic coefficients case}
In this subsection, we suppose that $f$ and $g$ fulfill Assumptions {\rm\ref{B1}} and {\rm\ref{A5}}, and $b$ and $\sigma=\sigma(x)$ fulfill Assumption {\rm\ref{A2}}. Consider the following averaged SDE:
\begin{equation}
d\bar{X}_{t}=\bar{b}_p(\bar{X}_t)\,dt+\sigma(\bar{X}_t)\,d W^1_t,\quad\bar{X}_{0}=x, \label{AVE31}
\end{equation}
where
\begin{eqnarray}
\bar{b}_p(x):=\frac{1}{\tau}\int^{\tau}_0\bar{b}(t,x)\,dt=\displaystyle\frac{1}{\tau}\int^{\tau}_0\int_{\RR^{m}}b(x,y)\,\mu^x_t(dy)\,dt.\label{barb31}
\end{eqnarray}

By \eref{barc1}, it is easy to see that for any $x_1,x_2\in \RR^n$,
\begin{align}
|\bar{b}_p(x_1)-\bar{b}_p(x_2)|\leq C|x_1-x_2|,\label{Lipsbarb}
\end{align}
and so the SDE \eref{AVE31} admits a unique solution $\{\bar{X}_t\}_{t\ge0}$. Moreover, by \eqref{Lipsbarb} and Assumption {\rm\ref{A2}}, for any $T>0$, there exists $C_T>0$ such that
 \begin{align}
\sup_{0\leq s\leq t\leq T}\EE|\bar{X}_{t}-\bar{X}_{s}|^2\leq C_T(1+|x|^2)|t-s|.\label{F5.6}
\end{align}

\begin{theorem}\label{main result 5}
Suppose that Assumptions {\rm\ref{B1}}, {\rm\ref{A2}} and {\rm\ref{A5}} hold. Let $\{(X_t^\varepsilon, Y_t^\varepsilon) \}_{t\ge0}$ and $\{\bar X_t\}_{t\ge0}$ be the solutions to the SDEs  \eqref{Equation} and \eqref{AVE31} respectively. Then, for any $T>0$, there exists $C_T>0$ such that for any $(x,y)\in\RR^n\times\RR^m$ and $\varepsilon>0$,
$$
\sup_{t\in [0, T]}\EE |X_{t}^{\varepsilon}-\bar{X}_{t} |^2\leq  C_{T}\varepsilon^2(1  +  |x|^{4}  +  |y|^{4})\Big[\varepsilon^{-4/3}  +  \sup_{0\leq t\leq T}|\Lambda_{\gamma}(t/\varepsilon)|^2  +  \int^{T/\varepsilon}_0 \alpha(s)\Lambda^2(s)\,ds\Big]. $$
\end{theorem}

\begin{proof} Let $\{\bar X_t^\varepsilon\}_{t\ge0}$ be the solution to the SDE \eqref{AVE}. Then, $$
\bar X_{t}^{\varepsilon}-\bar{X}_{t}= \int_{0}^{t} [\bar{b}(s/\varepsilon,\bar{X}^{\varepsilon}_{s})-\bar{b}_p(\bar{X}_{s}) ]\,ds
+\int_{0}^{t}[\sigma(\bar{X}^{\varepsilon}_{s})-\sigma(\bar{X}_{s}) ]\,dW^1_s.
$$
According to \eref{barc1} and \eref{A21}, for $t\in [0,T]$,
\begin{align*}
\sup_{s\in [0,t]}\EE|\bar X_s^\varepsilon-\bar{X}_{s}|^{2}\leq &C_T\EE\int_0^t|\bar{b}(s/\varepsilon,\bar{X}^{\varepsilon}_{s})  -  \bar{b}(s/\varepsilon,\bar{X}_{s})|^2\,ds  +  C_T\int_0^t \EE\|\sigma(\bar{X}^{\varepsilon}_s)-\sigma(\bar{X}_s)\|^2\,ds \\
&+C_T\EE\left|\int_0^t\bar{b}(s/\varepsilon,\bar{X}_{s})-\bar{b}_p(\bar{X}_{s}) \,ds\right|^2\\
\leq& C_T\int_0^t\EE|\bar X_s^\varepsilon-\bar{X}_{s}|^{2}\,ds+C_T\EE\left|\int_0^t\bar{b}(s/\varepsilon,\bar{X}_{s})-\bar{b}_p(\bar{X}_{s})\,ds\right|^2.
\end{align*}
Thus, the Gronwall inequality, \eref{barbtbarb} and \eref{EAVE} yield that
\begin{align*}
&\sup_{t\in [0,T]}\EE |\bar{X}_t^\varepsilon-\bar{X}_{t} |^{2} \\
&\leq  C_T\EE\Big|\int_0^T[\bar{b}(s/\varepsilon,\bar{X}_{s})-\bar{b}_p(\bar{X}_{s})]\,ds\Big|^2 \\
&\leq  C_T\EE\Big|\int_0^T[\bar{b}(s/\varepsilon,\bar{X}_{s(\delta)})-\bar{b}_p(\bar{X}_{s(\delta)})]\,ds\Big|^2+C_T\EE\int_0^T|\bar{b}(s/\varepsilon,\bar{X}_{s})-\bar{b}(s/\varepsilon,\bar{X}_{s(\delta)})|^2\,ds\nonumber\\
&\quad+C_T\EE\int_0^T|\bar{b}_p(\bar{X}_{s})-\bar{b}_p(\bar{X}_{s(\delta)})|^2\,ds\\
&=: \sum^3_{i=1}J_i(T),\end{align*}
where $s(\delta)=\left[{s}/{\delta}\right]\delta$ and $[s/{\delta}]$ is the integer part of $s/{\delta}$. Here, $\delta>0$ depends on $\varepsilon$, which will be chosen later.

According to \eref{barc1}, \eref{Lipsbarb} and \eref{F5.6},
\begin{align}
J_2(T)+J_3(T)\leq C_T(1+|x|^2)\delta.
\end{align}
On the other hand, since $\bar{b}(\cdot,x)$ is $\tau$-periodic by Lemma \ref{Le5.2}, Proposition \ref{P5.1}, \eqref{barc1} and \eqref{Lipsbarb} yield that
\begin{align*}
J_1(T)\leq& C_T\EE\left|\sum^{[T/\delta]-1}_{k=0}    \int^{(k+1)\delta}_{k\delta}    [\bar{b}(s/\varepsilon,\bar{X}_{k\delta})  -  \bar{b}_p(\bar{X}_{k\delta})]  \,ds
\right|^2  +  C_T\EE\Big|\int^{T}_{[T/\delta]\delta}  [\bar{b}(s/\varepsilon,\bar{X}_{k\delta})  -  \bar{b}_p(\bar{X}_{k\delta})]  \,ds\Big|^2\\
\leq& C_T\EE\left|\sum^{[T/\delta]-1}_{k=0}\delta\left(\frac{\varepsilon}{\delta}\int^{\frac{k\delta}{\varepsilon}+\frac{\delta}
{\varepsilon}}_{\frac{k\delta}{\varepsilon}}\bar{b}(s,\bar{X}_{k\delta})\,ds-\bar{b}_p(\bar{X}_{k\delta})\right)\right|^2+C_T(1+|x|^2)\delta^2\\
\leq& C_T\varepsilon^2\EE\left|\sum^{[T/\delta]-1}_{k=0}(1+|\bar{X}_{k\delta}|)\right|^2+C_T(1+|x|^2)\delta^2\\
\leq& C_T(1+|x|^2)\frac{\varepsilon^2}{\delta^2}+C_T(1+|x|^2)\delta^2.\end{align*}
Combining all the estimates above and taking $\delta=\varepsilon^{2/3}$, we get
\begin{align*}
\sup_{t\in [0,T]}\EE|\bar{X}_t^\varepsilon-\bar{X}_{t}|^{2} \leq& C_T(1+|x|^2)\varepsilon^{2/3}.
\end{align*}
Therefore, the desired assertion immediately follows from the inequality above and \eref{R1}.
\end{proof}

\begin{remark} Assume that Assumptions \ref{B1}, \ref{A2} and \ref{A5} hold with the rate function $\alpha$ being a positive constant. Then the functions $\Lambda$ and $\Lambda_{\gamma}$ are constants. Additionally, if $\sigma\equiv 0$, then we can improve the estimate \eref{F5.6} into
\begin{align*}
\sup_{0\leq s\leq t\leq T}\EE |\bar{X}_{t}-\bar{X}_{s} |\leq C_T(1+|x|+|y|)|t-s|.
\end{align*}
Following the proof of Theorem \ref{main result 5}, we can get
\begin{equation*}
\sup_{t\in [0, T]}\EE |X_t^{\varepsilon}-\bar{X}_t |^2\leq  C_{T}(1+|x|^{4}+|y|^{4})\varepsilon.
\end{equation*}
In contrast to \cite[Theorem 2.3]{W2013}, which claims that
$$\EE\left[\sup_{t\in [0, T]}|X_{t}^{\varepsilon}-\bar{X}_{t}|^2\right]=0,$$
here we can get the optimal strong convergence rate $1/2$ with the supremum being taken outside of the expectation.
\end{remark}

\subsection{Weak averaging principle: periodic coefficients case}

In this subsection, we suppose that $f$ and $g$ satisfy Assumptions {\rm\ref{B1}} and {\rm\ref{A5}}, and $b$ and $\sigma$ satisfy Assumption {\rm\ref{A3}}. We consider the averaged SDE
\begin{equation}
d\bar{X}_{t}=\bar{b}_p(\bar{X}_t)\,dt+\bar{\sigma}_p(\bar{X}_t)\,d \bar W_t,\quad\bar{X}_{0}=x, \label{AVE32}
\end{equation}
where $\bar{b}_p(x)$ is defined in \eref{barb31}, $$\bar{\sigma}_p(x):=\left[\left(\overline{\sigma\sigma^{\ast}}\right)_p(x)\right]^{1/2}:=\left[\frac{1}{\tau}\int^{\tau}_0\int_{\RR^{m}}\left(\sigma\sigma^{\ast}\right)(x,y)\,\mu^x_t(dy)\,dt\right]^{1/2},$$
and $\{\bar{W_t}\}_{t\ge0}$ is a $n$-dimensional Brownian motion.

Similar to \eqref{FourthDbarsigmA1}, now it holds that
\begin{align}
\sup_{t\in\RR,x\in\RR^n}\sum^4_{i=1}[\|\partial^i_x \bar{b}_p(x)\|+\|\partial^{i}_x\bar{\sigma}_p(x)\|]\leq C.\label{FourthDbarsigma3}
\end{align}
Hence, the SDE \eref{AVE32} admits a unique solution $\{\bar X^{x}_t\}_{t\geq 0}$.

For any $\varphi\in C^{4}_b(\RR^{n})$, consider the following Kolmogorov equation:
\begin{equation}
\label{KE3}
			\displaystyle
			\partial_t u(t,x)=\bar{\mathscr{L}}_p u(t,x),\quad
			u(0, x)=\varphi(x),
	\end{equation}
where  $\bar{\mathscr{L}}_p$ is the generator of the SDE \eref{AVE32} and is given by
$$
\bar{\mathscr{L}}_p\varphi(x):= \langle \bar{b}_{p}(x), \nabla \varphi(x) \rangle+\frac12\text{Tr} [\left(\overline{\sigma\sigma^{\ast}}\right)_{p}(x)\nabla^2\varphi(x)].
$$
We can verify that the equation \eref{KE3} has a unique solution which is given by $u(t,x)=\EE\varphi(\bar{X}^{x}_t)$ for all $t\ge 0,$ and that, for all $T>0$,
\begin{align}\label{F5.14}\begin{split}
&\sup_{0\leq t\leq T, x\in\RR^n}\sum^{4}_{i=1}\|\partial^i_{x}u(t,x)\|\leq C_{T},\\
&\sup_{0\leq t\leq T}\sum^{2}_{i=1}\|\partial_t(\partial^i_x u(t,x))\|\leq C_T(1+|x|).
\end{split}
\end{align}

Now, we state the last main result in this paper.

\begin{theorem} \label{main result 6}
Suppose that Assumptions {\rm\ref{B1}}, {\rm\ref{A3}} and {\rm\ref{A5}} hold. Let $\{(X_t^\varepsilon, Y_t^\varepsilon) \}_{t\ge0}$ and $\{\bar X_t\}_{t\ge0}$ be the solutions to the SDEs \eqref{Equation} and \eqref{AVE32} respectively. Then, for $T>0$ and $\varphi\in C^4_b(\RR^n)$, there is a constant $C_{\varphi,T}>0$ such that for any $(x,y)\in\RR^n\times\RR^m$ and $\varepsilon>0$,
\begin{align}
\sup_{0\leq t\leq T}|\EE \varphi(X^{\varepsilon}_t)-\EE \varphi(\bar{X}_t)|\leq& C_{\varphi,T}(1+|x|^2+|y|^2)\varepsilon\Big[\sup_{t\in [0,T]}\Lambda_\gamma(t/\varepsilon)+\varepsilon^{-2/3}\Big].\label{AVWO3}
\end{align}
\end{theorem}

\begin{proof} Let $\{\bar X_t^\varepsilon\}_{t\ge0}$ be the solution to the SDE \eqref{AVE3}.
Fix $t\le T$, and denote by $\tilde{u}^{t}(s,x)=u(t-s,x)$ for all $s\le t$. Following the proof of Theorem \ref{main result 4}, we find that
\begin{equation*}
\begin{split}
&\sup_{t\in [0,T]}|\EE\varphi(\bar X^{\varepsilon}_{t})-\EE\varphi(\bar{X}_{t})|\\
&\leq \sup_{t\in [0,T]}\Big|\EE  \int^t_0   \langle \bar b(s/\varepsilon,\bar X^{\varepsilon}_s)  -  \bar{b}_p(\bar X^{\varepsilon}_s),\partial_x \tilde{u}^{t}(s,\bar X^{\varepsilon}_s )\rangle\,ds\Big| \\
		&\quad +\frac{1}{2}\sup_{t\in [0,T]}\Big|\EE  \int^t_0   \text{Tr}\big[((\overline{\sigma\sigma^{\ast}})(s/\varepsilon,\bar X^{\varepsilon}_s)  -  (\overline{\sigma\sigma^{\ast}})_p(\bar X^{\varepsilon}_s))\partial_x^{2}\tilde{u}^{t}(s, \bar X^{\varepsilon}_s)\big]  \,ds\Big|\\
&\leq \sup_{t\in [0,T]}\Big|\EE  \int^t_0   \left[\langle\bar b(s/\varepsilon,\bar X^{\varepsilon}_{s})  -  \bar{b}_p(\bar X^{\varepsilon}_s),\partial_x \tilde{u}^{t}(s,\bar X^{\varepsilon}_s )\rangle\right.\\
&\quad\quad\quad\quad\quad\quad\quad-\left.\langle \bar b(s/\varepsilon,\bar X^{\varepsilon}_{s(\delta)})  -  \bar{b}_p(\bar X^{\varepsilon}_{s(\delta)}),\partial_x \tilde{u}^{t}({s(\delta)},\bar X^{\varepsilon}_{s(\delta)} )\rangle\right]\,ds\Big| \\
		&\quad+\frac{1}{2}\sup_{t\in [0,T]}\Big|\EE  \int^t_0   \left\{\text{Tr}\big[ ( (\overline{\sigma \sigma^{\ast}} )(s/\varepsilon,\bar X^{\varepsilon}_s)  -  (\overline{\sigma  \sigma ^{\ast}} )_p(\bar X^{\varepsilon}_s) )\partial_x^{2}\tilde{u}^{t}(s, \bar X^{\varepsilon}_s)\big]\right. \\
&\quad\quad\quad \quad\quad\quad\quad\quad\quad-\left.\text{Tr}\big[((\overline{\sigma\sigma^{\ast}})(s/\varepsilon,\bar X^{\varepsilon}_{s(\delta)})  -  (\overline{\sigma \sigma}^{\ast})_p(\bar X^{\varepsilon}_{s(\delta)}))\partial_x^{2}\tilde{u}^{t}(s(\delta), \bar X^{\varepsilon}_{s(\delta)})\big]\right\} \,ds\Big|\\
&\quad+\sup_{t\in [0,T]}\Big|\EE  \int^t_0    \langle \bar b(s/\varepsilon,\bar X^{\varepsilon}_{s(\delta)})  -  \bar{b}_p(\bar X^{\varepsilon}_{s(\delta)}),\partial_x \tilde{u}^{t}({s(\delta)},\bar X^{\varepsilon}_{s(\delta)} ) \rangle\,ds\Big|\\
&\quad+\frac{1}{2}\sup_{t\in [0,T]}\Big|\EE  \int^t_0   \text{Tr}\big[ ( (\overline{\sigma\sigma^{\ast}})(s/\varepsilon,\bar X^{\varepsilon}_{s(\delta)})  -  (\overline{\sigma\sigma^{\ast}})_p(\bar X^{\varepsilon}_{s(\delta)}) \partial_x^{2}\tilde{u}^{t}({s(\delta)}, \bar X^{\varepsilon}_{s(\delta)})\big]  \,ds\Big|=:\sum^4_{i=1}\tilde{J}_i(T).
	\end{split}\end{equation*}
By \eqref{FourthDbarb}, \eref{FourthDbarsigma3} and \eref{F5.14},
\begin{align*}
\sup_{0\leq s\leq t\leq T}\EE |\bar X^{\varepsilon}_t-\bar X^{\varepsilon}_s |^2\leq C_T(1+|x|^2)|t-s|,
\end{align*}
which combines with \eref{X} and \eref{F5.14} imply that
$$
\tilde{J}_1(T)+\tilde{J}_2(T)\leq C_T(1+|x|^2)\delta^{1/2}.$$

Using Proposition \ref{P5.1}, \eqref{F5.14} and \eqref{X}, we have
\begin{align*}
\tilde{J}_3(T)\leq& C_T\EE\left|\sum^{[T/\delta]-1}_{k=0}\int^{(k+1)\delta}_{k\delta}  \langle \bar{b}(s/\varepsilon,\bar{X}^{\varepsilon}_{k\delta})-\bar{b}_p(\bar{X}^{\varepsilon}_{k\delta}),\partial_x \tilde{u}^{t}(k\delta,\bar X^{\varepsilon}_{k\delta}) \rangle\,ds\right|\\
&\quad +C_T\EE\left|\int^{T}_{[T/\delta]\delta}  \langle \bar{b}(s/\varepsilon,\bar{X}^{\varepsilon}_{s(\delta)}) -\bar{b}_p(\bar{X}^{\varepsilon}_{s(\delta)}),\partial_x \tilde{u}^{t}(s(\delta),\bar X^{\varepsilon}_{s(\delta)}) )\rangle\,ds\right|\\
\leq& C_T\EE\left[\sum^{[T/\delta]-1}_{k=0}  \delta\left|\frac{\varepsilon}{\delta}\int^{\frac{k\delta}{\varepsilon}
+\frac{\delta}{\varepsilon}}_{\frac{k\delta}{\varepsilon}}    \bar{b}(s,\bar{X}^{\varepsilon}_{k\delta})\,ds
-  \bar{b}_p(\bar{X}^{\varepsilon}_{k\delta})\right||\partial_x \tilde{u}^{t}(k\delta,\bar X^{\varepsilon}_{k\delta})|\right]  +  C_T(1  +  |x|)\delta\\
\leq& C_T\varepsilon\EE\left|\sum^{[T/\delta]-1}_{k=0}(1+|\bar{X}^\varepsilon_{k\delta}|)\right|+C_T(1+|x|^2)\delta\\
\leq& C_T(1+|x|)\frac{\varepsilon}{\delta}+C_T(1+|x|^2)\delta.
\end{align*}
Similarly, it holds that
\begin{align}
\tilde{J}_4(T)
\leq C_T(1+|x|)\frac{\varepsilon}{\delta}+C_T(1+|x|^2)\delta.\label{tildeJ4}
\end{align}

Putting all the estimates together and taking $\delta=\varepsilon^{2/3}$ yield that
\begin{equation*}
\sup_{t\in [0,T]}\left|\EE\varphi(\bar X^{\varepsilon}_{t})-\EE\varphi(\bar{X}_{t})\right|\leq C_T(1+|x|^2)\varepsilon^{1/3}.
\end{equation*}
Therefore, the desired assertion follows from the inequality above and \eref{AVWO}.
\end{proof}

\section{Examples}
In this section, we give two concrete examples to illustrate main results of the paper. In particular, the first example demonstrates that the convergence rate given in Theorem \ref{main result 1} is optimal. For simplicity, we only consider one-dimensional setting.

\begin{example}\label{Example 1}
Let $\alpha: [0,\infty)\to (0,\infty)$ satisfy \eref{e:remarkadd}. Consider the following SDE
\begin{equation*}\left\{\begin{array}{l}
\displaystyle
dX^{\varepsilon}_t=Y^{\varepsilon}_t\, dt+dW^1_t,\quad X^{\varepsilon}_0=x\in\mathbb{R},  \\
\displaystyle dY^{\varepsilon}_t=-\varepsilon^{-1}\alpha(t/\varepsilon)Y^{\varepsilon}_t\,dt+[\varepsilon^{-1}\alpha(t/\varepsilon)]^{1/2}\,dW^2_t,\quad Y^{\varepsilon}_0=y\in\mathbb{R},
\end{array}\right.
\end{equation*} where $W^{1}:=\{W^1_t\}_{t\ge0}$ and $W^{2}:=\{W^2_t\}_{t\ge0}$ be two independent one-dimensional Brownian motions.
It is easy to see that the solution to the SDE above is given by
\begin{equation*}\left\{\begin{array}{l}
\displaystyle
X^{\varepsilon}_t=x+\int^t_0 Y^{\varepsilon}_s\, ds+W^1_t, \\
\displaystyle Y^{\varepsilon}_t=e^{-\varepsilon^{-1}\int^t_0 \alpha(r/\varepsilon)\,dr}y+\varepsilon^{-1/2}\int^t_0 e^{-\varepsilon^{-1}\int^t_s\alpha(v/\varepsilon)\,dv}\alpha^{1/2}(s/\varepsilon)\,dW^2_s.
\end{array}\right.
\end{equation*}

On the other hand, the associated frozen equation
\begin{equation}
dY_t=-\alpha(t)Y_t\,dt+\alpha^{1/2}(t)\,dW^2_t,\quad Y_s=y\label{Ex1Fro}
\end{equation}
has a unique solution
$$Y^{s,y}_t=e^{-\int^t_s \alpha(u)\,du}y+\int^{t}_s e^{-\int^t_r\alpha(u)\,du}\alpha^{1/2}(r)\,dW^2_r,$$
whose distribution is $N\Big(e^{-\int^t_s \alpha(u)\,du}\,y, \frac{1}{2}\big(1-e^{-2\int^t_s \alpha(r)\, dr}\big)\Big)$.
In particular, letting $s\to -\infty$, we get that
$$
\mu_t:=N\left(0, 1/2\right),\quad t\in \RR
$$
is an evolution family of measures for the SDE \eref{Ex1Fro}.
Hence,  the corresponding averaged equation \eqref{AVE} is given by
$$d\bar{X}_t=dW^1_t,\quad \bar{X}_0=x,$$ which is independent of $\varepsilon$.

As a result, we have
\begin{align*}
\EE  |X^{\varepsilon}_t-\bar{X}_t  |^2= &\EE\left|\int^t_0 Y^{\varepsilon}_s\,ds\right|^2=\EE\left|\int^{t}_0 Y^{0,y}_{s/\varepsilon}\,ds\right|^2\\
= &\varepsilon^2\EE\left|\int^{t/\varepsilon}_0 Y^{0,y}_{s}\,ds\right|^2= 2\varepsilon^2\int^{t/\varepsilon}_0 \int^{t/\varepsilon}_r \EE (F(s)F(r))\, ds \,dr,
\end{align*}
where $$F(s):=e^{-\int^s_0\alpha(u)\,du}y+ \int^{s}_0 e^{-\int^s_u\alpha(v)\,dv}\alpha^{1/2}(u)\,dW^2_u.$$
Furthermore, for $s\geq r$,
\begin{align*}
\EE (F(s)F(r))= &e^{-\int^s_0\alpha(u)\,du}e^{-\int^r_0\alpha(u)\,du}y^2  +  \int^r_0     e^{-\int^s_u\alpha(v)\,dv}e^{-\int^r_u\alpha(v)\,dv}\alpha(u)\,du\\
= &\Big(y^2- \frac{1}{2}\Big)e^{-\int^s_0\alpha(u)\,du}e^{-\int^r_0\alpha(u)\,du}+\frac{1}{2}e^{-\int^s_r\alpha(u)\,du}.
\end{align*}
Therefore,
\begin{align*}
\EE  |X^{\varepsilon}_t-\bar{X}_t |^2
= &(2y^2-1)\varepsilon^2\int^{t/\varepsilon}_0     \int^{t/\varepsilon}_r     e^{-\int^s_0\alpha(u)\,du}e^{-\int^r_0\alpha(u)\,du}\,ds \,dr  +  \varepsilon^2\int^{t/\varepsilon}_0  \int^{t/\varepsilon}_r     e^{-\int^s_r\alpha(u)\,du}\,ds \,dr\\
\asymp  &\varepsilon^2\int^{t/\varepsilon}_0 \Lambda(s)\,ds \asymp \varepsilon^2\int^{t/\varepsilon}_0 \alpha(s)\Lambda^2(s)\,ds,\quad \varepsilon\rightarrow 0,
\end{align*}
where $\Lambda(s)=\int^{+\infty}_s e^{-\int^t_s\alpha(u)\,du}\,dt$. Here, $f(\varepsilon)\asymp g(\varepsilon)$ means there are positive constants $c_1\le c_2$ such that $c_1\le f(\varepsilon)/g(\varepsilon)\le c_2$ for all $\varepsilon\in (0,1]$.

Take $\alpha(t)=c_0(1+t)^{\beta}$ with $c_0>0$ and $\beta\in (-1,\infty)$. According to the argument in Remark \ref{remark3.7}, we also have
\begin{equation*}
\sup_{t\in [0, T]}\EE |X_{t}^{\varepsilon}-\bar{X}^{\varepsilon}_{t}|\asymp \left\{\begin{array}{l}
\displaystyle \varepsilon^{{1+\beta}},\quad\quad\quad -1<\beta<1, \\
\displaystyle \varepsilon^2\log{1/\varepsilon},\quad \beta=1,\\
\displaystyle \varepsilon^2,\quad\quad\quad \quad\beta>1.
\end{array}\right.
\end{equation*}
\end{example}

\begin{example}\label{Example 2}
Let $\phi:[0,\infty)\to (0,\infty)$ be a bounded function. Consider the following SDE
\begin{equation*}\left\{\begin{array}{l}
\displaystyle
dX^{\varepsilon}_t=Y^{\varepsilon}_t \,dt+dW^1_t,\quad X^{\varepsilon}_0=x\in\mathbb{R}, \\
\displaystyle dY^{\varepsilon}_t=\varepsilon^{-1}\left[\phi(t/\varepsilon)-Y^{\varepsilon}_t\right]\,dt+\varepsilon^{1/2}\,dW^2_t,\quad Y^{\varepsilon}_0=y\in\mathbb{R},
\end{array}\right.
\end{equation*}where $W^{1}:=\{W^1_t\}_{t\ge0}$ and $W^{2}:=\{W^2_t\}_{t\ge0}$ are two independent one-dimensional Brownian motions. The SDE above
has the solution
\begin{equation*}\left\{\begin{array}{l}
\displaystyle
X^{\varepsilon}_t=x+\int^t_0 Y^{\varepsilon}_s\, ds+W^1_t, \\
\displaystyle Y^{\varepsilon}_t=e^{-\varepsilon^{-1}t}y+\varepsilon^{-1}\int^t_0e^{-\varepsilon^{-1}(t-r)}\phi(r/\varepsilon)\,dr+\varepsilon^{-1/2}\int^t_0 e^{-\varepsilon^{-1}(t-r)}\,dW^2_r.\\
\end{array}\right.
\end{equation*}

On the other hand, the associated frozen equation
\begin{equation}
dY_t=[\phi(t)-Y_t]\,dt+dW^2_t,\quad Y_s=y\label{Ex2Fro}
\end{equation}
admits a unique solution
$$Y^{s,y}_t=e^{-(t-s)}y+\int^t_se^{-(t-r)}\phi(r)\,dr+ \int^{t}_s e^{-(t-r)}\,dW^2_r,\quad t\ge s,$$
and the associated distribution is $$N\Big(e^{-(t-s)}y+\int^t_s e^{-(t-r)}\phi(r)\,dr, \frac{1}{2}(1-e^{-2(t-s)})\Big).$$
Letting $s\to -\infty$, we get
$$
\mu_t=N\left(\int^t_{-\infty}    e^{-(t-r)}\phi(r)\,dr, \frac{1}{2}\right),\quad t\in \RR
$$
is an evolution system of measures for the SDE \eref{Ex2Fro}.
Hence, the corresponding averaged equations \eqref{AVE} and \eqref{AVE3} are same,  and both are given by
\begin{align}
d\bar{X}^{\varepsilon}_t=\psi(t/\varepsilon)\,dt+dW^1_t,\quad \bar{X}^{\varepsilon}_0=x,\label{Ex2AVE1}
\end{align}
where $\psi(t):=\int^t_{-\infty}e^{-(t-r)}\phi(r)\,dr$.
One has
\begin{align*}
\EE |X^{\varepsilon}_t-\bar{X}^{\varepsilon}_t |^2= &\EE \left|\int^t_0 (Y^{\varepsilon}_s-\psi(s/\varepsilon))\, ds\right |^2= \EE \left|\int^{t}_0 (Y^{0,y}_{s/\varepsilon}-\psi(s/\varepsilon) )\,ds \right|^2\\
=& \varepsilon^2\EE\left|\int^{t/\varepsilon}_0 (Y^{0,y}_{s}-\psi(s)) \,ds\right|^2
= 2\varepsilon^2\int^{t/\varepsilon}_0     \int^{t/\varepsilon}_r \EE (\tilde F(s)\tilde F(r)) \,ds\, dr,
\end{align*}
where
$$\tilde F(s):=ye^{-s}-\int^0_{-\infty}h(u)e^{-(s-u)}\,du+ \int^{s}_0 e^{-(s-u)}\,dW^2_u.$$
Furthermore, for any $s\geq r$,
\begin{align*}
\EE \tilde F(s)\tilde F(r)= &\left[ye^{-s}  -  \int^0_{-\infty}    h(u)e^{-(s-u)}\,du\right]\left[y e^{-r}  -  \int^0_{-\infty}h(u)    e^{-(r-u)}\,du\right] +\int^r_0 e^{-(s-u)}e^{-(r-u)}\,du\\
= &\left((y-c)^2-\frac{1}{2}\right)e^{-(s+r)}+\frac{1}{2}e^{-s+r},
\end{align*}
where $c=\int^0_{-\infty}e^{r}\phi(r)\,dr$. Therefore, we obtain
\begin{align*}
\EE |X^{\varepsilon}_t-\bar{X}^{\varepsilon}_t |^2
= 2\varepsilon^2\int^{t/\varepsilon}_0     \int^{t/\varepsilon}_r    \left[\left((y-c)^2- \frac{1}{2}\right)e^{-(s+r)}  +  \frac{1}{2}e^{-s+r}\right]\,ds \,dr=O(\varepsilon).
\end{align*}

On the other hand,
\begin{align*}
|\EE X^{\varepsilon}_t-\EE\bar{X}^{\varepsilon}_t|= & \left|\int^t_0 \left(\EE Y^{\varepsilon}_s -\psi(s/\varepsilon)\right)\, ds\right |
=
\left|\int^t_0 (\EE Y^{0,y}_{s/\varepsilon} -\psi(s/\varepsilon))\, ds\right |\\
=&\varepsilon\Big|\int^{t/\varepsilon}_0     (\EE Y^{0,y}_{s}  -  \psi(s))\, ds\Big|
=\varepsilon\Big|\int^{t/\varepsilon}_0     \left(ye^{-s}  -  \int^{0}_{-\infty}e^{-(s-r)}    \phi(r)\,dr\right)\, ds\Big|= O(\varepsilon).
\end{align*}

We observe that both Assumptions \ref{B1} and \ref{A3} hold with $\alpha(t)\equiv 1$ and $\Lambda(t)\equiv 1$ in this case. Theorems \ref{main result 1} and \ref{main result 2} imply that
\begin{align*}
\sup_{t\in [0,T]}\EE |X^{\varepsilon}_t-\bar{X}^{\varepsilon}_t |^2+\sup_{t\in [0,T]}|\EE X^{\varepsilon}_t-\EE\bar{X}^{\varepsilon}_t|
\leq C_{T,x,y} \varepsilon,
\end{align*}
which coincides with the assertion above. Hence, the assertion above indicates the effectiveness of averaged equations \eqref{AVE} and \eqref{AVE3}.

\vspace{0.2cm}

In the spirit of Sections 4 and 5, we can expect the averaged SDE \eref{Ex2AVE1} to be independent of $\varepsilon$, if some additional assumptions imposed on $\phi$.

\textbf{\emph{Convergent coefficient case}}: Assume the following condition hold:
\begin{align*}
&\lim_{t\rightarrow +\infty}\phi(t)=0.
\end{align*}
In this situation, it is easy to check that both of the associated averaged equations \eqref{AVE21} and \eqref{AVE22} are given by
$$d\bar{X}_t=dW^1_t,\quad \bar{X}_0=x.$$
Thus, we can obtain
\begin{align*}
\EE |X^{\varepsilon}_t-\bar{X}_t |^2
= &\varepsilon^2\EE\Big|\int^{t/\varepsilon}_0 Y^{0,y}_{s}\,ds\Big|^2
= \varepsilon^2\EE\Big|\int^{t/\varepsilon}_0 (Y^{0,y}_{s}-\psi(s)+\psi(s)) \,ds\Big|^2\\
= &2\varepsilon^2\int^{t/\varepsilon}_0     \int^{t/\varepsilon}_r \EE \left[(\tilde F(s)+\psi(s))(\tilde F(r)+\psi(r))\right] \,ds \,dr\\
= &O(\varepsilon)+C\varepsilon^2\Big(\int^{t/\varepsilon}_0\psi(s)\,ds\Big)^2\rightarrow 0, \quad \text{as}\quad\varepsilon\rightarrow 0.
\end{align*}
On the other hand, it also holds that
\begin{align*}
|\EE X^{\varepsilon}_t-\EE \bar{X}_t|=& \left|\int^t_0 \EE Y^{\varepsilon}_s \, ds\right |= \varepsilon\Big|\int^{t/\varepsilon}_0 \EE Y^{0,y}_{s}\, ds\Big |\\
=~  & \varepsilon\left|\int^{t/\varepsilon}_0 (\EE Y^{0,y}_{s}-\psi(s)+\psi(s))\, ds\right |\\
=&\varepsilon\Big|\int^{t/\varepsilon}_0 \left(e^{-s}y+\int^{s}_{0}e^{-(s-r)}\phi(r)\,dr\right)\, ds\Big|\\
= &O(\varepsilon)+C\varepsilon\int^{t/\varepsilon}_0\psi(s)\,ds\rightarrow 0, \quad \text{as}\quad\varepsilon\rightarrow 0.
\end{align*}
Hence, the assertion above indicates the effectiveness of the averaged equations \eqref{AVE21} and \eqref{AVE22}.

\textbf{\emph{Periodic coefficient case}}: Assume $\phi$ is $\tau$-periodic, that is, $\phi(t+\tau)=\phi(t)$ for all $t\in\RR.$
In this situation, it is easy to check that the associated averaged equations \eqref{AVE31} and \eqref{AVE32} are given by
$$d\bar{X}_t=b\,dt+dW^1_t,\quad \bar{X}_0=x,$$
where $b=\frac{1}{\tau}\int^{\tau}_0\psi(t)\,dt$. Then, we obtain
\begin{align*}
\EE |X^{\varepsilon}_t  -  \bar{X}_t |^2
= &\varepsilon^2\EE\Big|\int^{t/\varepsilon}_0     (Y^{0,y}_{s}  -  b)\,ds\Big|^2
= \varepsilon^2\EE\Big|\int^{t/\varepsilon}_0 (Y^{0,y}_{s}  -  \psi(s)  +  \psi(s)  -  b) \,ds\Big|^2\\
= &O(\varepsilon)+C_T\Big|\frac{\varepsilon}{t}\int^{t/\varepsilon}_0\psi(s)\,ds-b\Big|^2\rightarrow 0, \quad \text{as}\quad\varepsilon\rightarrow 0,
\end{align*}
where we used the fact that $\psi$ is $\tau$-periodic.
Similarly, we have
\begin{align*}
|\EE X^{\varepsilon}_t-\EE \bar{X}_t|=&\Big|\int^t_0 (\EE Y^{\varepsilon}_s-b) \, ds\Big|= \varepsilon\Big|\int^{t/\varepsilon}_0  (\EE Y^{0,y}_{s}-b )\, ds\Big|\\
=& \varepsilon\Big|\int^{t/\varepsilon}_0 (\EE Y^{0,y}_{s}-\psi(s)+\psi(s)-b)\, ds\Big|\\
= &O(\varepsilon)+C\Big|\frac{\varepsilon}{t}\int^{t/\varepsilon}_0\psi(s)\,ds-b\Big|\rightarrow 0, \quad \text{as}\quad\varepsilon\rightarrow 0.
\end{align*}
Hence, the assertion above indicates the effectiveness of the averaged equations \eqref{AVE31} and \eqref{AVE32}.
\end{example}

\vspace{0.3cm}
\textbf{Acknowledgment}. The research of Xiaobin Sun is supported by the NSF of China (Nos.
12271219, 12090010 and 12090011).
The research of Jian Wang is supported by the National Key R\&D Program of China (2022YFA1006003) and
the NSF of China (Nos.\ 12071076 and 12225104). The research of Yingchao Xie is supported by the NSF of China  (Nos.\ 11931004 and 12471139) and the Priority Academic Program Development of Jiangsu Higher Education Institutions.

\end{document}